\documentclass{amsart}
\usepackage[colorlinks,citecolor=blue,urlcolor=black,linkcolor=black]{hyperref}
\usepackage{cmap, amssymb, bbm}

\newtheorem{thm}{Theorem}[section]
\newtheorem*{THMA}{Theorem A}
\newtheorem*{THMB}{Theorem B}
\newtheorem*{THMC}{Theorem C}
\newtheorem*{THMD}{Theorem D}
\newtheorem{lemma}[thm]{Lemma}
\newtheorem{cor}[thm]{Corollary}
\newtheorem{claim}{Claim}[thm]
\newtheorem{prop}[thm]{Proposition}
\newtheorem{fact}[thm]{Fact}
\newtheorem{example}[thm]{Example}
\newtheorem*{FACT}{Fact}

\theoremstyle{definition}
\newtheorem{defn}[thm]{Definition}

\theoremstyle{remark}
\newtheorem{results}[thm]{Results}
\newtheorem{remark}[thm]{Remark}
\newtheorem*{remarks}{Remarks}
\newtheorem{Q}[thm]{Question}

\renewcommand\mid{\mathrel{|}\allowbreak}
\renewcommand{\restriction}{\mathbin\upharpoonright}

\newcommand\bigdiagonal{\bigtriangleup}
\newcommand{\one}{\mathbbm{1}}
\newcommand\forces{\Vdash}
\newcommand{\etalchar}[1]{$^{#1}$}
\newcommand{\sqleft}[1]{\mathrel{_{#1}{\sqsubseteq}}}
\newcommand\sqx{\sqleft{\chi}}
\newcommand\sq{\sqsubseteq}
\newcommand\s{\subseteq}
\newcommand\br{\blacktriangleright}
\newcommand\id{\mathrm{id}}

\newcommand\axiomfont[1]{\textsf{\textup{#1}}}
\newcommand\zfc{\axiomfont{ZFC}}

\newcommand\ch{\textup{CH}}
\newcommand\gch{\textup{GCH}}
\newcommand\frp{\textup{FRP}}
\newcommand\ns{\textup{NS}}

\DeclareMathOperator{\p}{P}
\DeclareMathOperator{\suc}{succ}
\DeclareMathOperator{\reg}{Reg}
\DeclareMathOperator{\card}{Card}
\DeclareMathOperator{\cof}{cof}
\DeclareMathOperator{\cf}{cf}
\DeclareMathOperator{\cl}{cl}
\DeclareMathOperator{\refl}{Refl}
\DeclareMathOperator{\coll}{Coll}
\DeclareMathOperator{\dom}{dom}
\DeclareMathOperator{\rng}{Im}
\DeclareMathOperator{\otp}{otp}
\DeclareMathOperator{\acc}{acc}
\DeclareMathOperator{\nacc}{nacc}
\DeclareMathOperator{\col}{Col}
\DeclareMathOperator{\chr}{Chr}

\author{Chris Lambie-Hanson}
\address{Department of Mathematics, Bar-Ilan University, Ramat-Gan 5290002, Israel.}
\urladdr{http://u.math.biu.ac.il/~lambiec/}
\author{Assaf Rinot}
\address{Department of Mathematics, Bar-Ilan University, Ramat-Gan 5290002, Israel.}
\urladdr{http://www.assafrinot.com}


\title[Reflection on coloring and chromatic]{Reflection on the coloring and chromatic numbers}
\begin{document}
\begin{abstract}
  We prove that reflection of the coloring number of graphs is consistent with non-reflection of the chromatic number.
  Moreover, it is proved that incompactness for the chromatic number of graphs (with arbitrarily large gaps) is compatible
  with each of the following compactness principles: Rado's conjecture, Fodor-type reflection, $\Delta$-reflection, Stationary-sets reflection, Martin's Maximum,
  and a generalized Chang's conjecture. This is accomplished by showing that, under $\gch$-type assumptions,
  instances of incompactness for the chromatic number can be derived from square-like
  principles that are compatible with large amounts of compactness.

  In addition, we prove that, in contrast to the chromatic number, the coloring number does not admit arbitrarily
  large incompactness gaps.
\end{abstract}
\maketitle

\section{Introduction}

\begin{defn}\label{def11}
  A \emph{graph} is a pair $\mathcal G = (G,E)$, where $E \subseteq [G]^2$. Elements of
  $G$ are called the \emph{vertices} of $\mathcal G$, and elements of $E$ are called the
  \emph{edges} of $\mathcal G$. If $x \in G$, then
  the \emph{neighborhood} of $x$ in $\mathcal G$ is $N_{\mathcal G}(x):=\{y \in G \mid \{x,y\} \in E\}$;
  if $\lhd$ is an ordering of $G$, then $N_{\mathcal G}^\lhd(x):=\{y\in N_{\mathcal G}(x)\mid y\lhd x\}$.
\end{defn}
For an arbitrary graph $\mathcal G$, the set of
  vertices of $\mathcal G$ will often be denoted by $V(\mathcal G)$, and the set of edges by $E(\mathcal G)$.

\begin{defn}
  Suppose $\mathcal G$ is a graph.
  \begin{enumerate}
    \item A function $c$ on $V(\mathcal G)$ is called a
      \emph{chromatic coloring} of $\mathcal G$ if  $c(x) \neq c(y)$ for all $\{x,y\} \in E(\mathcal G)$.
      The \emph{chromatic number} of $\mathcal G$, denoted $\chr(\mathcal G)$, is the least cardinal $\chi$ for
      which there exists a chromatic coloring $c:V(\mathcal G)\rightarrow \chi$.
    \item The \emph{coloring number} of $\mathcal G$, denoted $\col(\mathcal G)$, is the least cardinal $\kappa$ for which
      there exists a well-ordering $\vartriangleleft$ of $V(\mathcal G)$ such that $|N_{\mathcal G}^\lhd(x)|<\kappa$ for all $x \in V(\mathcal G)$.
  \end{enumerate}
\end{defn}

It is evident that $\chr(\mathcal G) \leq \col(\mathcal G)$ for every graph $\mathcal G$.

\medskip

By a classic result of de Bruijn and Erd\H{o}s \cite{MR0046630}, if $\mathcal G$ is a graph,  $k$ is a positive integer,
and all finite subgraphs of $\mathcal G$ have chromatic number $\le k$, then $\chr(\mathcal G)\le k$.
Questions involving generalizations of this theorem (to infinite cardinal numbers,
as well as to other cardinal functions) have attracted a lot of attention;
we highlight a number of known results regarding compactness for chromatic and coloring numbers in Section~\ref{section2}.

Counterexamples to compactness are captured by the following concepts:

\begin{defn}
  Suppose $\mathcal G$ is a graph and $\mu \leq \kappa$ are cardinals. $\mathcal G$
  is said to be \emph{$(\mu, \kappa)$-chromatic} (resp. \emph{$(\mu, \kappa)$-coloring})
  if $\chr(\mathcal G) = \kappa$ (resp. $\col(\mathcal G) = \kappa$) and $\chr(\mathcal G') \leq \mu$ (resp. $\col(\mathcal G') \leq \mu$) for every subgraph
  $\mathcal G'$ of $\mathcal G$ with $|V(\mathcal G')| < |V(\mathcal G)|$.
\end{defn}

In \cite{paper12}, the second author introduces a family of graphs, denoted $G(\vec C)$,
and investigates their features.
It is established there that if $\vec C$ is a \emph{coherent} sequence of local clubs along a regular
cardinal $\kappa$ and $G$ is a subset of $\kappa$ all of whose proper initial segments are non-stationary,
then $G(\vec C)$ is $(\aleph_0,\theta)$-chromatic for some cardinal $\theta\le\kappa$.
In addition, in \cite{paper12}, various constructions are given of coherent sequences $\vec C$ and non-reflecting stationary sets $G$
for which $\theta$ --- that is, $\chr(G(\vec C))$ --- is arbitrarily large.

In this paper, it is proved that if $\vec C$ is coherent, then $G(\vec C)$ is $(\aleph_0,\theta)$-chromatic even if $G=\kappa$.
This eliminates the need for the existence of non-reflecting stationary sets,
thereby opening the door for compatibility of the incompactness for the chromatic number with compactness for the coloring number.\footnote{The existence of a non-reflecting stationary set
implies incompactness for the coloring number; see Lemma~\ref{nonreflecting_set_col}.}

Furthermore, it is shown here that weaker forms of coherence of $\vec C$ suffice to infer that $G(\vec C)$ is $(\chi,\theta)$-chromatic, even when $\theta\gg\chi$.
This allows the compatibility of the incompactness for the chromatic number with very large cardinals.

To succinctly state some of the consequences of the work in this paper,
let $\mathcal E(\chi,\kappa)$ stand for the assertion that there exists a $(\chi,\kappa)$-chromatic graph of size $\kappa$. We have:

\begin{THMA}
  Assuming the consistency of large cardinal axioms,\footnote{The strength needed differs depending on the statement;
  see Section~\ref{consistency_sect} for the precise large cardinal axioms that are used.} the following are consistent:
  \begin{enumerate}
    \item $(\aleph_{\omega+1}, \aleph_\omega) \twoheadrightarrow (\aleph_1, \aleph_0)$ together with $\mathcal E(\aleph_0, \aleph_{\omega+1})$;
    \item   $\frp(<\aleph_3)$ together with $\mathcal E(\aleph_0,\aleph_2)$;
    \item Rado's Conjecture together with  $\mathcal E(\aleph_2,\kappa)$ holding for all regular $\kappa > \aleph_2$;
    \item Martin's Maximum together with  $\mathcal E(\aleph_2,\kappa)$ holding for all regular $\kappa > \aleph_2$;
    \item $\chi$ is a supercompact cardinal together with $\mathcal E(\chi,\kappa)$ holding for all regular $\kappa > \chi$;
    \item \begin{enumerate}
        \item $\Delta_{\aleph_{\omega^2}, \aleph_{\omega^2+1}}$ together with $\mathcal E(\aleph_0, \aleph_{\omega^2+1})$;
        \item $\Delta_\kappa$ together with $\mathcal E(\aleph_0, \kappa)$, where $\kappa$ is inaccessible;
      \end{enumerate}
    \item \begin{enumerate}
        \item Reflection of stationary subsets of $E^{\aleph_2}_{\aleph_0}$ together with $\mathcal E(\aleph_0,\aleph_2)$;
        \item Reflection of stationary subsets of $\aleph_{\omega+1}$ together with $\mathcal E(\aleph_0,\aleph_{\omega+1})$;
        \item Reflection of stationary subsets of $\kappa$ together with $\mathcal E(\aleph_0,\kappa)$, where $\kappa$ is the least inaccessible cardinal.
      \end{enumerate}
  \end{enumerate}
\end{THMA}

\begin{proof}
  The proofs of all of the statements rely on Corollary \ref{incompactness_cor_1}(2).
  (1) then follows from Corollary \ref{chang_cor},
  (2) from Corollary \ref{frp_cor},
  and (3) and (4) from   Corollary \ref{aleph_2_cor}.
  (5)  follows from Corollary~\ref{supercompact_cor},
  (6)(a) from Corollary \ref{fontanella_hayut_cor}, (6)(b)
  from Theorem \ref{inaccessible_thm_1}, and
  (7) from  Theorem \ref{refl_thm} and \cite[$\S3.3$]{hayut_lh}.
\end{proof}

To put Theorem A into context, let us point out a few relationships between the above principles and reflection of cardinal functions.

\begin{FACT}
\begin{enumerate}

\item If $(\aleph_{\omega+1}, \aleph_\omega) \twoheadrightarrow (\aleph_1, \aleph_0)$ holds, and  $\theta<\kappa\le\aleph_{\omega+1}$  are infinite cardinals,
then every $\kappa$-sized graph, all of whose strictly smaller subgraphs have coloring number $\le\theta$, has coloring number $\le\theta^+$.
\item    $\frp(<\chi$) is equivalent to the assertion that any graph of size $<\chi$ of uncountable coloring number has an $\aleph_1$-sized subgraph of uncountable coloring number.
\item Rado's Conjecture is equivalent to the assertion that any tree whose comparability graph is uncountably chromatic has an $\aleph_1$-sized subtree whose comparability graph is uncountably chromatic.
\item If there exists an $(\aleph_0,\ge\aleph_1)$-coloring graph of size $\kappa$, then there exists a tree of size $\le\kappa^{\aleph_0}$ whose comparability graph is $(\aleph_0,\ge\aleph_1)$-chromatic. In particular, Rado's Conjecture implies $\frp$.
\item Suppose that $\theta<\chi\le\kappa$  are infinite cardinals such that $\chi$ is strongly compact. Then every graph of size $\kappa$ and chromatic number $>\theta$ has a
  subgraph of size $<\chi$ and chromatic number $>\theta$.
\item  Suppose that $\theta < \chi\le\kappa$ are infinite cardinals such that $\kappa$ is singular or $\Delta_{\chi, \kappa}$ holds.
  Then every graph of size $\kappa$ and coloring number $>\theta$ has a strictly smaller subgraph of coloring number $>\theta$.
\end{enumerate}
\end{FACT}
\begin{proof}
 (1) By Lemma~\ref{chang_coloring} below.
 (2) By Theorem 3.1 of \cite{more_frp}.
 (3) By Theorem~6 of \cite{MR686495}.
 (4) This will appear in \cite{rc_implies_frp}.
 (5) By the proof of Theorem 1 of  \cite{MR0046630}.
 (6) By \cite{shelah_singular_compactness} (See also Proposition~\ref{delta_reflection_coloring_prop} below).
 \end{proof}

On the purely combinatorial side, we prove that the combination of $\gch$ and square-like principles gives rise to incompactness graphs.
In order to state the next theorem, we shall need the following definition (for missing notions, see the Notation subsection below).

\begin{defn}[\cite{paper29}] For infinite regular cardinals $\chi<\kappa$,
the principle $\square(\kappa,{\sq_\chi})$ asserts the existence of a sequence $\vec C=\langle C_\alpha\mid\alpha<\kappa\rangle$ satisfying the following:
\begin{itemize}
\item for every limit ordinal $\alpha<\kappa$, $C_\alpha$ is a club in $\alpha$;
\item for every $\alpha<\kappa$ and $\bar\alpha\in\acc(C_\alpha)$, if $\otp(C_\alpha)\ge\chi$,  then $C_{\bar\alpha}=C_\alpha\cap\bar\alpha$;
\item for every club $D$ in $\kappa$, there exists some $\alpha\in\acc(D)$ such that $D\cap\alpha\neq C_\alpha$.
\end{itemize}
\end{defn}

The principle $\square(\kappa,{\sq_\omega})$ is commonly denoted by $\square(\kappa)$.

\begin{THMB} Suppose that $\lambda$ is an uncountable cardinal and that $\gch$ and $\square(\lambda^+)$ both hold.
\begin{enumerate}
\item If $\lambda$ is regular, then there exists an $(\aleph_0,\ge\lambda)$-chromatic graph of size $\lambda^+$;
\item If $\lambda$ is singular, then there exists an $(\aleph_0,\lambda^+)$-chromatic graph of size $\lambda^+$.
\end{enumerate}
More generally, suppose that $\aleph_0\le\cf(\chi)=\chi<\lambda$ are cardinals and that $\gch$ and $\square(\lambda^+,{\sq_\chi})$ both hold.
\begin{enumerate}
\item If $\lambda$ is regular, then there exists a $(\chi,\ge\lambda)$-chromatic graph of size $\lambda^+$;
\item If $\lambda$ is singular, then there exists a $(\chi,\lambda^+)$-chromatic graph of size $\lambda^+$.
\end{enumerate}
\end{THMB}

The preceding is an improvement in a certain direction upon results of the second author from \cite{paper12}, in which it is proved that,
for any infinite cardinal $\lambda$, $\ch_\lambda + \square_\lambda$ entails the existence of an $(\aleph_0, \mu)$-chromatic graph for all infinite $\mu\le\lambda$,
and if, additionally,  $\lambda$ is singular, then the existence of  an $(\aleph_0, \lambda^+)$-chromatic graph, as well.

In addition, it is a curious and a counterintuitive fact that the reflection of stationary sets actually helps in achieving a maximal degree of incompactness:
\begin{THMC} Suppose that $\lambda$ is an uncountable cardinal and that $\gch$ and $\square(\lambda^+)$ both hold.
  If $\refl(S)$ holds for some stationary $S\s\lambda^+$, then there exists an $(\aleph_0,\lambda^+)$-chromatic graph of size $\lambda^+$.
\end{THMC}

Finally, we apply the techniques of this paper to address a question about possible chromatic spectra of graphs,
a topic whose study was initiated in \cite{rinot17}.  It is proved:
\begin{THMD} The following statement is equiconsistent with $\zfc$. $\gch$ holds, and for every infinite cardinal $\kappa$,
there exists a graph $\mathcal G$ satisfying:
\begin{itemize}
\item $\mathcal G$ has size and chromatic number $\kappa$;
\item for every infinite cardinal $\lambda<\kappa$, there exists a cofinality-preserving, $\gch$-preserving forcing extension in which $\chr(\mathcal G)=\lambda$.
\end{itemize}
\end{THMD}

\subsection*{Organization of this paper}
Section~\ref{section2} is graph-theoretic in nature.
In Subsection~\ref{subsection21},
we first list various compactness and incompactness results for the chromatic numbers.
Then, we turn to generalize the results from \cite{paper12} concerning the $C$-sequence graph,
motivating the study of various $C$-sequences that is carried out in later sections.
In Subsection~\ref{subsection22}, we collect various compactness and incompactness results for the coloring numbers.
In addition, it is established that for every infinite cardinal $\mu$ and every graph $\mathcal G$,
if every strictly smaller subgraph $\mathcal G'$ of $\mathcal G$ satisfies $\col(\mathcal G')\le\mu$, then $\col(\mathcal G)\le\mu^{++}$.
We also provide a couple of sufficient conditions that allow one to reduce the bound $\mu^{++}$ down to $\mu^+$, which is the best one can hope for.

Section~\ref{improve_section} is set-theoretic in nature. It is dedicated to constructing $C$-sequences for which the corresponding $C$-sequence graphs
witness incompactness for the chromatic number with a very large gap.
Among other things, Subsection~\ref{subsection32} is concluded with the proofs of Theorems B and C.
In Subsection~\ref{forcing_subsection}, we analyze a notion of forcing for introducing
$C$-sequences for which the corresponding $C$-sequence graphs exhibit a maximal degree of incompactness for the chromatic number.

In Section~\ref{consistency_sect}, we combine the method of Subsection~\ref{forcing_subsection}
with various methods for producing models of compactness,
thus demonstrating that incompactness for the chromatic
number of graphs is compatible with a wide array of set-theoretic compactness principles.

In Section~\ref{section5}, we provide a proof of Theorem~D.

\subsection*{Notation}
For an infinite cardinal $\lambda$, write  $\ch_\lambda$ for the assertion that $2^\lambda=\lambda^+$.
Suppose that $C,D$ are sets of ordinals.
Write $\acc(C):=\{\alpha\in C\mid \sup (C\cap\alpha) = \alpha>0 \}$, $\nacc(C) := C \setminus \acc(C)$,
and $\acc^+(C) := \{\alpha<\sup(C)\mid \sup (C\cap\alpha) = \alpha>0 \}$.
Write $\cl(C):=C\cup\acc^+(C)$.
For any  $j < \otp(C)$, denote by $C(j)$ the unique element $\delta\in C$ for which $\otp(C\cap\delta)=j$.
For any ordinal $\sigma$, write
$\suc_\sigma(C) := \{ C(j+1)\mid j<\sigma\ \&\ j+1<\otp(C)\}$.
Write $D\sq C$ iff there exists some ordinal $\beta$ such that $D = C \cap \beta$.
Write $D\sqx C$ if either $D\sq C$ or $\cf(\sup(D))<\chi$.
Write $D \sq_\chi C$ if either
$D \sq C$ or ($\otp(C)< \chi$ and $\nacc(C)$ consists only of successor ordinals).
For an ordinal $\eta$ and an infinite, regular cardinal $\chi$, write
$E^\eta_\chi := \{\alpha < \eta\mid \cf(\alpha) = \chi\}$.

Suppose that $\kappa$ is a regular uncountable cardinal. Let $\reg(\kappa):=\{\chi\mid \aleph_0\le\cf(\chi)=\chi<\kappa\}$.
Denote by $\ns^+_\kappa$ the collection of all stationary subsets of $\kappa$;
whenever $V'$ is some class extending $V$, denote by $(\ns^+_\kappa)^V$ the collection of all stationary subsets of $\kappa$, as computed in $V$.
For $S\in\ns^+_\kappa$,
$\refl(S)$ is the assertion that every stationary subset of $S$ reflects;
$\refl^*(S)$ is the assertion that, for every $\kappa$-directed closed set-forcing $\mathbb{P}$,
$\Vdash_{\mathbb{P}} ``\refl(\check S)."$

\section{Compactness for chromatic and coloring numbers}\label{section2}
In this section, we outline a number of known graph theoretic results about compactness and incompactness
for chromatic and coloring numbers and then prove some combinatorial results that are behind Theorems A,B,C,D of the paper.
We begin by looking at chromatic numbers.

\subsection{Chromatic numbers}\label{subsection21}

Compactness and incompactness for the chromatic number of graphs have been the focus of a great deal of
work over the last half century. The following lists some of the notable results that have been achieved
through this work, providing some historical context and motivation for the questions considered in this
paper.

\begin{results}[Incompactness for the chromatic number]\hfill

\begin{itemize}
  \item (Erd\H{o}s-Hajnal, \cite{MR0263693}) If $2^{\aleph_0}=\aleph_1$, then there exists an $(\aleph_0,\aleph_1)$-chromatic graph of size $\aleph_2$.
\item (Galvin, \cite{MR0345859}) If $2^{\aleph_0}=2^{\aleph_1}<2^{\aleph_2}$, then there exists an $(\aleph_0,\aleph_2)$-chromatic graph of size $(2^{\aleph_1})^+$.
\item (Todorcevic, \cite{MR686495}) If $\kappa$ is a regular uncountable cardinal and there exists a nonreflecting stationary subset of $E^\kappa_{\omega}$,
then there exists an $(\aleph_0,\ge\aleph_1)$-chromatic graph of size $\kappa$.
\item (Baumgartner, \cite{MR736618}) It is consistent with $\gch$ that
there exists  an $(\aleph_0,\aleph_2)$-chromatic graph of size $\aleph_2$.
\item (Komj{\'a}th, \cite{MR941243}) It is consistent with $2^{\aleph_0}=\aleph_3$ that
there exists  an $(\aleph_0,\aleph_2)$-chromatic graph of size $\aleph_2$.
\item (Todorcevic, 1986 and, independently, Rinot, 2014 [both unpublished]) Martin's Axiom entails the existence of an $(\aleph_0,2^{\aleph_0})$-chromatic graph of size $2^{\aleph_0}$.
\item (Komj{\'a}th, \cite{MR941243}) It is consistent with $2^{\aleph_0}=\aleph_{\omega_1+1}$ that
there exists  an $(\aleph_0,\aleph_1)$-chromatic graph of size $\aleph_{\omega_1}$.
\item (Shelah, \cite{MR1117029}) It is consistent with $\gch$ that
there exists  an $(\aleph_0,\aleph_1)$-chromatic graph of size $\aleph_{\omega_1}$.
\item (Soukup, \cite{MR1066404}) For any cardinal $\kappa$, it is consistent that $2^{\aleph_0}\ge\kappa$
and there exists an $(\aleph_0,(2^{\aleph_0})^+)$-chromatic graph of size $(2^{\aleph_0})^+$.
\item (Shelah, \cite{MR1117029}) If $V=L$, then ($\gch$ holds, and) for every regular non-weakly compact cardinal $\kappa$,
there exists an $(\aleph_0,\kappa)$-chromatic graph of size $\kappa$.
\item (Shelah, \cite{MR3061483}) If $\mu < \kappa$ are regular cardinals, $\kappa^\mu = \kappa$, and
  there is a non-reflectioning stationary subset of $E^\kappa_\mu$, then there exists a $(\mu, \geq \mu^+)$-chromatic
  graph of size $\kappa$.
\item (Rinot, \cite{paper12}) If $\lambda$ is an infinite cardinal, $2^\lambda = \lambda^+$, and $\square_\lambda$
  holds, then there exists an $(\aleph_0, \mu)$-chromatic graph of size $\lambda^+$ for all infinite $\mu\le\lambda$. If, additionally,
  $\lambda$ is singular, then there exists an $(\aleph_0, \lambda^+)$-chromatic graph of size $\lambda^+$.
\end{itemize}
\end{results}

\begin{results}[Compactness for the chromatic number]\hfill
\begin{itemize}
  \item (de Bruijn-Erd\H{o}s, \cite{MR0046630}) If $\chi = \aleph_0$ or $\chi$ is strongly compact, $\theta < \chi$, and $\mathcal G$ is a graph
    such that every subgraph of size $<\chi$ has chromatic number at most $\theta$, then $\mathcal G$ has chromatic number at most $\theta$.
  \item (Foreman-Laver, \cite{MR925267}) Relative to a large cardinal hypothesis, it is consistent with $\gch$ that there does not
  exist an $(\aleph_0,\aleph_2)$-chromatic graph of size $\aleph_2$.
\item (Shelah, \cite{MR1117029}) Relative to a large cardinal hypothesis, it is consistent with $\gch$ that, whenever
  $1 \leq n < \omega$ and $\mathcal G$ is an $\aleph_{\omega+1}$-sized graph such that every subgraph of size $<\aleph_\omega$ has chromatic
  number at most $\aleph_n$, it follows that $\mathcal G$ has chromatic number at most $\aleph_n$.\footnote{The case $n=0$ remains open to this date.}
\item (Unger, \cite{unger_note}) Relative to a large cardinal hypothesis, it is consistent with $\gch$ that,
  whenever $1 \leq \alpha < \omega_1$ and $\mathcal G$ is an $\aleph_{\omega_1+1}$-sized graph such that every subgraph of size
  $<\aleph_{\omega_1}$ has chromatic number at most $\aleph_{\alpha+1}$, it follows that $\mathcal G$
  has chromatic number at most $\aleph_{\alpha+1}$.
  \end{itemize}
\end{results}

\begin{defn}\label{c-sequence} Let $\Gamma$ be a set of ordinals. A \emph{$C$-sequence} over $\Gamma$ is a sequence $\vec C=\langle C_\alpha\mid\alpha\in\Gamma\rangle$ such that, for all limit $\alpha\in\Gamma$,
  $C_\alpha$ is a club subset of $\alpha$. For any binary relation $\mathcal R$, the sequence $\vec C$ is said to be \emph{$\mathcal R$-coherent}, if, for all $\alpha\in\Gamma$ and
  all $\bar\alpha\in\acc(C_\alpha)$, we have $\bar\alpha\in\Gamma$ and $C_{\bar\alpha}\mathrel{\mathcal R}C_\alpha$.
  For any ordinal $\mu$, the sequence $\vec C$ is said to be \emph{$\mu$-bounded} if, for all $\alpha\in\Gamma$, we have $\otp(C_\alpha)\le\mu$.
\end{defn}

\begin{defn}[The $C$-sequence graph, \cite{paper12}]\label{c_graph_defn} To any $C$-sequence $\vec C=\langle C_\alpha\mid\alpha<\gamma\rangle$ and any subset $G\s\gamma$,
we attach a graph $G(\vec C):=(G,E)$, by letting:
\begin{itemize}
\item $E:=\{ \{\alpha,\beta\}\in[G]^2\mid \alpha\in N_\beta\}$, where for all $\beta<\gamma$:
\item $N_\beta:=\{\alpha\in C_\beta\cap G\mid \min(C_\alpha)>\sup(C_\beta\cap\alpha)\ge\min(C_\beta)\}$.
\end{itemize}
\end{defn}
\begin{remark} Note that $N^\in_{G(\vec C)}(\beta)$ of Definition~\ref{def11} coincides with $N_\beta$.
In particular, for any infinite cardinal $\mu$, if $\vec C$ is $\mu$-bounded, then $\col(G(\vec C))\le\mu^+$.
\end{remark}

\begin{remark}\label{remark26} One of the referees asked us to mention the Hajnal-M\'at\'e graphs, and to elaborate on the history of Definition~\ref{c_graph_defn}.

A \emph{Hajnal-M\'at\'e graph} \cite{MR0424569} is a graph $\mathcal G=(\omega_1,E)$ satisfying that for every $\beta<\omega_1$, $N_{\mathcal G}^\in(\beta)$ is either finite,
or a cofinal subset of $\beta$ of order-type $\omega$. So, in essence, such graphs $\mathcal G$ are derived from an $\omega$-bounded $C$-sequence over $\omega_1$.

By Theorem~8.1 of \cite{MR0424569}, $V=L$ entails the existence of a Hajnal-M\'at\'e graph which is uncountably chromatic.
Their idea is to use $\diamondsuit^+(\omega_1)$ (indeed, considerably weaker prediction principles suffice) to construct an $\omega$-bounded $C$-sequence $\langle C_\alpha\mid\alpha<\omega_1\rangle$
in such a way that for every function $c:\omega_1\rightarrow\omega$, there exists some $\beta<\omega_1$ such that  $c(\beta)\in c[C_\beta]$.

The $C$-sequence graphs are somewhat similar in the sense that they build on the same strategy for ensuring a high chromatic number for the graph.
However, the definition of the edge relation of the $C$-sequence graph is slightly more involved, as it is meant to ensure that, at the same time, smaller subgraphs will have a small chromatic number.
The definition was conceived in 2012, after Rinot noticed some similarity between the construction of \cite[\S1]{MR3061483} that just appeared in the arXiv, and Definition~1.3 of \cite{paper11} that was submitted for publication a year before.
Later on, in 2013, the $C$-sequence graphs from \cite{paper12} served as building blocks in Rinot's solution of the infinite weak Hedetniemi conjecture \cite{paper16}.
\end{remark}

Throughout this subsection, we fix infinite regular cardinals $\chi<\kappa$,
a $C$-sequence $\vec C$ over $\kappa$, and a cofinal subset $G$ of $\kappa$, satisfying the following two hypotheses:
\begin{itemize}
\item[($\aleph$)] For all $\alpha\in \kappa\setminus G$, we have $C_\alpha\cap G=\emptyset$;
\item[($\beth$)]  For all $\alpha\in G$ and $\bar\alpha\in\acc(C_\alpha)\cap\cof(\chi)$, we have $\bar\alpha\in G$ and $C_{\bar\alpha}=C_\alpha\cap\bar\alpha$.
\end{itemize}

Note that if $\vec C$ is $\sqx$-coherent, then we could have simply taken $G$ to be $\kappa$.
Now, let us study the corresponding graph $G(\vec C) = (G,E)$.

\begin{defn}  For any ordinal $\delta\le\kappa$,
we say that $c:\delta\rightarrow\chi$ is a \emph{suitable coloring}
if the following hold:
\begin{itemize}
\item $c$ is $E$-chromatic, that is, for all $\{\alpha,\beta\}\in E\cap[\delta]^2$, we have $c(\alpha)\neq c(\beta)$;
\item $|c[N_\gamma]|<\chi$ for all $\gamma<\kappa$.
\end{itemize}
\end{defn}

So, a suitable coloring is one that is easy to extend to a larger domain while keeping it chromatic. Indeed, this is the content of Lemma~\ref{lemma1} below.

\begin{defn} For all $\eta \leq \kappa$, write $G^\eta_\chi:=\{\gamma\in G\cap \eta \mid \cf(\gamma)=\chi\}$.
\end{defn}

\begin{lemma}\label{cohere}
For every $\delta<\kappa$ and every coloring $c:\delta\rightarrow\chi$, the following are equivalent:
\begin{enumerate}
\item $c$ is suitable;
\item $c$ is $E$-chromatic, and $|c[N_\gamma]|<\chi$ for all $\gamma\in G^{\delta+1}_\chi$.
\end{enumerate}
\end{lemma}
\begin{proof} Towards a contradiction, suppose that $\delta<\kappa$,
$c:\delta\rightarrow\chi$ is an $E$-chromatic coloring,  $|c[N_\gamma]|<\chi$ for all $\gamma\in G^{\delta+1}_\chi$,
and yet there exists some $\gamma^*<\kappa$ such that $|c[N_{\gamma^*}]|=\chi$.
In particular, by hypothesis $(\aleph)$, we have $\gamma^*\in G$.

Pick a subset $I\s N_{\gamma^*}\cap \delta$ of order-type $\chi$ such that
$c\restriction I$ is injective. Put $\gamma:=\sup(I)$, so that $\gamma\in(\acc(C_{\gamma^*})\cup\{\gamma^*\})\cap E^{\delta+1}_\chi$.
By hypothesis $(\beth)$, then, $\gamma\in G^{\delta+1}_\chi$ and $C_\gamma=C_{\gamma^*}\cap\gamma$, so $N_\gamma\cap I=N_{\gamma^*}\cap I$.
Finally, by $\gamma\in G^{\delta+1}_\chi$, we have  $\chi>|c[N_\gamma]|\ge |c[N_{\gamma^*}\cap I]|=\chi$.
This is a contradiction.
\end{proof}

We thank D. Soukup for pointing out the following Lemma.

\begin{lemma}\label{triangle_free_lemma}
  $G(\vec{C})$ is triangle-free.
\end{lemma}

\begin{proof}
  Suppose $\alpha < \beta < \gamma$ are in $G$ and $\{\alpha, \beta\}, \{\beta, \gamma\} \in E$.
\begin{itemize}
\item[$\circ$]  By $\{\beta, \gamma\} \in E$, we have $\min(C_\beta) > \sup(C_\gamma \cap \beta)$, so that $C_\beta\cap C_\gamma=\emptyset$.

\item[$\circ$]  By $\{\alpha, \beta\} \in E$, we have $\alpha \in C_\beta$, so that $\alpha\notin C_\gamma$.

\item[$\circ$]  By $\alpha\notin C_\gamma$ and $\alpha<\gamma$, we have  $\{\alpha, \gamma\} \notin E$.\qedhere
\end{itemize}
\end{proof}

\begin{lemma}\label{lemma1}
\begin{enumerate}
\item For all $\delta<\kappa$, the following holds: for every $x\in[\chi]^\chi$, every $\varepsilon \leq \delta$,
and every suitable coloring $c:\varepsilon\rightarrow\chi$,
there exists a suitable coloring $c':\delta\rightarrow\chi$ extending $c$ such that $c'[\delta\setminus\varepsilon]\s x$;
\item If there is a club $D$ in $\kappa$ such that $D\cap\eta=C_\eta$ for all $\eta\in\acc(D)\cap G_\chi^\kappa$, then $\chr(G(\vec C))\le\chi$.
\end{enumerate}
\end{lemma}
\begin{proof} (1) By induction on $\delta<\kappa$.

$\br$ The case $\delta=0$ is trivial. $\blacktriangleleft$

$\br$ Suppose that $\delta$ is an ordinal $<\kappa$ for which the claim holds.
Given $x\in[\chi]^\chi$ and a suitable coloring $c:\varepsilon\rightarrow\chi$ with $\varepsilon\le\delta+1$,
put $y:=c[N_\delta]$, so that $|y|<\chi$. Fix an arbitrary $\xi\in x\setminus y$.

If $\varepsilon=\delta+1$, then we are done by taking $c':=c$.
Thus, suppose that $\varepsilon\le\delta$
and appeal to the induction hypothesis with $x\setminus\{\xi\}$ and $c$ to find a suitable coloring
$c^*:\delta\rightarrow\chi$ extending $c$ with
$c^*[\delta\setminus\varepsilon]\s x\setminus\{\xi\}$. Finally, let $c':\delta+1\rightarrow\chi$ be the unique
extension of $c^*$ that satisfies $c'(\delta)=\xi$.

Evidently, $c'[(\delta+1)\setminus \varepsilon]\s x$.
We verify that $c'$ is suitable, using the criteria of Lemma~\ref{cohere}.

As $c'\restriction\delta=c^*$ and the latter is $E$-chromatic,
to show that $c'$ is $E$-chromatic it suffices to verify that, for all $\alpha\in N_\delta$, we have $c'(\alpha)\neq c'(\delta)$, i.e., $c'(\alpha)\neq\xi$.
Let $\alpha\in N_\delta$ be arbitrary. If $\alpha<\varepsilon$, then $c'(\alpha)=c^*(\alpha)=c(\alpha)\in c[N_\delta]= y$ and hence $c'(\alpha)\neq \xi$.
If $\alpha\ge\varepsilon$, then $c'(\alpha)=c^*(\alpha)\in c^*[\delta\setminus\varepsilon]\s x\setminus\{\xi\}$, and hence $c'(\alpha)\neq\xi$.

In addition, as $c'[N_\gamma]=c^*[N_\gamma]$ for all $\gamma\le\delta$ and $c^*$ is suitable,
we infer that $|c'[N_\gamma]|<\chi$ for all $\gamma\in G^{\delta+2}_\chi=G^{\delta+1}_\chi$. $\blacktriangleleft$

$\br$ Suppose that $\delta$ is a limit ordinal $<\kappa$ and that the claim holds for all $\eta<\delta$.
Given $x\in[\chi]^\chi$ and a suitable coloring $c:\varepsilon\rightarrow\chi$ with $\varepsilon\le \delta$,
put  $y:=c[N_\delta]$ and fix some $\xi\in x\setminus y$.

If $\varepsilon=\delta$, then we are done by taking $c':=c$.
Thus, suppose this is not the case, so that  $\epsilon:=\min(C_\delta\setminus \varepsilon)$ is $<\delta$.

We shall recursively construct a $\s$-increasing chain of suitable colorings $\{c_\eta:\eta\rightarrow\chi\mid \eta\in C_\delta\cup\{\delta\}\}$ satisfying all of the following for all $\eta\in C_\delta\cup\{\delta\}$:
\begin{enumerate}
\item[(i)] $c_\eta\restriction \varepsilon\s c$;
\item[(ii)] $c_\eta[\eta\setminus \varepsilon]\s x$;
\item[(iii)] $c_\eta[N_\delta\setminus(\epsilon+1)]\s\{\xi\}$;
\item[(iv)] $c_\eta^{-1}\{\xi\}\s N_\delta\cup(\epsilon+1)$.
\end{enumerate}

Of course, if we succeed, then $c':=c_\delta$ will be as sought. We proceed as follows.

\begin{itemize}
\item For $\eta\in C_\delta\cap\epsilon$, we  simply let $c_\eta:=c\restriction\eta$.
\item For $\eta=\epsilon$, we appeal to the induction hypothesis and find a suitable coloring
$c_\eta:\eta\rightarrow\chi$ extending $c$ such that $c_\eta[\eta\setminus\varepsilon]\s x$.\footnote{Of course, if $\epsilon=\varepsilon$, then $c_\eta=c$.}

\item For $\eta\in\nacc(C_\delta)$ above $\epsilon$, let $\varepsilon':=\sup(C_\delta\cap \eta)$,
so that $\varepsilon\le\epsilon\le\varepsilon'<\eta<\delta$. By the induction hypothesis, we may pick a suitable coloring
$d:\eta\rightarrow\chi$ extending $c_{\varepsilon'}$ and satisfying $d[\eta\setminus\varepsilon']\subseteq x\setminus\{\xi\}$.
Then, we define $c_\eta:\eta\rightarrow\chi$ by letting, for all $\beta<\eta$:
$$c_\eta(\beta):=\begin{cases}
    \xi &\text{if }\beta\in N_\delta\setminus(\epsilon+1);\\
    d(\beta) &\text{otherwise}.
\end{cases}$$

As $d$ extends $c_{\varepsilon'}$ and the latter is assumed to satisfy Clause~(iii) above, we get that $c_\eta$ extends $c_{\varepsilon'}$.
We also have $c_\eta[\eta\setminus \varepsilon']\s d[\eta\setminus\varepsilon']\cup\{\xi\}=x$,
so $c_\eta$ is seen to satisfy Clauses (i)--(iv) above.
Since $d$ is suitable,
and $c_\eta$ differs from $d$ by at most a single color, we have $|c_\eta[N_\gamma]|<\kappa$ for all $\gamma<\kappa$.
Thus, to prove that $c_\eta$ is suitable, it suffices to verify that it is $E$-chromatic.

Towards a contradiction, suppose that $\alpha<\beta<\eta$ are such that $\{\alpha,\beta\}\in E$ and yet $c_\eta(\alpha)=c_\eta(\beta)$.
Since $d$ is $E$-chromatic, the definition of $c_\eta$ makes it clear that $\beta\ge\epsilon+1$ and $c_\eta(\alpha)=c_\eta(\beta)=\xi$.

\begin{itemize}
\item[$\circ$] By $c_\eta(\beta)=\xi$ and $\beta\ge\epsilon+1$, Clause~(iv) implies that  $\beta\in N_\delta$.
\item[$\circ$] By $\beta\in N_\delta$ and $\epsilon\in C_\delta$, we have $\min(C_\beta) > \sup(C_\delta \cap \beta) \geq \epsilon$.
\item[$\circ$] By $\{\alpha, \beta\} \in E$, we have  $\alpha \in C_\beta$, so that $\alpha\ge\min(C_\beta)>\epsilon$.

\item[$\circ$] By $c_\eta(\alpha) = \xi$ and $\alpha\ge\epsilon+1$, Clause~(iv) implies that  $\alpha \in N_\delta$.
\end{itemize}
Altogether, we have established that $\{\alpha,\beta,\delta\}$ is a triangle, contradicting Lemma~\ref{triangle_free_lemma}.

\item For $\eta\in\acc(C_\delta)\cup\{\delta\}$ above $\epsilon$, let $c_\eta:=\bigcup_{\eta'\in C_\delta\cap\eta}c_{\eta'}$.
We now verify that $c_\eta$ is suitable, using the criteria of Lemma~\ref{cohere}.

As $c_\eta$ is the limit of a chain of $E$-chromatic colorings, it is $E$-chromatic. Next, let $\gamma\in G^{\eta+1}_\chi$ be arbitrary.
\begin{itemize}
\item[$\circ$] If $\gamma<\eta$, then, for $\eta':=\min(C_\delta\setminus\gamma)$, we have
that $|c_\eta[N_\gamma]|=|c_{\eta'}[N_\gamma]|<\chi$.

\item[$\circ$] If $\gamma=\eta$, then, by $\gamma\in(\acc(C_\delta)\cup\{\delta\})\cap G$ and hypothesis $(\aleph)$,
we infer that $\delta\in G$. Then, by $\cf(\gamma)=\chi$ and hypothesis $(\beth)$,
we have $N_\gamma=N_\delta\cap\eta$,
so that $$c_\eta[N_\gamma]=c_\eta[N_\delta]=c_\epsilon[N_\delta]\cup\bigcup\{c_{\eta'}[N_\delta\setminus\epsilon]\mid \epsilon\in \eta'\in C_\delta\cap\eta\}.$$
It then follows from Clause~(iii) above that $c_\eta[N_\gamma]\s c_\epsilon[N_\delta]\cup\{c_{\min(C_\delta\setminus(\epsilon+1))}(\epsilon),\xi\}$,
and hence  $|c_\eta[N_\gamma]|<\chi$. $\blacktriangleleft$
\end{itemize}
\end{itemize}

(2) Put $\epsilon:=\min(D)$ and $N:=\{\alpha\in D\mid \min(C_\alpha)>\sup(D\cap\alpha)\ge\epsilon\}$.
Just as in the proof of the case in which $\delta$ is a limit ordinal in Clause~(1),
we recursively construct a chain of suitable colorings $\{c_\eta:\eta\rightarrow\chi\mid \eta\in D\cup\{\kappa\}\}$ satisfying the following two requirements for all $\eta\in D\cup\{\kappa\}$:
\begin{itemize}
\item $c_\eta[N\setminus(\epsilon+1)]\s\{0\}$;
\item $c_\eta^{-1}\{0\}\s N\cup(\epsilon+1)$.
\end{itemize}

Then $c_\kappa\restriction G$ witnesses that $\chr(G(\vec C))\le\chi$.
\end{proof}

In Remark~\ref{remark26}, we outlined the strategy for ensuring that a Hajnal-M\'at\'e graph is uncountably chromatic.
For a $C$-sequence graph, we have the following variation.

\begin{defn}\label{capturing} We say that an ordinal $\delta<\kappa$ \emph{captures} a sequence $\langle A_i\mid i<\theta\rangle$ if the
  following two conditions hold:
\begin{itemize}
\item $\min(C_\delta)\ge\min(A_0)$;\footnote{This is not a typing error. We do mean $A_0$.}
\item for all $i<\min\{\delta,\theta\}$, there exists $\iota\in\otp(C_\delta)$ such that $C_\delta(\iota),C_\delta(\iota+1)\in A_i$.
\end{itemize}
\end{defn}

\begin{lemma} \label{large_chromatic_number_lemma} Suppose that $0<\theta<\kappa$,
and that any sequence $\langle A_i\mid i<\theta\rangle$ of cofinal subsets of $G$
is captured by some ordinal $\delta\in G\setminus\theta$.
Then $\chr(G(\vec C))>\theta$.
\end{lemma}
\begin{proof}
Let $c:G\rightarrow\theta$ be an arbitrary coloring. We shall show that $c$ is not chromatic.

Let $i<\theta$ be arbitrary.
Put $H_i:=\{\alpha\in G\mid c(\alpha)=i\}$ and $M_i:=\{\min(C_\alpha)\mid \alpha\in H_i\}$.

$\br$ If $\sup(M_i)=\kappa$, then define $f_i:\kappa\rightarrow H_i$ by stipulating:
$$f_i(\eta):=\min\{\alpha\in H_i\mid \min(C_\alpha)>\eta\}.$$

$\br$ If $\sup(M_i)<\kappa$, then let $f_i:\kappa\rightarrow G\setminus \sup(M_i)$ be the order-preserving bijection.

\medskip

Consider the club $D:=\bigcap_{i<\theta}\{\beta<\kappa\mid f_i[\beta]\s \beta>0\}$.
For each $i<\theta$, let $A_i$ be some sparse enough cofinal subset of $\rng(f_i)$ such that the following two conditions hold:
\begin{enumerate}
\item $\min(A_i)\ge\min(D)$;
\item for every $\beta<\alpha$, both from $A_i$, we have $D\cap(\beta,\alpha)\neq\emptyset$.
\end{enumerate}

Now, fix some $\delta\in G\setminus\theta$ that captures $\langle A_i\mid i<\theta\rangle$.
Set $j:=c(\delta)$. Note that $\sup(M_j)=\kappa$,
because otherwise $$\sup(M_j)=f_j(0)<\min(D)\le\min(A_0)\le\min(C_\delta),$$ contradicting the fact that
$j=c(\delta)$ entails $\sup(M_j)\ge\min(C_\delta)$.

Pick $\iota\in\otp(C_\delta)$ such that $C_\delta(\iota),C_\delta(\iota+1)\in A_j$.
Denote $\beta:=C_\delta(\iota)$ and $\alpha:=C_\delta(\iota+1)$. Recalling Clause~(2) above,
let us fix some $\gamma\in D\cap(\beta,\alpha)$.
By $\alpha\in A_j$ and $\sup(M_j)=\kappa$, we have $\alpha\in\rng(f_j)$,
so let us fix $\eta<\kappa$ such that $f_j(\eta)=\alpha$.  Then $\min(C_\alpha)>\eta$.

By $f_j[\gamma]\s\gamma<\alpha=f_j(\eta)$, we have $\eta\ge\gamma$,
and hence $$\min(C_\alpha)>\eta\ge\gamma>\beta=\sup(C_\delta\cap\alpha)=C_\delta(\iota)\ge\min(C_\delta).$$
It follows that $\{\alpha,\delta\}\in E$. Recalling that $\alpha\in\rng(f_j)\s H_j$,
we conclude that $c(\alpha)=j=c(\delta)$,  which means that $c$ is not a chromatic coloring of $G(\vec C)$.
\end{proof}

Thus we have established that $\sqx$-coherent and capturing $C$-sequences give rise to graphs witnessing incompactness for the chromatic number.
In later sections, we shall address the existence of such $C$-sequences.

\subsection{Coloring numbers}\label{subsection22}

In this subsection, we discuss compactness and incompactness for the coloring number of graphs.
The primary new result is Theorem \ref{coloring_small_gap_lemma}, indicating that there is a
limit to the amount of incompactness that can be exhibited by the coloring number. We also
review some of the previously known results about obtaining instances of compactness and incompactness for
the coloring number, in particular in connection with general set-theoretic reflection principles
such as $\Delta$-reflection and Fodor-type reflection.
First, a basic observation.

\begin{lemma}[folklore] \label{neighborhood_lemma}
  Suppose $\mathcal G = (G, E)$ is a graph, $\mu$ is an infinite cardinal, and there are
  subsets $A, B \subseteq G$ such that $\mu \leq |A| < |B|$ and, for every
  $y \in B$, $|N_{\mathcal G}(y) \cap A| \geq \mu$. Then $\col(\mathcal G) > \mu$.
\end{lemma}
\begin{proof}
  Suppose not, and let $\lhd$ be a well-ordering of $G$ such that, for all $x \in G$,
  $|N_{\mathcal G}^\lhd(x)| < \mu$. We can then find $y \in B \setminus \bigcup_{x \in A} N_{\mathcal G}^\lhd(x)$,
  and, in turn, $x \in (N_{\mathcal G}(y) \cap A) \setminus N_{\mathcal G}^\lhd(y)$. But then we have
  $\{x,y\} \in E$, $x \notin N_{\mathcal G}^\lhd(y)$, and $y \notin N_{\mathcal G}^\lhd(x)$, which is
  a contradiction.
\end{proof}

The next few lemmas deal with graphs of the form $\mathcal G = (\kappa, E)$.
For each $\alpha<\kappa$, we denote by $\mathcal G_\alpha$ the induced subgraph $(\alpha, E \cap [\alpha]^2)$ .

The following Lemma is essentially due to Shelah \cite{shelah_partition_calculus} and can be found in
its present form in \cite{komjath_colouring_number}.

\begin{lemma}[Shelah] \label{stat_coloring_lemma}
  Suppose $\mathcal G = (\kappa, E)$ is a graph over some regular uncountable cardinal $\kappa$.
  For an infinite cardinal $\mu<\kappa$, consider the set
  \[
    S_\mu(\mathcal G) := \{\alpha < \kappa \mid |N_{\mathcal G}(\beta) \cap \alpha| \geq \mu \text{ for some } \beta \geq \alpha\}.
  \]
  \begin{enumerate}
    \item If $S_\mu(\mathcal G)$ is stationary in $\kappa$, then $\col(\mathcal G) > \mu$.
    \item If $S_\mu(\mathcal G)$ is non-stationary in $\kappa$ and $\col(\mathcal G_\alpha) \leq \mu$ for every $\alpha < \kappa$,
      then $\col(\mathcal G) \leq \mu$.\qed
  \end{enumerate}
\end{lemma}

We now show that a graph can exhibit only a limited amount of incompactness with respect to the coloring number.
This is in sharp contrast to the situation for the chromatic number.

\begin{thm} \label{coloring_small_gap_lemma}
  Suppose that $\mu$ and $\kappa$ are infinite cardinals such that $\kappa$ is regular and is not
  the successor of a singular cardinal of cofinality $\cf(\mu)$.
  If $\mathcal G = (\kappa, E)$ is a graph and $\col(\mathcal G_\alpha) \leq \mu$ for every $\alpha < \kappa$, then $\col(\mathcal G) \leq \mu^+$.
\end{thm}
\begin{proof}
  Assume for sake of contradiction that $\col(\mathcal G_\alpha) \leq \mu$ for every $\alpha < \kappa$, but $\col(\mathcal G) > \mu^+$.
  Then $\kappa>\mu^+$, and by Lemma~\ref{stat_coloring_lemma}(2), $S_{\mu^+}(\mathcal G)$ is stationary in $\kappa$. For each $\alpha \in S_{\mu^+}(\mathcal G)$, fix $\delta_\alpha \geq \alpha$
  such that $|N_{\mathcal G}(\delta_\alpha) \cap \alpha| \geq \mu^+$, and then let $\varepsilon_\alpha < \alpha$ be least such that $|N_{\mathcal G}(\delta_\alpha) \cap \varepsilon_\alpha| = \mu$.

  By Fodor's Lemma, let us fix $\varepsilon^* < \kappa$ and a stationary $S \subseteq S_{\mu^+}(\mathcal G)$ such that, for all
  $\alpha \in S$, $\varepsilon_\alpha = \varepsilon^*$.

  \begin{claim}
    There exists $\mathcal E \subseteq \varepsilon^*$ such that:
    \begin{enumerate}
      \item $|\mathcal E|^+ < \kappa$;
      \item $\{ \alpha\in S\mid |N_{\mathcal G}(\delta_\alpha) \cap\mathcal E| = \mu\}$ is stationary.
    \end{enumerate}
  \end{claim}

  \begin{proof}
    Let $\lambda := |\varepsilon^*|$.

    $\br$ If $\kappa>\lambda^+$, then we simply let $\mathcal E := \varepsilon^*$.

    $\br$ If $\kappa=\lambda^+$, then, by assumption, we have $\cf(\lambda) \neq \cf(\mu)$.
    Fix a bijection $f:\lambda\leftrightarrow\varepsilon^*$.
    For every $\alpha \in S$,
    there is $i_\alpha < \lambda$ such that $|N_{\mathcal G}(\delta_\alpha) \cap f[i_\alpha]| = \mu$. Fix
    an $i^* < \lambda$ and a stationary $S' \subseteq S$ such that, for all $\alpha \in S'$,
    $i_\alpha = i^*$. Then $\mathcal E:= f[i^*]$ is as desired.
  \end{proof}

  Fix $\mathcal E$ as given by the preceding claim, so that $S' := \{\alpha \in S \mid |N_{\mathcal G}(\delta_\alpha) \cap\mathcal E| = \mu\}$ is stationary.
  In particular, $\Delta:=\{ \delta_\alpha\mid \alpha\in S'\}$ is cofinal in $\kappa$.
  By $\kappa>|\mathcal E|^+$, find $\gamma\in(\varepsilon^*,\kappa)$ such that $|\Delta\cap\gamma|\ge|\mathcal E|^+$.
  Now $\mathcal E\s\gamma$, $\Delta \cap \gamma \subseteq \gamma$, $\mu \leq |\mathcal E| < |\Delta \cap \gamma|$ and, for all $\delta \in \Delta \cap \gamma$,
  $|N_{\mathcal G_\gamma}(\delta) \cap\mathcal E| \geq \mu$. Thus, by Lemma~\ref{neighborhood_lemma}, $\col(\mathcal G_\gamma) > \mu$, contradicting
  our assumption.
\end{proof}

We remark that Theorem~\ref{coloring_small_gap_lemma} is consistently sharp:
\begin{lemma}[Shelah, {\cite[Lemma~3.1(1)]{shelah_partition_calculus}}]\label{nonreflecting_set_col} Suppose $\mu < \kappa$ are
infinite regular cardinals and there exists a non-reflecting stationary subset of $E^\kappa_\mu$. Then there exists a $(\mu,\mu^+)$-coloring graph of size $\kappa$.
\end{lemma}
\begin{proof} Let $\Gamma \subseteq E^\kappa_\mu$ be stationary and non-reflecting.
Fix a $\mu$-bounded $C$-sequence over $\Gamma$, $\vec C=\langle C_\alpha\mid \alpha\in\Gamma\rangle$,
and then derive a graph $\mathcal G := (\kappa, E)$ by letting $E := \{\{\alpha, \beta\} \in [\kappa]^2 \mid \beta \in \Gamma, \alpha \in C_\beta\}$.\footnote{Recall Definition~\ref{c-sequence} and Remark~\ref{remark26}.}

We first show that, for all $\gamma < \kappa$, $\col(\mathcal G_\gamma) \leq \mu$. We proceed by
induction on $\gamma$. Thus, suppose $\gamma < \kappa$ and, for all $\eta < \gamma$,
$\col(\mathcal G_\eta) \leq \mu$. If $\gamma = \gamma_0+1$, then fix a well-ordering $\lhd_0$ of $\gamma_0$ witnessing that
$\col(\mathcal G_{\gamma_0}) \leq \mu$, and let $\lhd := \lhd_0 \cup \{(\gamma_0, \alpha) \mid \alpha < \gamma_0\}$. Then
$\lhd$ is a well-ordering of $\gamma$ witnessing that $\col(\mathcal G_\gamma) \leq \mu$. We may thus assume that
$\gamma$ is a limit ordinal. Let $\nu := \cf(\gamma)$. As $\Gamma$ is non-reflecting, fix a club $D$ in $\gamma$
such that $\otp(D) = \nu$, $D \cap \Gamma = \emptyset$ and $0\in D$. For $i < \nu$, let $I_i$ denote the half-open interval
$[D(i), D(i+1))$, and let $\lhd_i$ be a well-ordering of $I_i$ witnessing that the graph $(I_i, E \cap [I_i]^2)$
has coloring number at most $\mu$. For each $\alpha < \gamma$, let $i_\alpha$ be the unique $i < \nu$ such that
$\alpha \in I_i$. Now define a well-ordering $\lhd$ of $\gamma$ by letting $\alpha \lhd \beta$ iff one of the
following two conditions holds:
\begin{itemize}
  \item $i_\alpha < i_\beta$;
  \item $i_\alpha = i_\beta = i$ and $\alpha \lhd_i \beta$.
\end{itemize}
$\lhd$ is easily seen to be a well-ordering of $\gamma$. To see that it witnesses $\col(\mathcal G_\gamma) \leq \mu$,
fix $\beta < \gamma$. Then $$N^\lhd_{\mathcal G_\gamma}(\beta) = \left(N^\in_{\mathcal G_\gamma}(\beta) \cap \bigcup_{i < i_\beta} I_i\right)
\cup N^{\lhd_{i_\beta}}_{(I_{i_\beta}, E \cap [I_{i_\beta}]^2)}(\beta).$$

The second component of this union has size less than $\mu$ by the fact that $\lhd_{i_\beta}$ witnesses $\col(I_{i_\beta}, E \cap [I_{i_\beta}]^2) \leq \mu$.
To deal with the first component, first notice that, if $\beta \in \Gamma$, then $\beta \notin D$, so $\sup(\bigcup_{i < i_\beta} I_i) < \beta$;
by the fact that $\vec{C}$ is $\mu$-bounded, it follows that the first component has size less than $\mu$. If $\beta \notin \Gamma$, then
the first component is empty. This finishes the proof that $\col(\mathcal G_\gamma) \leq \mu$.

We finally show that $\col(\mathcal G) = \mu^+$. Since $\vec C$ is $\mu$-bounded, we have $|N^\in_{\mathcal G}(\beta)| \leq \mu$ for all $\beta < \kappa$,
so that $\col(\mathcal G) \leq \mu^+$. Suppose for sake of contradiction that $\col(\mathcal G) < \mu^+$, and let
$\lhd$ be a well-ordering of $\kappa$ witnessing this. Define a function
$f:\kappa \rightarrow \kappa$ by stipulating $f(\alpha) := \sup(N^{\lhd}_{\mathcal G}(\alpha))$. As $\Gamma$ is stationary,
we can find $\beta \in \Gamma$ such that $f[\beta]\subseteq \beta$. In particular, for all $\alpha \in C_\beta$,
we have $\beta \notin N^{\lhd}_{\mathcal G}(\alpha)$, and hence $\alpha \lhd \beta$. But then $|N^{\lhd}_{\mathcal G}(\beta)| \ge |C_\beta| = \mu$,
which is a contradiction.
\end{proof}

\begin{cor} \label{small_gap_cor}
  Suppose that $\mu$ and $\kappa$ are infinite cardinals, with $\kappa$ regular. If $\mathcal G = (\kappa, E)$ is a graph, and, for every
  $\alpha < \kappa$, $\col(\mathcal G_\alpha) \leq \mu$, then $\col(\mathcal G) \leq \mu^{++}$.
\end{cor}

\begin{proof}
  If $\kappa \le \mu^+$ or if $\kappa$ is not the successor of a singular cardinal of cofinality $\cf(\mu)$, then
  $\col(\mathcal G) \leq \mu^+$ trivially or by Theorem~\ref{coloring_small_gap_lemma}, respectively. If
  $\kappa > \mu^+$ and $\kappa$ is the successor of a cardinal of cofinality $\cf(\mu)$, then
  apply Theorem~\ref{coloring_small_gap_lemma} with $\mu^+$ in place of $\mu$ to obtain
  $\col(\mathcal G) \leq \mu^{++}$.
\end{proof}

To the best of our knowledge, it is unknown whether Corollary \ref{small_gap_cor} is consistently sharp:
\begin{Q} Is it consistent that for some infinite cardinals $\mu<\kappa$, there exists a graph $\mathcal G=(\kappa,E)$ with $\col(\mathcal G)=\mu^{++}$, and yet $\col(\mathcal G_\alpha)\le\mu$ for all $\alpha<\kappa$?
\end{Q}
We can show that certain instances of Chang's Conjecture give us situations in which Corollary \ref{small_gap_cor} is not sharp.
For instance, we have the following.

\begin{lemma}\label{chang_coloring}
  Suppose that $\lambda$ is a singular cardinal of countable cofinality
  and $(\lambda^+, \lambda) \twoheadrightarrow (\aleph_1, \aleph_0)$ holds.
  \begin{enumerate}
    \item If $\mathcal{G} = (\lambda^+, E)$ and $\col(\mathcal{G}_\alpha) \le \aleph_0$ for all $\alpha < \lambda^+$,
      then $\col(\mathcal{G}) \le \aleph_1$.
    \item In some forcing extension, $(\lambda^+, \lambda) \twoheadrightarrow (\aleph_1, \aleph_0)$ remains valid
      and there exists a graph $\mathcal G = (\lambda^+, E)$ such that $\col(\mathcal G_\alpha) \leq \aleph_0$ for all $\alpha < \lambda^+$,
      and $\col(\mathcal{G}) = \aleph_1$.
  \end{enumerate}
\end{lemma}

\begin{proof}
  (1) Suppose not, and let $\mathcal{G}$ be a counterexample.  By Lemma \ref{stat_coloring_lemma}(2),
  $S_{\aleph_1}(\mathcal G)$ is stationary in $\lambda^+$. As in the proof of Theorem~\ref{coloring_small_gap_lemma},
  this allows us to find $\varepsilon^* < \lambda^+$ and a cofinal set $\Delta \subseteq \lambda^+$
  such that, for all $\delta \in \Delta$, $|N_{\mathcal G}(\delta) \cap \varepsilon^*| \geq \aleph_0$.
  By $(\lambda^+, \lambda) \twoheadrightarrow (\aleph_1, \aleph_0)$, we can find
  an elementary substructure $M \prec (H(\lambda^{++}), \in, E, \varepsilon^*, \Delta)$
  such that $|M \cap \lambda^+| = \aleph_1$ and $|M \cap \lambda| = \aleph_0$.

  Let $B = \Delta \cap M$ and note that, by elementarity, $B$ is cofinal in $\sup(M \cap \lambda^+)$
  and hence $|B| = \aleph_1$. Also by elementarity, for every
  $\delta \in B$, $|N_{\mathcal G}(\delta) \cap M \cap \varepsilon^*| \geq \aleph_0$. Therefore, applying
  Lemma \ref{neighborhood_lemma} to $A := M \cap \varepsilon^*$ and $B$, we obtain
  $\col(\mathcal G_{\sup(M \cap \lambda^+)}) > \aleph_0$, contradicting our assumptions.

(2) Let $\mathbb P$ be the standard poset for adding a non-reflecting stationary subset of $E^{\lambda^+}_\omega$ (cf.~\cite[Example~6.5]{MR2768691}),
  and work in the forcing extension by $\mathbb{P}$.
  As $\mathbb{P}$ is $\lambda^+$-strategically closed and $(\kappa_1, \lambda_1) \twoheadrightarrow (\kappa_0, \lambda_0)$ is
  preserved by $\kappa_1$-strategically closed forcing,\footnote{Recall Definition~\ref{thegame}.} $(\lambda^+, \lambda) \twoheadrightarrow (\aleph_1, \aleph_0)$ holds.
  Now, appeal to Lemma~\ref{nonreflecting_set_col}.
\end{proof}

\begin{cor}
  Suppose that $(\aleph_{\omega+1}, \aleph_\omega) \twoheadrightarrow (\aleph_1, \aleph_0)$ holds.
  If $\mu$ is an infinite cardinal and $\mathcal G$ is a  graph of size $\leq \aleph_{\omega \cdot 2}$
  all of whose strictly smaller subgraphs have coloring number $\le\mu$,
  then $\col(\mathcal G) \leq \mu^+$.
\end{cor}

\begin{proof}
  Suppose not, and write $\kappa$ for the size of $\mathcal G$. Clearly, $\mu<\kappa$.
  By Shelah's compactness theorem for singular cardinals \cite{shelah_singular_compactness}, $\kappa$ must be regular.
  It then follows from Theorem~\ref{coloring_small_gap_lemma} that $\kappa=\aleph_{\omega+1}$ and $\mu=\aleph_0$,
  contradicting Lemma~\ref{chang_coloring}(1) with $\lambda=\aleph_\omega$.
\end{proof}

We next recall a strong reflection principle, introduced by Magidor and Shelah in \cite{MR1249391}.

\begin{defn}[Magidor-Shelah, \cite{MR1249391}]
  Suppose $\lambda \leq \kappa$ are infinite cardinals, with $\kappa$ regular. $\Delta_{\lambda, \kappa}$
  is the assertion that, for every stationary $S \subseteq E^\kappa_{<\lambda}$
  and every algebra $A$ on $\kappa$ with fewer than $\lambda$ operations, there is a subalgebra $A'$ of $A$ such that,
  letting $\eta = \otp(A')$, we have:
  \begin{enumerate}
    \item $\eta$ is a regular cardinal;
    \item $\eta < \lambda$;
    \item $S \cap A'$ is stationary in $\sup(A')$.
  \end{enumerate}
  $\Delta_\lambda$ is the assertion that, for all regular $\nu \geq \lambda$, $\Delta_{\lambda, \nu}$ holds.
\end{defn}

Instances of this reflection principle imply instances of compactness for the coloring number.
The following Proposition follows from the arguments of \cite[\S2]{shelah_singular_compactness}
and \cite[\S2]{MR1249391}. We provide a direct proof for completeness.

\begin{prop} \label{delta_reflection_coloring_prop}
  Suppose $\aleph_0 \leq \mu < \kappa = \cf(\kappa)$ and there is a cardinal $\lambda$ such that $\mu < \lambda \leq \kappa$
  and $\Delta_{\lambda, \kappa}$ holds. Then any $\kappa$-sized graph of coloring number $>\mu$ has a $(<\kappa)$-sized subgraph of coloring number $>\mu$.
\end{prop}

\begin{proof}
  Let $\mathcal G$ be a graph of size $\kappa$ such that, for every smaller subgraph $\mathcal G'$ of $\mathcal G$, $\col(\mathcal G') \leq \mu$.
  Suppose for sake of contradiction that $\col(\mathcal G) > \mu$. Without loss of generality, $V(\mathcal G) = \kappa$.
  Then, by Lemma \ref{stat_coloring_lemma},
  \[
    S_0 := \{\alpha < \kappa \mid     |N_{\mathcal G}(\beta) \cap \alpha| \geq \mu \text{ for some } \beta \geq \alpha\}
\]
  is stationary in $\kappa$. It is easily seen that this implies that $S := S_0 \cap E^\kappa_{\cf(\mu)}$
  is stationary in $\kappa$. Let $A$ be an algebra on $\kappa$ equipped with the following functions:
  \begin{itemize}
    \item a function $f$ on $\kappa$ such that, for each $\alpha \in S$, $f(\alpha)$ is the least
      $\delta \geq \alpha$ such that $|N_{\mathcal G}(\delta) \cap \alpha| \geq \mu$;
    \item for each $\zeta < \mu$, a function $g_\zeta$ on $\kappa$ such that, for each $\alpha \in S$,
      $g_\zeta(\alpha)$ is the unique element $\epsilon$ of $N_{\mathcal G}(f(\alpha))$ such that $\otp(N_{\mathcal G}(f(\alpha)) \cap \epsilon) = \zeta$;
    \item for each $\zeta < \mu$, a constant function $h_\zeta$ on $\kappa$ taking value $\zeta$.
  \end{itemize}
  Apply $\Delta_{\lambda, \kappa}$ to $A$ and $S$ to find a subalgebra $A'$ such that $\eta:=\otp(A')$ is a regular cardinal $< \lambda$
  and $S \cap A'$ is stationary in $\sup(A')$. Since $A'$ is closed under $h_\zeta$ for $\zeta < \mu$,
  we have $\mu \s A'$. Let $\pi:A' \rightarrow \eta$ be
  the unique order-preserving bijection, and let $\mathcal H := (\eta, F)$ be the graph on $\eta$ defined by
  $\{\alpha, \delta\} \in F$ iff $\{\pi^{-1}(\alpha), \pi^{-1}(\delta)\} \in E$. Note that, since
  $A' \cap S$ is stationary in $\sup(A')$, $T := \pi[S]$ is stationary in $\eta$. Also, since $A'$ is closed under
  $f$ and $g_\zeta$ for each $\zeta < \mu$ and $\mu \subseteq A'$, we have that, for all $\alpha \in T$, there is $\delta \geq \alpha$
  such that $|N_{\mathcal H}(\delta) \cap \alpha| \geq \mu$. Therefore, by Lemma \ref{stat_coloring_lemma}, $\col(\mathcal H) > \mu$.
  But $\pi$ witnesses that $\mathcal H$ and $\mathcal G \restriction A' := (A', E \cap [A']^2)$ are isomorphic graphs,
  so that $\col(\mathcal G \restriction A') > \mu$, contradicting the assumption that every smaller subgraph of $\mathcal G$
  has coloring number at most $\mu$.
\end{proof}

In \cite{frp}, Fuchino et al. introduce the following reflection principle.

\begin{defn}[Fuchino et al., \cite{frp}]
  Let $\kappa$ be a cardinal of uncountable cofinality. The \emph{Fodor-type Reflection Principle
  for $\kappa$} ($\frp(\kappa$)) is the assertion that, for every stationary $S \subseteq E^\kappa_\omega$
  and every function $g:S \rightarrow [\kappa]^{\leq \aleph_0}$, there is $I \in [\kappa]^{\aleph_1}$
  such that:
  \begin{enumerate}
    \item $\cf(I) = \omega_1$;
    \item $g(\alpha) \subseteq I$ for all $\alpha \in I \cap S$;
    \item for every regressive $f:S \cap I \rightarrow \kappa$ such that $f(\alpha) \in g(\alpha)$ for
      all $\alpha \in S \cap I$, there is $\xi < \kappa$ such that $f^{-1}``\{\xi\}$ is
      stationary in $\sup(I)$.
  \end{enumerate}
  For an uncountable cardinal $\lambda$, $\frp(<\lambda$) is the assertion that $\frp(\kappa$) holds for
  every regular, uncountable $\kappa < \lambda$.
\end{defn}

Note that $\frp(\aleph_1$) is trivially true, so the first interesting case is $\frp(\aleph_2$).
In \cite{frp_consistency}, Miyamoto shows that the consistency strength of $\frp(\aleph_2$) is
precisely that of a Mahlo cardinal. In particular, starting in a model with a Mahlo cardinal,
$\kappa$, he produces a forcing extension in which $\kappa = \aleph_2$ and $\gch$ and $\frp(\aleph_2$)
both hold.

As mentioned in the Introduction, in \cite{more_frp}, it is proven that instances of $\frp$ are in fact equivalent to
instances of compactness for countable coloring numbers.

\begin{fact}[Fuchino et al., {\cite[Theorem 3.1]{more_frp}}] \label{frp_thm}
  For any cardinal $\lambda \geq \aleph_2$, $\frp(<\lambda$) is equivalent to the assertion that,
  if $\mathcal G$ is a graph such that $|V(\mathcal G)| < \lambda$ and $\col(\mathcal G') \leq \aleph_0$ for every subgraph
  $\mathcal G'$ of $\mathcal G$ with $|V(\mathcal G')| \leq \aleph_1$, then $\col(\mathcal G) \leq \aleph_0$.
\end{fact}

In general, stationary reflection assumptions of the form $\refl(S)$ are not sufficient to imply
instances of compactness for the coloring number. For example, in \cite{more_frp}, Fuchino et al. produce
a model in which $\refl(S^{\aleph_2}_{\aleph_0})$ holds and yet $\frp(\aleph_2)$ fails. By Fact~\ref{frp_thm},
there is an $(\aleph_0, \ge \aleph_1)$-coloring graph of size $\aleph_2$ in this model.
Nevertheless, it is the case that Shelah's model from \cite{shelah_reflection} for
exhibiting the maximum possible extent of stationary reflection also exhibits the maximum possible extent of
compactness for the coloring number.
\begin{fact}[implicit in Shelah, \cite{shelah_reflection}]
  Suppose there is a proper class of supercompact cardinals. Then there is a class forcing extension in
  which \zfc\ holds and, for every infinite cardinal $\mu$ and every graph $\mathcal{G}$, if
  $\col(\mathcal{G}) > \mu$, then there is a subgraph $\mathcal{G}'$ of $\mathcal{G}$ such that
  $|V(\mathcal{G})| = \col(\mathcal{G}) = \mu^+$.
\end{fact}

\section{Obtaining coherent and capturing $C$-sequences}\label{improve_section}

Throughout this section, $\kappa$ denotes a regular, uncountable cardinal.

\medskip

In  \cite{paper22} and \cite{paper23}, as an alternative foundation for constructing $\kappa$-Souslin trees,
Brodsky and Rinot introduce the parameterized proxy principle $\p^-(\kappa,\ldots)$.
As will soon be made clear, the main results of Subsection~\ref{subsection21} suggest that instances of this proxy principle give rise to incompactness graphs.
The goal of this section is to establish this connection.

\begin{defn}[special case of {\cite{paper22}}]\label{proxy}
Suppose that:
\begin{itemize}
\item $\mathcal R$ is a binary relation over $[\kappa]^{<\kappa}$;
\item $\theta$ is a cardinal such that $1 \leq \theta \leq \kappa$;
\item $\mathcal S$ is a nonempty collection of stationary subsets of $\kappa$;
\item $\sigma$ is an ordinal $\leq \kappa$.
\end{itemize}
The principle $\p^-(\kappa,2,\mathcal R,\theta,\mathcal S,2,\sigma)$
asserts the existence of a sequence $\vec C=\langle C_\alpha \mid \alpha < \kappa \rangle$
such that:
\begin{itemize}
\item for every limit ordinal $\alpha < \kappa$, $C_\alpha$ is a club subset of  $\alpha$;
\item for every limit ordinal $\alpha < \kappa$ and every $\bar\alpha \in \acc(C_\alpha)$, we have $C_{\bar\alpha} \mathrel{\mathcal R} C_\alpha$;
\item
for every sequence $\langle A_i \mid i < \theta \rangle$ of cofinal subsets of $\kappa$
and every $S \in \mathcal S$, there exist stationarily many $\alpha \in S$ such that for all $i<\min\{\alpha, \theta\}$:
$$\sup\{ \beta \in C_\alpha \mid \suc_\sigma (C_\alpha \setminus \beta) \subseteq A_i \} = \alpha.$$
\end{itemize}
Finally, $\p(\kappa,2, \mathcal R, \theta, \mathcal S,2,\sigma)$
asserts that  $\p^-(\kappa,2, \mathcal R, \theta, \mathcal S,2,\sigma)$
and $\diamondsuit(\kappa)$ both hold.
\end{defn}

Looking at Lemmas \ref{lemma1} and \ref{large_chromatic_number_lemma}, we see that if $\vec C$
witnesses the validity of $\p^-(\kappa,2,{\sqx},\theta,\{\kappa\},2,2)$, then the graph $G(\vec C)$ with $G:=\acc(\kappa)$
is very close to being $(\chi,\ge\theta^+)$-chromatic.
Specifically, by Lemma~\ref{lemma1}, we have $\chr(G(\vec C\restriction\delta))\le\chi$ for every $\delta<\kappa$,
so that $G(\vec C)$ is indeed $(\chi,\mu)$-chromatic for some cardinal $\mu$.
Now, to establish that $\mu\ge\theta^+$, we would like to take advantage of Lemma~\ref{large_chromatic_number_lemma},
however the first bullet of Definition~\ref{capturing} is not addressed by the proxy principle.
Nevertheless, in Theorem~\ref{thm37} below, we provide four scenarios in which this missing feature may be added.

Another worry is to derive instances of the proxy principle from simple combinatorial hypotheses (such as the conjunction of $\square$ and $\diamondsuit$) and via forcing.
The former approach is taken in Subsection~\ref{subsection32}, and the latter approach is taken in Subsection~\ref{forcing_subsection}.

The first subsection, Subsection~\ref{subsection31}, develops some of the machinery needed to establish the results of Subsection~\ref{subsection32}.
However, due to its technical nature, the reader may prefer to first read Subsection~\ref{subsection32} before digging into Subsection~\ref{subsection31}.

\subsection{Postprocessing functions}\label{subsection31}

The next two definitions are taken from \cite{paper29}.
\begin{defn}$\mathcal K(\kappa):=\{ x\in\mathcal P(\kappa)\mid x\text{ is a nonempty club subset of }\sup(x)\}$.
\end{defn}
\begin{defn}
A function $\Phi:\mathcal K(\kappa)\rightarrow\mathcal K(\kappa)$ is a \emph{postprocessing function}  if for every $x\in\mathcal K(\kappa)$:
\begin{itemize}
\item  $\Phi(x)$ is a club in $\sup(x)$;
\item $\acc(\Phi(x)) \s \acc(x)$;
\item for all $\bar\alpha\in\acc(\Phi(x))$, we have $\Phi(x)\cap\bar\alpha=\Phi(x\cap\bar\alpha)$.
\end{itemize}
The function $\Phi$ is said to be \emph{$\acc$-preserving} iff $\acc(\Phi(x))=\acc(x)$ for every $x$.
\end{defn}
By convention, for every postprocessing function $\Phi$ and every $x \in \mathcal{P}(\kappa) \setminus \mathcal{K}(\kappa)$,
we set $\Phi(x) = x$. The point is that, for various binary relations $\mathcal R$, if  $\langle C_\alpha\mid\alpha<\kappa\rangle$ is an $\mathcal R$-coherent $C$-sequence,
then so is $\langle \Phi(C_\alpha)\mid \alpha<\kappa\rangle$.
\begin{remark} Note that the composition of ($\acc$-preserving) postprocessing functions is again an ($\acc$-preserving) postprocessing function.
\end{remark}

\begin{example}[\cite{paper23}]\label{fact13} Suppose that $\xi<\kappa$ is an ordinal.
  Define $\Phi_\xi:\mathcal K(\kappa)\rightarrow\mathcal K(\kappa)$ by stipulating:
$$\Phi_\xi(x):=\begin{cases}x\setminus x(\xi)&\text{if }\otp(x)>\xi;\\x&\text{otherwise}.\end{cases}$$

Then $\Phi_\xi$ is a postprocessing function.\qed
\end{example}

\begin{example}[\cite{paper29}]\label{phiZ}  Suppose that $\mathfrak Z=\langle Z_{x,\beta}\mid x\in\mathcal K(\kappa), \beta\in\nacc(x)\rangle$ is
  a matrix of elements of $\mathcal P(\kappa)$.
For each $x\in\mathcal K(\kappa)$, define $g_{x,\mathfrak Z}:x\rightarrow\sup(x)$ by stipulating:
$$g_{x,\mathfrak Z}(\beta) :=
\begin{cases}
\beta   &\text{if } \beta \in \acc(x);  \\
\min((Z_{x,\beta}\cap\beta)\cup\{\beta\})           &\text{if } \beta = \min(x);\\
\min\left(((Z_{x,\beta}\cap\beta)\cup\{\beta\})\setminus(\sup(x\cap\beta)+1)\right) &\text{otherwise}.
\end{cases}$$

Then:
\begin{enumerate}
\item $g_{x,\mathfrak Z}$ is strictly increasing, continuous, and cofinal in $\sup(x)$;
\item $\acc(\rng(g_{x,\mathfrak Z}))=\acc(x)$ and $\nacc(\rng(g_{x,\mathfrak Z}))=g_{x,\mathfrak Z}[\nacc(x)]$;
\item if $Z_{x,\beta}=Z_{x\cap\bar\alpha,\beta}$ for all $x\in\mathcal K(\kappa)$, $\bar\alpha\in\acc(x)$ and  $\beta\in\nacc(x\cap\bar\alpha)$,
then $x\stackrel{\Phi_{\mathfrak Z}}{\mapsto}\rng(g_{x,\mathfrak Z})$ is an $\acc$-preserving postprocessing function.\qed
\end{enumerate}
\end{example}

\begin{lemma}\label{PhiB} Suppose that $B \subseteq \kappa$, and define $\Phi:\mathcal K(\kappa)\rightarrow\mathcal K(\kappa)$ by stipulating:
$$\Phi(x):=\begin{cases}\cl(\nacc(x)\cap B) &\text{if }\sup(\nacc(x)\cap B)=\sup(x);\\
x\setminus\sup(\nacc(x)\cap B) &\text{otherwise}.\end{cases}$$

Then $\Phi$ is a postprocessing function.
\end{lemma}
\begin{proof} Let $x\in\mathcal K(\kappa)$ be arbitrary.
It is easy to see that $\Phi(x)$ is a club in $\sup(x)$, and $\acc(\Phi(x))\s\acc(x)$.
Next, suppose that $\bar\alpha\in\acc(\Phi(x))$. Put $\varepsilon:=\sup(\nacc(x)\cap B)$. There are two cases to consider:

$\br$ If $\varepsilon<\sup(x)$, then  $\sup(\nacc(x\cap\bar\alpha)\cap B)=\varepsilon<\bar\alpha$,
and hence $\Phi(x\cap\bar\alpha)=(x\cap\bar\alpha)\setminus\varepsilon=(x\setminus\varepsilon)\cap\bar\alpha=\Phi(x)\cap\bar\alpha$.

$\br$ If $\varepsilon=\sup(x)$, then $\sup(\nacc(x\cap\bar\alpha)\cap B)=\bar\alpha$,
and hence $\Phi(x\cap\bar\alpha)=\cl(\nacc(x\cap\bar\alpha)\cap B)=\cl(\nacc(x)\cap B)\cap\bar\alpha=\Phi(x)\cap\bar\alpha.$
\end{proof}

The next lemma provides a tool for transforming a witness to $\p^-(\kappa,2,\mathcal R,1,\ldots)$ into a witness to $\p^-(\kappa,2,\mathcal R,\kappa,\ldots)$.
This is of interest, because, by Lemma~\ref{large_chromatic_number_lemma}, having value $\kappa$ as the fourth parameter of the proxy principle is tied to having maximal degree of incompactness.

\begin{lemma}\label{Phi} Let $\rho\in\acc(\kappa)$, and suppose that  $\diamondsuit(\kappa)$ holds.
Then there exists a postprocessing function $\Phi^\rho:\mathcal K(\kappa)\rightarrow\mathcal K(\kappa)$ such that, for every sequence $\vec A=\langle A_i\mid i<\kappa\rangle$ of cofinal subsets of $\kappa$,
there exists some stationary subset $G\s \kappa$ that codes $\vec A$, as follows.
For every $x\in\mathcal K(\kappa)$:
\begin{enumerate}
\item if $\sup(\nacc(x)\cap G)=\sup(x)$, $\otp(x)\le\rho$, and $(\cf(\sup(x)))^+=\kappa$,
then for all $i<\sup(x)$:$$\sup(\nacc(\Phi^\rho(x))\cap A_i)=\sup(x);$$
\item if $\otp(\nacc(x)\cap G)=\sup(x)>\rho$,
then for all $i<\sup(x)$:$$\sup(\nacc(\Phi^\rho(x))\cap A_i)=\sup(x);$$
\item if $\otp(x)$ is a cardinal $\le\rho$ whose successor is $\kappa$, and $\nacc(x)\s G$, then, for all $\sigma<\otp(x)$ and all $i<\sup(x)$:
$$\sup\{ \beta\in x \mid \suc_\sigma (\Phi^\rho(x) \setminus \beta) \s A_i \} = \sup(x).$$
\end{enumerate}
\end{lemma}
\begin{proof}
Fix a $\diamondsuit(\kappa)$-sequence, $\langle S_\beta\mid\beta<\kappa\rangle$.
Denote $\lambda:=|\rho|$.

$\br$ If $\kappa=\lambda^+$, then attach an injection $\varphi_x:\sup(x)\rightarrow\otp(x)\times\lambda$ to each $x\in\mathcal K(\kappa)$
in such a way that $\varphi_{x\cap\bar\alpha}\s\varphi_x$ for all $x\in\mathcal K(\kappa)$ and all $\bar\alpha\in\acc(x)$. This is indeed possible, as established in \cite[\S3]{paper29}.

$\br$ If $\kappa\neq\lambda^+$, then let $\varphi_x:\sup(x)\rightarrow\sup(x)\times\lambda$ be such that $\varphi_x(\beta)=(\beta,0)$ for all $\beta<\sup(x)$.

Next, fix a bijection $f:\kappa\leftrightarrow\kappa\times\kappa$ such that $f\restriction\lambda$ is a bijection from $\lambda$ to $\rho\times\lambda$.
Fix a surjection $g:\kappa\rightarrow\kappa$ such that $g\restriction\lambda$ satisfies that, for every $j,\sigma<\lambda$,
the set $\{ k<\lambda\mid g``(k,k+\sigma)=\{j\}\}$ is cofinal in $\lambda$.
Fix a bijection $\pi:\kappa\times\kappa\leftrightarrow\kappa$
and let $$E:=\{\alpha<\kappa\mid \pi[\alpha\times\alpha]=f^{-1}[\alpha\times\alpha]=g[\alpha]=\alpha\}.$$
Of course, every nonzero element of $E$ is an indecomposable ordinal.

Let $x\in \mathcal K(\kappa)$ be arbitrary. Put
$$N_x := \left\{ \beta\in\nacc(x)\cap E \mid  \text{for all } \varepsilon,\gamma<\beta,
\text{ there exists } \tau\in\beta\setminus\gamma \text{ with } \pi(\varepsilon,\tau)\in S_\beta\right\}.$$

Define $h_x:\nacc(x)\rightarrow\otp(x)$ by letting, for all $\gamma\in x$:
$$h_x(\gamma):=\begin{cases}
\otp\{\beta\in N_x\cap\gamma\mid S_\beta=S_\gamma\cap\beta, \otp(x\cap\beta)>\rho\} &\text{if }\otp(x)>\rho;\\
\otp\{\beta\in N_x\cap\gamma\mid S_\beta=S_\gamma\cap\beta\}\pmod\lambda &\text{otherwise}.
\end{cases}$$

Then, define $\phi_x:\nacc(x)\rightarrow\sup(x)$ by letting:
$$\phi_x(\beta):=\begin{cases}
\varepsilon &\text{if }\varepsilon<\beta\ \&\ \varphi_x(\varepsilon)=f(g(h_x(\beta)));\\
0 &\text{otherwise}.\end{cases}$$

As $\varphi_x$ is injective, $\phi_x$ is well-defined.
Define $\mathfrak Z=\langle Z_{x,\beta}\mid x\in\mathcal K(\kappa), \beta\in\nacc(x)\rangle$ by stipulating:
$$Z_{x,\beta}:=\{\tau<\beta\mid  \pi(\phi_x(\beta),\tau)\in S_\beta\}.$$

Let $\Phi_\rho$ be given by Example~\ref{fact13}.
Let $\Phi_{\mathfrak Z}$ be given by Example~\ref{phiZ}. Note that the definition of $h_x$ prevents
$\Phi_{\mathfrak Z}$ from being a postprocessing function. Nonetheless, we have the following Claim.

\begin{claim} $\Phi^\rho:=\Phi_\rho\circ\Phi_{\mathfrak Z}$ is a postprocessing function.
\end{claim}
\begin{proof} Let $x\in\mathcal K(\kappa)$ be arbitrary.
By Clauses (1) and (2) of Example~\ref{phiZ}, $\Phi_{\mathfrak Z}(x)$ is a club in $\sup(x)$ with $\acc(\Phi_{\mathfrak Z}(x))=\acc(x)$.
Consequently, $\Phi^\rho(x)$ is a club in $\sup(x)$ with $\acc(\Phi^\rho(x))\s\acc(x)$. Next, suppose that $\bar\alpha\in\acc(\Phi^\rho(x))$. There are two cases to consider:

$\br$ If $\otp(x)\le\rho$, then $\otp(\Phi_{\mathfrak Z}(x))\le\rho$,
so that $\Phi^\rho(x)=\Phi_{\mathfrak Z}(x)$ and $\Phi^\rho(x\cap\bar\alpha)=\Phi_{\mathfrak Z}(x\cap\bar\alpha)$.
Thus it suffices to prove that $g_{x,\mathfrak Z}(\gamma)\restriction\bar\alpha=g_{x\cap\bar\alpha,\mathfrak Z}$.

Clearly, $N_{x\cap\bar\alpha}=N_x\cap\bar\alpha$, and so, by $\otp(x\cap\bar\alpha)<\otp(x)\le\rho$, we have $h_{x\cap\bar\alpha}\s h_x$.
Consequently, $\phi_{x\cap\bar\alpha}\s \phi_x$ and $Z_{x\cap\bar\alpha,\beta}=Z_{x,\beta}$ for all $\beta\in\nacc(x\cap\bar\alpha)$.
It then immediately follows that $g_{x\cap\bar\alpha,\mathfrak Z}\s g_{x,\mathfrak Z}$.

$\br$ If $\otp(x)>\rho$,  then, by $\acc(\Phi_{\mathfrak Z}(x))=\acc(x)$, we have $\otp(\Phi_{\mathfrak Z}(x))>\rho$,
and hence $\Phi^\rho(x)=\Phi_{\mathfrak Z}(x)\setminus(\Phi_{\mathfrak Z}(x))(\rho)$. But $\rho$ is a nonzero limit ordinal,
so that $\Phi_{\mathfrak Z}(x)(\rho)=x(\rho)$ and hence $\Phi^\rho(x)=\Phi_{\mathfrak Z}(x)\setminus x(\rho)$.
By $\bar\alpha\in\acc(\Phi_{\mathfrak Z}(x)\setminus x(\rho))$, we have $\otp(x\cap\bar\alpha)>\rho$,
so a similar argument shows that $\Phi^\rho(x\cap\bar\alpha)=\Phi_{\mathfrak Z}(x\cap\bar\alpha)\setminus x(\rho)$.

Let $\zeta$ denote the unique element of $x$ satisfying $\otp(x\cap\zeta)=\rho+1$.
Then we have established that $\Phi^\rho(x)=\rng(g_{x,\mathfrak Z}\restriction(x\setminus\zeta))$
and $\Phi^\rho(x\cap\bar\alpha)=\rng(g_{x\cap\bar\alpha,\mathfrak Z}\restriction((x\cap\bar\alpha)\setminus\zeta))$.
Thus it suffices to prove that $g_{x,\mathfrak Z}(\gamma)\restriction [\zeta,\bar\alpha)=g_{x\cap\bar\alpha,\mathfrak Z}\restriction [\zeta,\bar\alpha)$.

Clearly, $N_{x\cap\bar\alpha}=N_x\cap\bar\alpha$, and so, by $\otp(x)>\otp(x\cap\bar\alpha)>\rho$, we have $h_{x\cap\bar\alpha}\s h_x$.
Consequently, $\phi_{x\cap\bar\alpha}\s \phi_x$ and $Z_{x\cap\bar\alpha,\beta}=Z_{x,\beta}$ for all $\beta\in\nacc(x\cap\bar\alpha)$.
It then immediately follows that $g_{x\cap\bar\alpha,\mathfrak Z}=g_{x,\mathfrak Z}\restriction\bar\alpha$.
\end{proof}

Next, suppose that $\langle A_i\mid i<\kappa\rangle$ is a sequence of cofinal subsets of $\kappa$.
Let $\pi_0:\kappa\rightarrow\kappa$ be such that $\pi_0(\varepsilon)=i_0$ iff $\pi(i_0,i_1)=\varepsilon$ for some  $i_1$.
For all $\varepsilon<\kappa$, let $B_\varepsilon:=A_{\pi_0(\varepsilon)}$.

Consider the club $D:=E\cap\bigdiagonal_{\varepsilon<\kappa}(\acc^+(B_\varepsilon))$, the set $S:=\{ \pi(\varepsilon,\tau)\mid \varepsilon<\kappa, \tau\in B_\varepsilon\}$,
and the stationary set $G:=\{\beta\in D\mid S\cap\beta= S_\beta\}$.

\begin{claim}\label{c122} For every $x\in\mathcal K(\kappa)$ and $\gamma\in G$, $\nacc(x)\cap G\cap\gamma=\{\beta\in N_x\cap\gamma\mid S_\gamma\cap\beta = S_\beta\}$.
\end{claim}
\begin{proof} Let $x\in\mathcal K(\kappa)$ be arbitrary. For all $\gamma\in G$, we have $S_\gamma=S\cap\gamma$.
Thus, let us prove that $\nacc(x)\cap G=\{\beta\in N_x\mid S\cap\beta = S_\beta\}$.

$(\subseteq)$: Let $\beta\in \nacc(x)\cap G$ be arbitrary. Then $\beta\in D\s E$ and $S\cap \beta=S_\beta$.
By $\beta\in D$, we also have  $\beta\in \bigcap_{\varepsilon<\beta}\acc^+(B_\varepsilon)$.
Thus, for all $\varepsilon, \gamma < \beta$, there is some $\tau\in B_\varepsilon \cap (\beta \setminus \gamma)$
such that $\pi(\varepsilon, \tau)\in S$ and (since $\beta\in E$) $\pi(\varepsilon, \tau) < \beta$,
giving $\pi(\varepsilon, \tau)\in S \cap \beta = S_\beta$.
Thus $\beta\in N_x$.

$(\supseteq)$: Suppose that $\beta\in N_x$ satisfies $S\cap\beta=S_\beta$.
By $\beta\in N_x \s \nacc(x) \cap E$,
it remains to show that $\beta\in \bigcap_{\varepsilon<\beta}(\acc^+(B_\varepsilon))$.
Consider any $\varepsilon, \gamma < \beta$.
Since $\beta\in N_x$,
we can fix $\tau\in\beta\setminus\gamma$ such that
$\pi(\varepsilon,\tau)\in S_\beta$. That is, $\pi(\varepsilon,\tau)\in S\cap \beta$
and $\tau\in B_\varepsilon \cap (\beta \setminus \gamma)$, as required.
\end{proof}

\begin{claim}\label{c143} Suppose that $x\in\mathcal K(\kappa)$, $\sup(\nacc(x)\cap G)=\sup(x)$, $\otp(x)\le\rho$, and $(\cf(\sup(x)))^+=\kappa$.
Then $\sup(\nacc(\Phi^\rho(x))\cap A_i)=\sup(x)$ for all $i<\sup(x)$.
\end{claim}
\begin{proof} Denote $\delta:=\sup(x)$. Since $\lambda$ is a cardinal $<\kappa$ and $(\cf(\delta))^+=\kappa$, we have $\lambda\le\cf(\delta)<\kappa$,
so that $\cf(\delta)=\lambda$. By $\otp(x)\le\rho$, we have $\Phi^\rho(x)=\Phi_{\mathfrak Z}(x)$, so that $\nacc(\Phi^\rho(x))=g_{x,\mathfrak Z}[\nacc(x)]$.

Let $i,\alpha<\delta$ be arbitrary. We shall find $\beta\in\nacc(x)$ such that $g_{x,\mathfrak Z}(\beta)\in A_i\setminus\alpha$.
By increasing $\alpha$, we may assume that  $\alpha>i$ and $\alpha\in\nacc(x)\cap G$.
 In particular,  $\pi[\alpha\times\alpha]=\alpha$, and we may find some $\varepsilon<\alpha$ such that $\pi_0(\varepsilon)=i$.
As $\rng(\varphi_x)\s\rho\times\lambda$ and $f\restriction\lambda$ is a bijection from $\lambda$ to $\rho\times\lambda$, $j:=f^{-1}(\varphi_x(\varepsilon))$ is an element of $\lambda$.

Let $M_x:=\nacc(x)\cap G$. By Claim~\ref{c122}, for all $\gamma\in M_x$, $h_x(\gamma) = \otp(M_x \cap \gamma) \pmod \lambda$.
By $\sup(M_x)=\delta$ and $\cf(\delta)=\lambda$, we know that $h_x[M_x\setminus(\alpha+1)]$ is co-bounded in $\lambda$.
By the choice of $g$, then, we may pick $k\in h_x[M_x\setminus(\alpha+1)]$ such that $g(k)=j$.
Pick $\beta\in M_x\setminus(\alpha+1)$ such that $h_x(\beta)=k$.
Then $f(g(h_x(\beta)))=\varphi_x(\varepsilon)$ and $\varepsilon<\alpha<\beta$ so that $\phi_x(\beta)=\varepsilon$,
and hence
$$Z_{x,\beta}=\{\tau<\beta\mid  \pi(\varepsilon,\tau)\in S\cap\beta\}=B_\varepsilon\cap \beta.$$

By $\varepsilon<\alpha<\beta$,$\alpha\in\nacc(x)$, and $\beta\in G\s \acc^+(B_\varepsilon)$, we have
$$g_{x,\mathfrak Z}(\beta)\in B_\varepsilon\setminus(\sup(x\cap\beta)+1)\s A_i\setminus(\alpha+1),$$
as sought.
\end{proof}
\begin{claim}  Suppose that $x\in\mathcal K(\kappa)$ and $\otp(\nacc(x)\cap G)=\sup(x)>\rho$.
Then $\sup(\nacc(\Phi^\rho(x))\cap A_i)=\sup(x)$ for all $i<\sup(x)$.
\end{claim}
\begin{proof} Denote $\delta:=\sup(x)$. As $\delta$ is an accumulation point of $E$,
we know that $\delta$ is indecomposable.
By $\otp(x)>\rho$, let $\zeta$ denote the unique element of $x$ satisfying $\otp(x\cap\zeta)=\rho+1$.
Then $\nacc(\Phi^\rho(x))=g_{x,\mathfrak Z}[\nacc(x\setminus\zeta)]$.

Let $i<\alpha<\delta$ be arbitrary.
We shall find $\beta\in\nacc(x\setminus \zeta)$ such that $g_{x,\mathfrak Z}(\beta)\in A_i\setminus\alpha$.
By increasing $\alpha$, we may assume that $\alpha\in\nacc(x)\cap G\setminus\zeta$.
In particular, $\alpha\in E$, and we may fix some $\varepsilon<\alpha$ such that $\pi_0(\varepsilon)=i$.
Put $j:=f^{-1}(\varphi_x(\varepsilon))$. By $\alpha\in G\s E$, we have $j<\alpha$.
By $\alpha\in E$, we may also fix $k<\alpha$ such that $g(k)=j$.

Let $M_x:=\nacc(x)\cap G\setminus\zeta$.
By $\otp(x\cap\zeta)=\rho+1$ and Claim~\ref{c122}, for all $\gamma\in M_x$, $h_x(\gamma) = \otp(M_x \cap \gamma)$.
As $\otp(\nacc(x)\cap G)=\delta$ and the latter is indecomposable, we have $\otp(M_x)=\delta>\alpha>k$,
and hence we may pick some $\beta\in M_x$ such that $h_x(\beta)=k$.
Then $f(g(h_x(\beta)))=\varphi_x(\varepsilon)$,  $\phi_x(\beta)=\varepsilon$,
and $Z_{x,\beta}=B_\varepsilon\cap\beta$.
Then $g_{x,\mathfrak Z}(\beta)\in B_\varepsilon\setminus(\sup(x\cap\beta)+1)\s A_i\setminus(\alpha+1)$,
as sought.
\end{proof}
\begin{claim} Suppose that $x\in\mathcal K(\kappa)$, $\otp(x)$ is a cardinal $\le\rho$ whose successor is $\kappa$, and $\nacc(x)\s G$.
Then, for all $i<\sup(x)$ and $\sigma<\otp(x)$, we have
$$\sup\{ \beta\in x \mid \suc_\sigma (\Phi^\rho(x) \setminus \beta) \s A_i \} = \sup(x).$$
\end{claim}
\begin{proof} Denote $\delta:=\sup(x)$. Clearly, $\otp(x)=\lambda$ and $\Phi^\rho(x)=\rng(g_{x,\mathfrak Z})$.

By $\nacc(x)\s G$, we get from Claim~\ref{c122} that $h_x:\nacc(x)\rightarrow\lambda$ is the order-preserving bijection.
Let $i<\alpha<\delta$ be arbitrary, and let $\sigma<\lambda$ be arbitrary.
By increasing $\alpha$, we may assume that $\alpha\in\nacc(x)$ and there exists $\varepsilon<\alpha$ such that $B_\varepsilon=A_i$.
As $\rng(\varphi_x)\s\rho\times\lambda$ and $f\restriction\lambda$ is a bijection from $\lambda$ to $\rho\times\lambda$, $j:=f^{-1}(\varphi_x(\varepsilon))$ is an element of $\lambda$.
By the choice of $g$, then, we may pick a large enough $k\in(\otp(x\cap\alpha),\lambda)$ such that $g``(k,k+\sigma+2)=\{j\}$.
Let $\beta\in x$ be such that $\otp(x\cap\beta)=k$.
Then $\suc_\sigma (\Phi^\rho(x) \setminus g_{x,\mathfrak Z}(\beta)))\s A_i$.
\end{proof}
This completes the proof.
\end{proof}

The next lemma provides a tool for transforming a witness to $\p^-(\kappa,2,\mathcal R,1,\allowbreak\{E^\kappa_\theta\},\ldots)$ into a witness to $\p^-(\kappa,2,\mathcal R,\theta,\ldots)$.

\begin{lemma}\label{Phitheta} Suppose that $\theta<\kappa$ are regular, infinite cardinals and $\diamondsuit(\kappa)$ holds.
Then there exists an $\acc$-preserving postprocessing function $\Phi:\mathcal K(\kappa)\rightarrow\mathcal K(\kappa)$ satisfying the following:
For every sequence $\langle A_i\mid i<\theta\rangle$ of cofinal subsets of $\kappa$,
there exists some stationary subset $G\s \kappa$ such that, for all $x\in\mathcal K(\kappa)$,
 if $\sup(\nacc(x)\cap G)=\sup(x)$ and $\cf(\sup(x))=\theta$,
then $\sup(\nacc(\Phi(x))\cap A_i)=\sup(x)$ for all $i<\theta$.
\end{lemma}
\begin{proof} We follow the proof of Lemma~\ref{Phi} as much as possible.
Fix a $\diamondsuit(\kappa)$-sequence, $\langle S_\beta\mid\beta<\kappa\rangle$.
Fix a bijection $\pi:\kappa\times\kappa\leftrightarrow\kappa$,
and let $E:=\{\alpha<\kappa\mid \pi[\alpha\times\alpha]=\alpha\}$.

Let $x\in \mathcal K(\kappa)$ be arbitrary. Put
$$N_x := \left\{ \beta\in\nacc(x)\cap E \mid  \text{for all } \varepsilon,\gamma<\beta,
\text{ there exists } \tau\in\beta\setminus\gamma \text{ with } \pi(\varepsilon,\tau)\in S_\beta\right\}.$$

Define $h_x:\nacc(x)\rightarrow\theta$ by letting for all $\gamma\in x$:
$$h_x(\gamma):=
\otp(\{\beta\in N_x\cap\gamma\mid S_\beta=S_\gamma\cap\beta\})\pmod\theta.$$

Define $\mathfrak Z=\langle Z_{x,\beta}\mid x\in\mathcal K(\kappa), \beta\in\nacc(x)\rangle$ by stipulating:
$$Z_{x,\beta}:=\{\tau<\beta\mid  \pi(h_x(\beta),\tau)\in S_\beta\}.$$

Let $\Phi_{\mathfrak Z}$ be the corresponding $\acc$-preserving postprocessing function given by Example~\ref{phiZ}.

Next, suppose that $\langle A_i\mid i<\theta\rangle$ is a sequence of cofinal subsets of $\kappa$.
Let $\langle B_\varepsilon\mid \varepsilon<\kappa\rangle$ be a sequence of cofinal subsets of $\kappa$
such that for all $i<\theta$, $\{\varepsilon<\theta\mid B_\varepsilon=A_i\}$ is cofinal in $\theta$.
Consider the club $D:=E\cap\bigdiagonal_{\varepsilon<\kappa}(\acc^+(B_\varepsilon))$, the set $S:=\{ \pi(\varepsilon,\tau)\mid \varepsilon<\kappa, \tau\in B_\varepsilon\}$,
and the stationary set $G:=\{\beta\in D\mid S\cap\beta= S_\beta\}$.

\begin{claim}\label{c271} Suppose that $x\in\mathcal K(\kappa)$, $\sup(\nacc(x)\cap G)=\sup(x)$, and $\cf(\sup(x))=\theta$. Then:
\begin{enumerate}
\item $\nacc(x)\cap G\cap\gamma=\{\beta\in N_x\cap\gamma\mid S_\gamma\cap\beta = S_\beta\}$ for every $\gamma\in G$;
\item $\sup(\nacc(\Phi_{\mathfrak Z}(x))\cap A_i)=\sup(x)$ for all $i<\theta$.
\end{enumerate}
\end{claim}
\begin{proof}  (1)  By the proof of Claim~\ref{c122}.

(2)  Denote $\delta:=\sup(x)$. Let $i<\theta$ and $\alpha<\delta$ be arbitrary.
As $\nacc(\Phi_{\mathfrak Z}(x))=g_{x,\mathfrak Z}[\nacc(x)]$,
we shall want to find $\beta\in\nacc(x)$ such that $g_{x,\mathfrak Z}(\beta)\in A_i\setminus\alpha$.

Let $M_x:=\nacc(x)\cap G$. By Clause~(1), for all $\gamma\in M_x$, $h_x(\gamma) = \otp(M_x \cap \gamma) \pmod \theta$.
By $\sup(M_x)=\delta$ and $\cf(\delta)=\theta$, we know that $h_x[M_x\setminus(\alpha+1)]$ is co-bounded in $\theta$.
Pick $\varepsilon\in h_x[M_x\setminus(\alpha+1)]$ such that $B_\varepsilon=A_i$.
Pick $\beta\in M_x\setminus(\alpha+1)$ such that $h_x(\beta)=\varepsilon$.
By $\beta\in E$, we have $$Z_{x,\beta}=\{\tau<\beta\mid  \pi(\varepsilon,\tau)\in S\cap\beta\}=B_\varepsilon\cap \beta.$$

By $\alpha\in\nacc(x)\cap\beta$ and $\beta\in G\s \acc^+(B_\varepsilon)$, we have
$$g_{x,\mathfrak Z}(\beta)\in B_\varepsilon\setminus(\sup(x\cap\beta)+1)\s A_i\setminus(\alpha+1),$$
as sought.
\end{proof}

Therefore the postprocessing function $\Phi_{\mathfrak Z}$ satisfies the needed requirements.
\end{proof}

\subsection{Combinatorial constructions}\label{subsection32}
In Theorem~\ref{improved_square_thm_1} below, we give a list of sufficient conditions for $\p^-(\kappa,2,{\sq},\kappa,\{\kappa\},2,2)$ to hold.
Later on, in Corollary~\ref{incompactness_cor_1}, we prove that $\p^-(\kappa,2,{\sq},\kappa,\{\kappa\},2,2)$ entails
the existence of a ($C$-sequence) graph of size $\kappa$ which is $(\aleph_0,\kappa)$-chromatic.

The idea is to use the postprocessing functions from the preceding subsection to turn simple instances of the proxy principle into more substantial ones.
The simplest instance of the proxy principle, being $\p^-(\kappa,2,{\sq},1,\{S\},2,1)$, is denoted by $\boxtimes^-(S)$:
\begin{defn}[\cite{paper22}] For a stationary subset $S\s\kappa$, $\boxtimes^-(S)$ asserts the existence of a $\sq$-coherent $C$-sequence,
$\langle C_\alpha\mid\alpha<\kappa\rangle$, such that for every cofinal subset $A\s\kappa$, there exists some $\alpha\in S$ for which $\sup(\nacc(C_\alpha)\cap A)=\alpha$.
\end{defn}

\begin{thm} \label{improved_square_thm_1} Suppose that $\kappa$ is a regular cardinal $\ge\aleph_2$, satisfying at least one of the following:
\begin{enumerate}
\item $\diamondsuit(\kappa)+\boxtimes^-(\reg(\kappa))$;
\item $\diamondsuit(\kappa)+\boxtimes^-(T)+\refl(T)$ for some stationary $T\s\kappa$;
\item $\diamondsuit(\kappa)+\boxtimes^-(E^\kappa_\lambda)$ and $\kappa=\lambda^+$;
\item $\diamondsuit(\kappa)+\boxtimes^-(\kappa)+\refl(E^\kappa_{<\lambda})$ and $\kappa=\lambda^+$;
\item $\ch_\lambda+\square(\kappa)+\refl(E^\kappa_{<\beth_\omega})$ and $\kappa = \lambda^+ >\beth_{\omega}$;
\item $\gch+\square(\kappa)+\refl(E^\kappa_{\omega})$ and $\kappa$ is a successor cardinal.
\end{enumerate}

Then $\p(\kappa,2,{\sq},\kappa,\{\kappa\},2,2)$ holds.
\end{thm}
\begin{proof} First, let us simplify some things:
\begin{itemize}
\item By \cite{paper24}, $(5)\vee(6)\implies(2)$.
\item By the exact same proof as that of \cite[Lemma 4.12]{paper24},
if $\diamondsuit(\kappa)+\boxtimes^-(\kappa)$ holds,
then for every partition $\kappa=T_0\uplus T_1$,
there exists some $i<2$ such that $\boxtimes^-(T_i)$ holds. In particular, by taking $T_0=E^\kappa_{<\lambda}$ and $T_1=E^\kappa_{\lambda}$,
we see that $(4)\implies(2)\vee(3)$.
\item By \cite{paper23}, $\diamondsuit(\kappa)+\p^-(\kappa,2,{\sq},\kappa,\{\kappa\},2,1)$ is equivalent to $\p(\kappa,2,{\sq},\kappa,\{\kappa\},\allowbreak2,n)$ for every positive integer $n$.
\end{itemize}
Altogether, it suffices to prove that $(1)\vee(2)\vee(3)\implies\p^-(\kappa,2,{\sq},\kappa,\{\kappa\},2,1)$.

(1) Let $\vec C=\langle C_\alpha\mid\alpha<\kappa\rangle$ be a witness to $\boxtimes^-(\reg(\kappa))$.
Let $\Phi^\omega$ be given by Lemma~\ref{Phi}. Denote $C_\alpha^\bullet:=\Phi^\omega(C_\alpha)$.
To see that $\langle C_\alpha^\bullet\mid\alpha<\kappa\rangle$ witnesses $\p^-(\kappa,2,{\sq},\kappa,\{\kappa\},2,1)$,
let $\vec A=\langle A_i\mid i<\kappa\rangle$ be an arbitrary sequence of cofinal subsets of $\kappa$.
Let $G$ be the stationary set given by Lemma~\ref{Phi} for $\vec A$.

\begin{claim} $S:=\{\alpha<\kappa\mid \otp(\nacc(C_\alpha)\cap G)=\alpha\}$ is stationary.
\end{claim}
\begin{proof} As $\vec C$ witnesses $\boxtimes^-(\reg(\kappa))$,
we know that $T:=\{\alpha\in\reg(\kappa)\mid \sup(\nacc(C_\alpha)\cap G)=\alpha\}$ is stationary.
For all $\alpha\in T$, we have $\alpha\ge\otp(\nacc(C_\alpha)\cap G)\ge\cf(\alpha)=\alpha$ and hence $\alpha\in S$.
\end{proof}

Let $\alpha\in S\setminus(\omega+1)$ be arbitrary. Put $x:=C_\alpha$ and $\rho:=\omega$.
Then $\otp(\nacc(x)\cap G)=\sup(x)>\rho$, and hence $\sup(\nacc(\Phi^\rho(x))\cap A_i)=\alpha$ for all $i<\sup(x)$.
That is, $\{\alpha<\kappa\mid \forall i<\alpha[\sup(\nacc(C_\alpha^\bullet)\cap A_i)=\alpha]\}$ covers the stationary set $S\setminus(\omega+1)$.

(2) Let $\vec C=\langle C_\alpha\mid\alpha<\kappa\rangle$ be a witness to $\boxtimes^-(T)$. As made clear by the proof of the previous clause,
it suffices to prove the following.
\begin{claim} $\{\alpha<\kappa\mid \otp(\nacc(C_\alpha)\cap G)=\alpha\}$ is stationary for every stationary $G\s\kappa$.
\end{claim}
\begin{proof} Suppose that $G$ is a counterexample.
As $\vec C$ witnesses $\boxtimes^-(T)$, we altogether infer the existence of some $\varepsilon<\kappa$ such that the following set is stationary:
$$T':=\{\alpha\in T\mid \sup(\nacc(C_\alpha)\cap G)=\alpha\ \&\ \otp(\nacc(C_\alpha)\cap G)=\varepsilon\}.$$

By $\refl(T)$, pick $\delta\in E^\kappa_{>\omega}$ such that $T'\cap\delta$ is stationary. Fix $\alpha<\beta$ both from $T'\cap\acc(C_\delta)$.
Then $\nacc(C_\beta)\cap G$ is a proper end-extension of $\nacc(C_\alpha)\cap G$,
contradicting the fact that $\otp(\nacc(C_\beta)\cap G)=\varepsilon=\otp(\nacc(C_\alpha)\cap G)$.
\end{proof}

(3)  Let $\vec C=\langle C_\alpha\mid\alpha<\kappa\rangle$ be a witness to $\boxtimes^-(E^\kappa_{\lambda})$.
By $\diamondsuit(\kappa)$, fix a matrix $\langle S_\gamma^\rho\mid \rho<\kappa,\gamma<\kappa\rangle$ such that for every sequence $\langle S^\xi\mid \xi<\kappa\rangle$
of subsets of $\kappa$, the set $\{\gamma<\kappa\mid \forall \rho<\gamma(S^\rho_\gamma=S^\rho\cap\gamma)\}$ is stationary.
In particular, for every cofinal $S\s\kappa$ and $\rho<\kappa$, $G^\rho(S):=\{\gamma<\kappa\mid S_\gamma^\rho=S\cap\gamma\ \&\ \sup(S^\rho_\gamma)=\gamma\}$ is stationary. We distinguish two cases:

$\br$ Suppose that, for every cofinal $S\s\kappa$, the set $\{\alpha<\kappa\mid \otp(\nacc(C_\alpha)\cap G^0(S))=\alpha\}$ is stationary.
Define $\mathfrak Z=\langle Z_{x,\beta}\mid x\in\mathcal K(\kappa), \beta\in\nacc(x)\rangle$ by stipulating $Z_{x,\beta}:=S_\beta^0$.
Let $\Phi_{\mathfrak Z}$ be the corresponding postprocessing function given by Example~\ref{phiZ}.

As made clear by the proof of Clause~(1), it now suffices to prove the following.

\begin{claim}\label{213} $\{\alpha<\kappa\mid \otp(\nacc(\Phi_{\mathfrak Z}(C_\alpha))\cap G)=\alpha\}$ is stationary for every stationary $G\s\kappa$.
\end{claim}
\begin{proof} Let $G$ be an arbitrary stationary subset of $\kappa$.
In particular, $G$ is cofinal in $\kappa$,
so that $T:=\{\alpha<\kappa\mid \otp(\nacc(C_\alpha)\cap G^0(G))=\alpha\}$ is stationary.
Let $\alpha\in T$ and $\beta\in\nacc(C_\alpha)\cap G^0(G)$ be arbitrary.
Then $Z_{C_\alpha,\beta}=S^0_\beta=G\cap\beta$ and $\sup(Z_{C_\alpha,\beta})=\beta$,
so that $g_{C_\alpha,\mathfrak Z}(\beta)\in G\cap\beta$.
Consequently, $\alpha\ge\otp(\nacc(\Phi_{\mathfrak Z}(C_\alpha))\cap G)\ge\otp(\nacc(C_\alpha)\cap G^0(G)))=\alpha$.

Thus, we have established that $\{\alpha<\kappa\mid \otp(\nacc(\Phi_{\mathfrak Z}(C_\alpha))\cap G)=\alpha\}$ covers the stationary set $T$.
\end{proof}

$\br$ Suppose that there exists some cofinal $S^0\s\kappa$
and a club $E\s\kappa$ such that
$E\s \{\alpha<\kappa\mid \otp(\nacc(C_\alpha)\cap G^0(S^0))<\alpha\}$.
Let $\Phi_B$  be the postprocessing function given by Lemma~\ref{PhiB} for $B:=G^0(S^0)$.
Denote $C_\alpha^\circ:=\Phi_B(C_\alpha)$.

\begin{claim}\label{214} For some nonzero $\rho<\kappa$, for every cofinal $S\s\kappa$, the following set is stationary:
$$\{ \alpha\in E^\kappa_\lambda\mid \otp(C_\alpha^\circ)=\rho\ \&\ \sup(\nacc(C_\alpha^\circ)\cap G^\rho(S))=\alpha\}.$$
\end{claim}
\begin{proof} Suppose not. For each nonzero $\rho<\kappa$, pick a counterexample $S^\rho$.
As $T:=\{\gamma\in\bigdiagonal_{\rho<\kappa}\acc^+(S^\rho)\mid \forall\rho<\gamma(S^\rho_\gamma=S^\rho\cap\gamma)\}$ is stationary
and $\vec C$ witnesses $\boxtimes^-(E^\kappa_\lambda)$,
the set $R:=\{\alpha\in E^\kappa_\lambda\cap E\mid \sup(\nacc(C_\alpha)\cap T)=\alpha\}$ is stationary.

Let $\alpha\in R$ be arbitrary.
By $T\setminus 1\s G^0(S^0)=B$, we have $C_\alpha^\circ=\cl(\nacc(C_\alpha)\cap B)$.
Put $\rho_\alpha:=\otp(C_\alpha^\circ)$. By $\alpha\in E$, we have $\rho_\alpha<\alpha$.
But then, by $T\setminus(\rho_\alpha+1)\s B\cap G^{\rho_\alpha}(S^{\rho_\alpha})$, we have $\sup(\nacc(C_\alpha^\circ)\cap G^{\rho_\alpha}(S^{\rho_\alpha}))=\alpha$.
We can now fix a stationary $R' \subseteq R$ and $\rho < \kappa$ such that, for all $\alpha \in R'$,
$\rho_\alpha = \rho$. But then $\{ \alpha\in E^\kappa_\lambda\mid \otp(C_\alpha^\circ)=\rho\ \&\ \sup(\nacc(C_\alpha^\circ)\cap G^\rho(S^\rho))=\alpha\}$
covers the stationary set $R'$, contradicting the choice of $S^\rho$.
\end{proof}

Let $\rho$ be given by the preceding. Clearly, $\rho$ is a limit ordinal.
Define $\mathfrak Z=\langle Z_{x,\beta}\mid x\in\mathcal K(\kappa), \beta\in\nacc(x)\rangle$ by stipulating $Z_{x,\beta}:=S_\beta^\rho$,
and let $\Phi_{\mathfrak Z}$ be the corresponding postprocessing function given by Example~\ref{phiZ}.
Put $T:=\{ \alpha\in E^\kappa_\lambda\mid \otp(\Phi_{\mathfrak Z}(C_\alpha^\circ))\le\rho\}$.

\begin{claim} $\{\alpha\in T\mid \sup(\nacc(\Phi_{\mathfrak Z}(C^\circ_\alpha))\cap G)=\alpha\}$ is stationary for every stationary $G\s\kappa$.
\end{claim}
\begin{proof} Let $G$ be an arbitrary stationary subset of $\kappa$.
By Claim~\ref{214}, $T':=\{\alpha\in T\mid \sup(\nacc(C^\circ_\alpha)\cap G^\rho(G))=\alpha\}$ is stationary.
Let $\alpha\in T'$ and $\beta\in\nacc(C^\circ_\alpha)\cap G^\rho(G)$ be arbitrary.
Then $Z_{C_\alpha,\beta}=S^\rho_\beta=G\cap\beta$ and $\sup(Z_{C_\alpha,\beta})=\beta$,
so that $g_{C_\alpha,\mathfrak Z}(\beta)\in G\cap\beta$.
Consequently, $\sup(\nacc(\Phi_{\mathfrak Z}(C^\circ_\alpha))\cap G)=\alpha$.
\end{proof}

Let $\Phi^\rho$ be given by Lemma~\ref{Phi}. Denote $C_\alpha^\bullet:=\Phi^\rho(\Phi_{\mathfrak Z}(C_\alpha^\circ))$.
To see that $\langle C_\alpha^\bullet\mid\alpha<\kappa\rangle$ witnesses $\p^-(\kappa,2,{\sq},\kappa,\{\kappa\},2,1)$,
let $\vec A=\langle A_i\mid i<\kappa\rangle$ be an arbitrary sequence of cofinal subsets of $\kappa$.
Let $G$ be the stationary set given by Lemma~\ref{Phi} for $\vec A$.

Fix an arbitrary $\alpha\in T$ such that $\sup(\nacc(\Phi_{\mathfrak Z}(C^\circ_\alpha))\cap G)=\alpha$.
Write $x:=\Phi_{\mathfrak Z}(C^\circ_\alpha)$.
Then $\sup(\nacc(x)\cap G)=\sup(x)$, $\otp(x)\le\rho$, and $(\cf(\sup(x)))^+=\kappa$,
and hence the choice of $\Phi^\rho$ entails that  $\sup(\nacc(\Phi^\rho(x))\cap A_i)=\sup(x)$ for all $i<\sup(x)$.
That is, $\sup(\nacc(C_\alpha^\bullet)\cap A_i)=\alpha$ for all $i<\alpha$.
\end{proof}

The purpose of the next theorem is to make a connection between the proxy principle and the concept of \emph{capturing} from Definition~\ref{capturing}.
We remind the reader that the definition of the binary relations $\sqx$ and $\sq_\chi$ may be found in the Notation subsection of the paper's Introduction.

\begin{thm}\label{thm37} Suppose that $\chi<\kappa$ are infinite regular cardinals, and $\theta<\kappa$ is nonzero.
\begin{enumerate}
\item If $\p^-(\kappa,2,{\sqx},\kappa,\{E^\kappa_{\ge\chi}\},2,2)$ holds,
then there exists a $\sqx$-coherent $C$-sequence over $\kappa$
such that $S(\vec A)=\{\delta<\kappa\mid\delta\text{ captures }\vec A\}$ is stationary for every sequence $\vec A=\langle A_i\mid i<\kappa\rangle$ of cofinal subsets of $\kappa$.
\item If $\p^-(\kappa,2,{\sq_\chi},\kappa,\{\kappa\},2,2)$ holds,
then there exists a $\sq_\chi$-coherent $C$-sequence over $\kappa$
such that $S(\vec A)=\{\delta<\kappa\mid\delta\text{ captures }\vec A\}$ is stationary for every sequence $\vec A=\langle A_i\mid i<\kappa\rangle$ of cofinal subsets of $\kappa$.
\item If $\p(\kappa,2,{\sqx},\theta,\{E^\kappa_{\ge\chi}\},2,1)$ holds,
then there exists a $\sqx$-coherent $C$-sequence over $\kappa$
such that $S(\vec A)=\{\delta<\kappa\mid\delta\text{ captures }\vec A\}$ is stationary for every sequence $\vec A=\langle A_i\mid i<\theta\rangle$ of cofinal subsets of $\kappa$.
\item If  $\p(\kappa,2,{\sq_\chi},\theta,\{\kappa\},2,1)$ holds,
then there exists a $\sq_\chi$-coherent $C$-sequence over $\kappa$
such that $S(\vec A)=\{\delta<\kappa\mid\delta\text{ captures }\vec A\}$ is stationary for every sequence $\vec A=\langle A_i\mid i<\theta\rangle$ of cofinal subsets of $\kappa$.
\end{enumerate}
\end{thm}
\begin{proof} {\bf (1)} Let $\vec{C}=\langle C_\delta\mid\delta<\kappa\rangle$  be a witness to $\p^-(\kappa,2,{\sqx},\kappa,\{E^\kappa_{\ge\chi}\},2,2)$.
Denote $C_\delta^\xi:=\Phi_\xi(C_\delta)$, where $\Phi_\xi$ is the postprocessing function from Example~\ref{fact13}.
We claim that there exists some $\xi<\kappa$ such that $\langle C_\delta^\xi\mid\delta<\kappa\rangle$ is as sought.

Suppose not. Then for each $\xi<\kappa$, let us fix a sequence $\langle A^\xi_i\mid i<\kappa\rangle$ of cofinal subsets of $\kappa$
such that, for club many $\delta<\kappa$, at least one of the following two conditions fails:
\begin{itemize}
\item $\min(C_\delta^\xi)\ge\min(A^\xi_0)$;
\item for all $i<\delta$, there exists $\iota\in\otp(C^\xi_\delta)$  such that $C_\delta^\xi(\iota),C^\xi_\delta(\iota+1)\in A^\xi_i$.
\end{itemize}

Fix a bijection $\pi:\kappa\leftrightarrow\kappa\times\kappa$.
For each $j<\kappa$, put $A_j:=A_i^\xi$ iff $\pi(j):=(i,\xi)$.
As $\vec{C}$ witnesses $\p^-(\kappa,2,{\sq},\kappa,\{E^\kappa_{\geq\chi}\},2,2)$, the following set is stationary:
$$S:=\{\delta\in E^\kappa_{\ge\chi}\mid \pi[\delta]=\delta\times\delta\ \&\ \forall j<\delta[\sup\{\gamma\in C_\delta\mid \suc_2(C_\delta\setminus\gamma)\s A_j\}=\delta]\}.$$

Let $\xi<\kappa$ and $\delta\in S\setminus(\xi+1)$ be arbitrary.
For each $i<\delta$ there exists some $j<\delta$ such that $\pi(j)=(i,\xi)$ and some $\gamma\in C^\xi_\delta$ such that $\suc_2(C^\xi_\delta\setminus\gamma)\s A_j$.
Thus, we have established that, for every $\xi<\kappa$,
$$\{\delta<\kappa\mid
\forall i<\delta\exists\iota\in\otp(C^\xi_\delta)[C_\delta^\xi(\iota),C_\delta^\xi(\iota+1)\in A^\xi_i]\}$$ covers the stationary set $S\setminus(\xi+1)$.
So this must mean that, for some club $E_\xi\s\kappa$,
we have $\min(C_\delta^\xi)<\min(A_0^\xi)$ for every $\delta\in S\cap E_\xi$.

Let $E:=\bigdiagonal_{\xi<\kappa}E_\xi$. Following the proof of \cite[Claim~3.2.1]{paper18},
we consider the club $D:=\{\delta\in E\mid \forall\xi<\delta[\min(A_0^\xi)<\delta]\}$,
the set
$S':=\{\beta\in S\mid  \otp(C_\beta)=\beta\}$, and the set
$B:=\{ \beta\in\acc(D)\cap S'\mid \sup((D\cap\beta)\setminus C_\beta)=\beta\}$.
There are three cases to consider, each of which leads to a contradiction.

$\br$ If $B\neq\emptyset$, then let us pick $\beta\in B$ and $\alpha\in(D\cap\beta)\setminus C_\beta$.
For all $\xi<\alpha$, by $\beta\in S'\cap E_\xi$, we have $C_\beta(\xi)=\min(C_\beta^\xi)<\min(A_0^\xi)<\alpha$.
Since the map $\xi\mapsto C_\beta(\xi)$ is increasing and continuous, we then get that $C_\beta(\alpha)=\alpha$,
contradicting the fact that $\alpha\notin C_\beta$.

$\br$ If $S'$ is non-stationary, then, by Fodor's lemma,
there exists some $\varepsilon<\kappa$ such that $T_\varepsilon=\{\beta\in S\cap D\mid \otp(C_\beta)=\varepsilon\}$ is stationary.
Let  $\zeta:=\sup_{\xi<\varepsilon}\min(A_0^\xi)$.
Pick $\beta\in T_\varepsilon$ above $\zeta$.
Then $C_\beta(\xi)=\min(C_\beta^\xi)<\min(A_0^\xi)\le\zeta$ for all $\xi<\varepsilon$.
Consequently, $\beta=\sup(C_\beta)\le\zeta$, contradicting the fact that $\beta>\zeta$.

$\br$ If $B=\emptyset$ and $S'$ is stationary, then let us fix some $\varepsilon<\kappa$ such that $S_\varepsilon:=\{\beta\in S'\mid \sup((D\cap\beta)\setminus C_\beta)=\varepsilon\}$ is stationary.
For every pair of ordinals $\alpha<\beta$ both in $\acc(D\setminus\varepsilon)\cap S_\varepsilon$, we have $\alpha\in\acc(C_\beta)\cap E^\kappa_{\ge\chi}$ and hence $C_\alpha\sqsubseteq C_\beta$.
So $\{ C_\delta\mid\delta\in\acc(D\setminus\varepsilon)\cap S_\varepsilon\}$ is a $\sqsubseteq$-chain, converging to the club $C:=\bigcup\{ C_\delta\mid \delta\in\acc(D\setminus\varepsilon)\cap S_\varepsilon\}$.
Put $A:=\acc(C)$. As $\vec C$ witnesses $\p^-(\kappa,2,{\sqx},\kappa,\{E^\kappa_{\ge\chi}\},2,2)$,
we may pick some $\beta\in \acc(C)\cap E^\kappa_{\ge\chi}$ such that $\sup(\nacc(C_\beta)\cap A)=\beta$,
so that  $\nacc(C_\beta)\cap\acc(C\cap\beta)\neq\emptyset$ and  hence $C_\beta\neq C\cap\beta$.
On the other hand, by definition of $C$, we have  $\beta\in \acc(C_\delta)\cap E^\kappa_{\ge\chi}$ for some $\delta\in\acc(D\setminus\varepsilon)\cap S_\varepsilon$,
and then $C\cap\beta=(C\cap\delta)\cap\beta=C_\delta\cap\beta=C_\beta$. This is a contradiction.

{\bf (2)} Let $\vec{C}=\langle C_\delta\mid\delta<\kappa\rangle$  be a witness to $\p^-(\kappa,2,{\sq_\chi},\kappa,\{\kappa\},2,2)$.
Let $\Gamma:=\{\delta\in\acc(\kappa)\mid \forall\gamma\in\acc(C_\delta)[C_\delta\cap\gamma=C_\gamma]\}$ denote the so-called \emph{support} of $\vec C$ (cf.~\cite{paper29}).
For each $\xi<\kappa$, let $\Phi_\xi$ be the postprocessing function from Example~\ref{fact13}, and then put:
$$C_\delta^\xi:=\begin{cases}\Phi_\xi(C_\delta)&\text{if }\delta\in\Gamma;\\
C_\delta&\text{otherwise}.\end{cases}$$
It is not hard to see that $\langle C_\delta^\xi\mid \delta<\kappa\rangle$ is $\sq_\chi$-coherent.
As in the previous case, we claim that there exists some $\xi<\kappa$ such that $\langle C_\delta^\xi\mid\delta<\kappa\rangle$ is as sought.
The verification is nearly identical, and differs in a single point, as follows.
In the above proof, we identified a stationary subset $S$ of $E^\kappa_{\ge\chi}$ and
a club $D\s\kappa$ and derived a contradiction by inspecting the sets:
\begin{itemize}
\item $S':=\{\beta\in S\mid  \otp(C_\beta)=\beta\}$ and
\item $B:=\{ \beta\in\acc(D)\cap S'\mid \sup((D\cap\beta)\setminus C_\beta)=\beta\}$.
\end{itemize}

This time, $S$ will be the following subset of $\kappa$:
$$S:=\{\delta<\kappa\mid \pi[\delta]=\delta\times\delta\ \&\ \forall j<\delta[\sup\{\gamma\in C_\delta\mid \suc_2(C_\delta\setminus\gamma)\s A_j\}=\delta]\}.$$

Note, however, that by throwing one more set into the collection $\{A_j\mid j<\kappa\}$, we can arrange that $A_0=\acc(\kappa)$.
Consequently, for every nonzero $\delta\in S$, we have $\nacc(C_\delta)\cap\acc(\kappa)\neq\emptyset$, so that $\nacc(C_\delta)$ does not consist only of successor ordinals.
So, by $\sq_\chi$-coherence of $\vec C$, we infer that $S\s\Gamma$, which ensures that $C_\delta^\xi=\Phi_\xi(C_\delta)$ for all $\delta\in S$, exactly as in Clause~(1).

Next, looking at the three cases from the proof of Clause~(1), we see that the argument for the case ``$B=\emptyset$ and $S'$ is stationary'' is the only one to utilize the fact that $S\s E^\kappa_{\ge\chi}$.
Let us show that ``$S\s\Gamma$'' is a satisfying replacement.

By $B=\emptyset$, let us fix some $\varepsilon<\kappa$ such that $S_\varepsilon:=\{\beta\in S'\mid \sup((D\cap\beta)\setminus C_\beta)=\varepsilon\}$ is stationary.
For every pair of ordinals $\alpha<\beta$, both in $\acc(D\setminus\varepsilon)\cap S_\varepsilon$, we have $\alpha\in\acc(C_\beta)$ and $\beta\in\Gamma$,
and hence $C_\alpha\sqsubseteq C_\beta$.
So $\{ C_\delta\mid\delta\in\acc(D\setminus\gamma)\cap S_\varepsilon\}$ is a $\sqsubseteq$-chain, converging to the club $C:=\bigcup\{ C_\delta\mid \delta\in\acc(D\setminus\gamma)\cap S_\varepsilon\}$.
Put $A:=\acc(C)$. As $\vec C$ witnesses $\p^-(\kappa,2,{\sq_\chi},\kappa,\{\kappa\},2,2)$,
we may pick some $\beta\in \acc(C)$ such that $\sup(\nacc(C_\beta)\cap A)=\beta$,
so that $C_\beta\neq C\cap\beta$.
On the other hand, by definition of $C$, we have  $\beta\in \acc(C_\delta)$ for some $\delta\in\acc(D\setminus\varepsilon)\cap S_\varepsilon$,
and then, by $\delta\in\Gamma$, we have $C\cap\beta=(C\cap\delta)\cap\beta=C_\delta\cap\beta=C_\beta$. This is a contradiction.

{\bf (3)} By \cite{paper23}, $\diamondsuit(\kappa)+\p^-(\kappa,2,{\sqx},\theta,\{E^\kappa_{\ge\chi}\},2,1)$ is equivalent to $\p(\kappa,2,{\sqx},\theta,\allowbreak\{E^\kappa_{\ge\chi}\},2,n)$ for every positive integer $n$,
so let $\vec{C}=\langle C_\delta\mid\delta<\kappa\rangle$  be a witness to $\p^-(\kappa,2,{\sqx},\theta,\allowbreak\{E^\kappa_{\ge\chi}\},2,2)$.
By $\diamondsuit(\kappa)$, fix a matrix $\langle S_\gamma^\rho\mid \rho<\kappa,\gamma<\kappa\rangle$ such that for every sequence $\langle S^\xi\mid \xi<\kappa\rangle$
of subsets of $\kappa$, the set $\{\gamma<\kappa\mid \forall \rho<\gamma(S^\rho_\gamma=S^\rho\cap\gamma)\}$ is stationary.
For every $\xi<\kappa$, define $\mathfrak Z^\xi=\langle Z^\xi_{x,\beta}\mid x\in\mathcal K(\kappa), \beta\in\nacc(x)\rangle$ by stipulating
$Z^\xi_{x,\beta}:=S^\xi_\beta$, and let $\Phi_{\mathfrak Z^\xi}$ be the corresponding postprocessing function given by Example~\ref{phiZ}.
Let $\Phi_\xi$ be the postprocessing function from Example~\ref{fact13}.
Finally,  denote $C_\delta^\xi:=\Phi_{\mathfrak Z^\xi}(\Phi_\xi(C_\delta))$.
We claim that there exists some $\xi<\kappa$ such that $\langle C_\alpha^\xi\mid\alpha<\kappa\rangle$ is as sought.

Suppose not. Then, for each $\xi<\kappa$, let us fix a sequence $\langle A^\xi_i\mid i<\theta\rangle$ of cofinal subsets of $\kappa$
such that, for club many $\delta<\kappa$, at least one of the following fails:
\begin{itemize}
\item $\min(C_\delta^\xi)\ge\min(A^\xi_0)$;
\item for all $i<\theta$, there exists $\iota\in\otp(C^\xi_\delta)$ such that $C_\delta^\xi(\iota),C^\xi_\delta(\iota+1)\in A^\xi_i$.
\end{itemize}

Evidently,  for each $i<\theta$, the following set is stationary in $\kappa$:
$$T_i:=\{\gamma\in\acc(\kappa)\mid \forall\xi<\gamma[S^\xi_\gamma=A^\xi_i\cap\gamma\ \&\ \sup(S^\xi_\gamma)=\gamma]\}.$$

As $\vec{C}$ witnesses $\p^-(\kappa,2,{\sqx},\theta,\{E^\kappa_{\ge\chi}\},2,2)$, the following set is also stationary:
$$S:=\{\delta\in E^\kappa_{\ge\chi}\mid \forall i<\theta[\sup\{\gamma\in C_\delta\mid \suc_2(C_\delta\setminus\gamma)\s T_i\} = \delta]\}.$$

Let $\xi<\kappa$ and $\delta\in S\setminus(\xi+1)$ be arbitrary.
As $\Phi_\xi(C_\delta)$ is a final segment of $C_\delta$, for each $i<\theta$, let us pick $\gamma_i\in \Phi_\xi(C_\delta)$
such that $\suc_2(C_\delta\setminus\gamma_i)\s T_i$.
Set $\alpha_i:=\max(\suc_2(C_\delta\setminus\gamma_i))$ and $\tau_i:=\sup(C_\delta\cap\alpha_i)$.
Then $Z_{\Phi_\xi(C_\delta),\alpha_i}^\xi=S^\xi_{\alpha_i}=A_i^\xi\cap\alpha_i$ and $\sup(S^\xi_{\alpha_i})=\alpha_i$.
Likewise,  $Z_{\Phi_\xi(C_\delta),\tau_i}^\xi=S^\xi_{\tau_i}=A_i^\xi\cap\tau_i$ and $\sup(S^\xi_{\tau_i})=\tau_i$.
Then $g_{\Phi_\xi(C_\delta),\mathfrak Z^\xi}(\tau_i)$ and $g_{\Phi_\xi(C_\delta),\mathfrak Z^\xi}(\alpha_i)$ are two successive elements of  $C^\xi_\delta$ that belong to $A_i^\xi$.

Thus, we have established that, for every $\xi<\kappa$,
$$\{\delta<\kappa\mid
\forall i<\theta\exists\iota\in\otp(C_\delta^\xi)[C_\delta^\xi(\iota),C_\delta^\xi(\iota+1)\in A^\xi_i]\}$$ covers the stationary set $S\setminus(\xi+1)$.
So this must mean that for some club $E_\xi\s\kappa$,
we have  $\min(C_\delta^\xi)<\min(A_0^\xi)$ for every $\delta\in S\cap E_\xi$.
But, as seen in the proof of Clause~(1), this yields a contradiction.

{\bf (4)} By the proof of Clause~(3) with the same adjustment we gave in moving from Clause~(1) to Clause~(2).
\end{proof}

\begin{cor} \label{incompactness_cor_1}  For all infinite regular cardinals $\chi<\kappa$ and every cardinal $\theta<\kappa$:
\begin{enumerate}
  \item\label{c4} $\p(\kappa,2,{\sq_\chi},\theta,\{\kappa\},2,1)$ entails the existence of a $(\chi,>\theta)$-chromatic graph of size $\kappa$;
  \item\label{c2} $\p^-(\kappa,2,{\sq_\chi},\kappa, \{\kappa\},2,2)$ entails the existence of a $(\chi, \kappa)$-chromatic graph of size $\kappa$;
  \item\label{c1} $\p^-(\kappa,2,{\sqx},\kappa,\{E^\kappa_{\ge\chi}\},2,2)$ entails the existence of a $(\chi,\kappa)$-chromatic graph of size $\kappa$;
  \item\label{c3} $\p(\kappa,2,{\sqx},\theta,\{E^\kappa_{\ge\chi}\},2,1)$ entails the existence of a $(\chi,>\theta)$-chromatic graph of size $\kappa$.
\end{enumerate}
\end{cor}
\begin{proof} The results follow from Lemmas \ref{lemma1} and \ref{large_chromatic_number_lemma}, using the appropriate $\vec C$, as follows.

(1) follows from Theorem~\ref{thm37}(4),
(2) from Theorem~\ref{thm37}(2),
(2) from Theorem~\ref{thm37}(1),
and (4) from Theorem~\ref{thm37}(3).
\end{proof}

The proof of Theorem B(1) goes through the following.
\begin{lemma}\label{lemma21} Suppose that $\chi,\theta<\kappa$ are infinite, regular cardinals and $\mathcal R\in\{{\sqx},{\sq_\chi}\}$.

If $\p(\kappa,2,\mathcal R,1,\{E^\kappa_\theta\},2,1)$ holds, then so does $\p(\kappa,2,\mathcal R,\theta,\{E^\kappa_\theta\},2,1)$.
\end{lemma}
\begin{proof}
Let $\vec C=\langle C_\alpha\mid\alpha<\kappa\rangle$ be a witness to $\p^-(\kappa,2,\mathcal R,1,\{E^\kappa_\theta\},2,1)$.
Without loss of generality, we may assume that $C_{\alpha+1}=\{\alpha\}$ for all $\alpha<\kappa$.
Let $\Phi$ be given by Lemma~\ref{Phitheta}. For all $\alpha<\kappa$, put:
$$C_\alpha^\bullet:=\begin{cases}
C_\alpha&\text{if }\mathcal R=\sq_\chi\ \&\ \exists\bar\alpha\in\acc(C_\alpha)[C_{\bar\alpha}\neq C_\alpha\cap\bar\alpha];\\
\Phi(C_\alpha)&\text{otherwise}.
\end{cases}$$
It is not hard to see that $\langle C_\alpha^\bullet\mid\alpha<\kappa\rangle$ witnesses $\p^-(\kappa,2,\mathcal R,\theta,\{E^\kappa_\theta\},2,n)$ for $n=0$.
We claim that this is also the case for $n=1$.

To see this, let $\vec A=\langle A_i\mid i<\theta\rangle$ be a sequence of cofinal subsets of $\kappa$.
Let $G$ be the stationary set given by Lemma~\ref{Phitheta} for $\vec A$.
Then $S:=\{ \alpha\in E^\kappa_\theta\mid \sup(\nacc(C_\alpha)\cap( \acc(\kappa)\cap G))=\alpha\}$ is stationary.
Let $\alpha\in S$ be arbitrary. Then, letting $x:=C_\alpha$,
by $\sup(\nacc(x)\cap G)=\sup(x)$ and $\cf(\sup(x))=\theta$, the choice of $\Phi$ entails
$\sup(\nacc(\Phi(x))\cap A_i)=\sup(x)$ for all $i<\theta$.

$\br$ If $\mathcal R$ is $\sqx$, then, by definition of $C_\alpha^\bullet$, we have $C_\alpha^\bullet=\Phi(C_\alpha)$,
 so that $\sup(\nacc(C_\alpha^\bullet)\cap A_i)=\alpha$ for all $i<\theta$, as sought.

$\br$ If $\mathcal R$ is $\sq_\chi$, then, by $\alpha\in S$, we know that $\nacc(C_\alpha)\cap\acc(\kappa)\neq\emptyset$,
and hence, for every $\bar\alpha\in\acc(C_\alpha)$, by $C_{\bar\alpha}\sq_\chi C_\alpha$, we infer that $C_{\bar\alpha}\sq C_\alpha$,
so that $C_\alpha^\bullet=\Phi(C_\alpha)$.
\end{proof}

We are now ready to prove Theorems B and C:

\begin{cor} \label{incompactness_cor_2} Suppose $\aleph_0\le\cf(\chi)=\chi<\lambda$ are cardinals, and $\gch$ and $\square(\lambda^+,{\sq_\chi})$ both hold.
\begin{enumerate}
\item If $\lambda$ is regular, then there exists a $(\chi,\ge\lambda)$-chromatic graph of size $\lambda^+$;
\item If $\lambda$ is singular, then there exists a $(\chi,\lambda^+)$-chromatic graph of size $\lambda^+$;
\item If $\refl(S)$ holds for some stationary $S\s E^{\lambda^+}_{\ge\chi}$, then there exists a $(\chi,\lambda^+)$-chromatic graph of size $\lambda^+$.
\end{enumerate}
\end{cor}
\begin{proof} (1) By \cite{paper29}, $\gch+\square(\lambda^+,{\sq_\chi})$ entails $\p(\lambda^+,2,{\sq_\chi},1,\{E^{\lambda^+}_\theta\mid \theta\in\reg(\lambda)\},\allowbreak2,1)$.
For each $\theta\in\reg(\lambda)$,
by Lemma~\ref{lemma21}, we infer that $\p(\lambda^+,2,{\sq_\chi},\theta,\{E^{\lambda^+}_\theta\},2,1)$ holds,
and so by Corollary~\ref{incompactness_cor_1}(\ref{c4}), we may fix a $(\chi,>\theta)$-chromatic graph $\mathcal G_\theta$ of size $\lambda^+$.
Let $\mathcal G$ be the disjoint sum of the graphs $\{\mathcal G_\theta\mid \theta\in\reg(\lambda)\}$. Then $\mathcal G$ is $(\chi,\ge\lambda)$-chromatic of size $\lambda^+$.

(2) By \cite{paper29}, for every singular strong limit cardinal $\lambda$ such that $2^\lambda=\lambda^+$, $\square(\lambda^+,{\sq_\chi})$
entails $\p(\lambda^+,2,{\sq_\chi},\lambda^+,\{\lambda^+\},2,2)$.
It now follows from Corollary~\ref{incompactness_cor_1}(\ref{c2}) that there exists a $(\chi,\lambda^+)$-chromatic graph of size $\lambda^+$.

(3) Recalling Clause~(2), we may assume that $\lambda$ is regular. By $\refl(S)$, we know that $S\cap E^{\lambda^+}_{<\lambda}$ is stationary.
As $\lambda$ is regular and $S\cap E^{\lambda^+}_{\ge\chi}\cap E^{\lambda^+}_{<\lambda}$ is stationary,
we get from \cite{paper29} that $\gch+\square(\lambda^+,{\sq_\chi})$ entails $\p(\lambda^+,2,{\sq_\chi},1,\{S\},2,1)$.
Then by the proof of Theorem~\ref{improved_square_thm_1}(2), we obtain
$\p(\lambda^+,2,{\sq_\chi},\lambda^+,\{\lambda^+\},2,2)$.
So, by Corollary~\ref{incompactness_cor_1}(\ref{c2}), there exists a $(\chi,\lambda^+)$-chromatic graph of size $\lambda^+$.
\end{proof}

\subsection{Forcing constructions} \label{forcing_subsection}

In this subsection, we show that, for all infinite regular cardinals $\chi<\kappa$, there exists a forcing poset $\mathbb{P}$
for introducing $\p^-(\kappa,2,{\sq_\chi},\kappa,\allowbreak(\ns^+_\kappa)^V,2,\sigma)$ such that $\mathbb{P}$ is $\chi$-directed closed and $\kappa$-strategically closed
and therefore preserves all cardinalities and cofinalities $\leq \kappa$.
If, additionally, $\kappa^{<\kappa} = \kappa$, then $\mathbb{P}$ has the $\kappa^+$-c.c. and
thus preserves all cardinalities and cofinalities.

\begin{defn}\label{theposet}
  Let $\chi < \kappa$ be infinite, regular cardinals. $\mathbb{P}(\kappa, \chi)$ is the forcing poset consisting
  of all conditions of the form $p=\emptyset$ or $p = \langle C^p_\alpha \mid \alpha \leq \gamma^p \rangle$, where
  $\gamma^p < \kappa$ is a limit ordinal and $\langle C^p_\alpha \mid \alpha \leq \gamma^p \rangle$
  is a $\sq_\chi$-coherent $C$-sequence over $\gamma^p+1$,
satisfying $C^p_{\alpha+1}=\{\alpha\}$ for all $\alpha<\gamma^p$.

 For $p,q \in \mathbb{P}(\kappa, \chi)$, we let $q \leq p$ iff $q \supseteq p$.
\end{defn}

For the rest of this subsection, fix infinite, regular cardinals $\chi < \kappa$ and let $\mathbb{P} := \mathbb{P}(\kappa, \chi)$.

\begin{lemma}
  $\mathbb{P}$ is $\chi$-directed closed.
\end{lemma}

\begin{proof}
  Note that $\mathbb{P}$ is tree-like, i.e., if $p,q,r \in \mathbb{P}$ and $r \leq p,q$, then $p$ and $q$ are
  comparable in $\mathbb{P}$. Therefore, it suffices to show that $\mathbb{P}$ is $\chi$-closed. To this end,
  fix a limit ordinal $\eta < \chi$, and let $\langle p_\xi \mid \xi < \eta \rangle$ be a strictly decreasing
  sequence of conditions in $\mathbb{P}$. Define a condition $q$ extending $\langle p_\xi \mid \xi < \eta \rangle$
  by letting $\gamma := \sup\{\gamma^{p_\xi} \mid \xi < \eta\}$, fixing a club $D$ in $\gamma$ of order type
  $\cf(\eta)$, and letting $q$ be the unique extension of $\bigcup_{\xi < \eta} p_\xi$ such that
  $\gamma^q = \gamma$ and $C^q_\gamma = D$. $q$ is easily verifed to be a lower bound for $\langle p_\xi
  \mid \xi < \eta \rangle$.
\end{proof}

\begin{defn}\label{thegame} A forcing poset $\mathbb P$ is said to be \emph{$\alpha$-strategically} closed if \textrm{II} has a winning strategy for $\Game_\alpha(\mathbb P)$,
which is the following two-player game of perfect information:

The two players, named \textrm{I} and \textrm{II}, respectively, take turns to play conditions from $\mathbb P$ for $\alpha$ many moves, with
\textrm{I} playing at odd stages and \textrm{II} at even stages (including all limit stages).
\textrm{II} must play $\one_{\mathbb P}$ at move zero.
Let $p_\beta$ be the condition played at move $\beta$;
the player who plays $p_\beta$ loses immediately unless $p_\beta\le p_\gamma$ for all $\gamma<\beta$.
If neither player loses at any stage $\beta<\alpha$,  then \textrm{II} wins.
\end{defn}

\begin{lemma} \label{strat_closed_lemma}
  $\mathbb{P}$ is $\kappa$-strategically closed.
\end{lemma}
\begin{proof} The proof is identical to that of  \cite[Proposition 33]{lh_covering}.
\end{proof}

\begin{lemma} \label{extension_lemma}
  Suppose $p \in \mathbb{P}$, $\dot{A}$ is a $\mathbb{P}$-name for a cofinal subset of $\kappa$,
  and $\sigma < \kappa$. Then there is $q \leq p$ such that:
  \begin{enumerate}
    \item $C^p_{\gamma^p} \sq C^q_{\gamma^q}$;
    \item $q \Vdash_{\mathbb{P}}``\suc_\sigma(C^q_{\gamma^q} \setminus \gamma^p) \subseteq \dot{A}."$
  \end{enumerate}
\end{lemma}

\begin{proof}
  By increasing $\sigma$ if necessary, we may assume that $\sigma$ is a limit ordinal. We will recursively construct
  a strictly decreasing sequence $\langle p_\xi \mid \xi \leq \sigma \rangle$ of conditions in $\mathbb{P}$ and a strictly
  increasing sequence of ordinals $\langle \delta_\xi \mid \xi < \sigma \rangle$ such that the following
  will hold, where, for $\xi < \sigma$, we denote $\gamma^{p_\xi}$ by $\gamma_\xi$:
  \begin{itemize}
    \item $p_0 = p$;
    \item for all $\xi < \sigma$, $\gamma_\xi < \delta_\xi < \gamma_{\xi+1}$ and $p_{\xi+1} \Vdash ``\delta_\xi
      \in \dot{A};"$
    \item if $\eta \leq \xi \leq \sigma$ are limit ordinals, then $C^p_{\gamma_0} \sq C^{p_\eta}_{\gamma_\eta}
      \sq C^{p_\xi}_{\gamma_\xi}$ and $\suc_\xi(C^{p_\xi}_{\gamma_\xi} \setminus \gamma_0) = \{\delta_\eta \mid \eta < \xi\}$.
  \end{itemize}
  There are three cases to deal with in the recursion.

  $\br$ If $\xi = \eta + 1 < \sigma$ and $p_\eta$, $\langle \delta_\epsilon \mid \epsilon < \eta \rangle$ have been defined, find $p^* \leq p_\eta$ and $\delta$ such that:
  \begin{itemize}
    \item $\gamma_\eta < \delta < \gamma^{p^*}$;
    \item $p^* \Vdash_{\mathbb{P}}``\delta \in \dot{A}."$
  \end{itemize}
  Let $p_\xi := p^*$ and $\delta_\eta := \delta$.

  $\br$ If $\xi = \eta + \omega \leq \sigma$ and $\langle (p_\epsilon, \delta_\epsilon) \mid \epsilon < \xi \rangle$ has been defined,
  then let $p_\xi$ be the unique condition extending $\langle p_\epsilon \mid \epsilon < \xi \rangle$ such that:
  \begin{itemize}
    \item $\gamma^{p_\xi} = \gamma_\xi = \sup\{\gamma_\epsilon \mid \epsilon < \xi\}$;
    \item $C^{p_{\xi}}_{\gamma_\xi} = C^{p_\eta}_{\gamma_\eta} \cup \{\gamma_\eta\} \cup \{\delta_{\eta + n} \mid n < \omega\}$.
  \end{itemize}
  It is easily verified that $p_\xi$ satisfies all of our requirements.

  $\br$ Finally, if $\xi \leq \sigma$ is a limit of limit ordinals and $\langle (p_\epsilon, \delta_\epsilon)\mid \epsilon < \xi \rangle$
  has been defined, then let $p_\xi$ be the unique condition extending $\langle p_\epsilon \mid \epsilon < \xi \rangle$ such that:
  \begin{itemize}
    \item $\gamma^{p_\xi} = \gamma_\xi = \sup\{\gamma_\epsilon \mid \epsilon < \xi\}$;
    \item $C^{p_\xi}_{\gamma_\xi} = \bigcup_{\eta \in \acc(\xi)} C^{p_\eta}_{\gamma_\eta}$.
  \end{itemize}
  It is easily verified, using our inductive hypotheses, that $p_\xi$ is in $\mathbb{P}$ and satisfies our requirements.

  At the end of the construction, let $q := p_\sigma$. By the requirements satisfied by the construction, $q$ is as desired
  in the statement of the Lemma.
\end{proof}

It is straightforward to show that,
for all $\alpha < \kappa$, the set of $p \in \mathbb{P}$ such that $\alpha \leq \gamma^p$ is a dense,
open subset of $\mathbb{P}$.
Consequently, if $g$ is $(V,\mathbb P)$-generic, then $\vec{C} := \bigcup g = \langle C_\alpha \mid \alpha < \kappa \rangle$ is a $\sq_\chi$-coherent
$C$-sequence over $\kappa$.

\begin{thm}\label{genericity_lemma}
  In $V[g]$, $\vec{C}$ witnesses $\p^{-}(\kappa, 2, {\sq_\chi}, \kappa, (\ns^+_\kappa)^V, 2, \sigma)$ simultaneously for every $\sigma < \kappa$.
\end{thm}

\begin{proof}
  Work in $V$, and fix $\mathbb{P}$-names $\langle \dot{A}_i \mid i < \kappa \rangle$ for cofinal subsets
  of $\kappa$, a $\mathbb{P}$-name $\dot{D}$ for a club in $\kappa$, a stationary set $S \subseteq \kappa$, an ordinal $\sigma < \kappa$, and a
  condition $p \in \mathbb{P}$. We will find $q \leq p$
  and $\beta \in S$ such that $q \Vdash_{\mathbb{P}} ``\beta \in \dot{D}\text{ and, for all }i < \beta,~\sup\{\alpha \in C^q_\beta \mid \suc_\sigma(C^q_\beta
  \setminus \alpha) \subseteq \dot{A}_i\} = \beta."$

  Fix a partition $\langle B_i \mid i < \kappa \rangle$ of $\kappa$ into pairwise disjoint, cofinal subsets. For $\alpha < \kappa$,
  let $i_\alpha$ denote the unique $i < \kappa$ such that $\alpha \in B_i$. Using Lemma \ref{extension_lemma}, it is straightforward
  to build a strictly decreasing sequence $\langle p_\alpha \mid \alpha < \kappa \rangle$ of conditions and a
  strictly increasing sequence $\langle \epsilon_\alpha \mid \alpha < \kappa \rangle$ such that the following
  hold, where, for $\alpha < \kappa$, we denote $\gamma^{p_\alpha}$ by $\gamma_\alpha$:
  \begin{itemize}
    \item $p_0 = p$;
    \item for all $\alpha < \beta < \kappa$, $C^{p_\alpha}_{\gamma_\alpha} \sq C^{p_\beta}_{\gamma_\beta}$;
    \item $\langle \gamma_\alpha \mid \alpha < \kappa \rangle$ is increasing and continuous;
    \item for all $\alpha < \kappa$, $p_{\alpha+1} \Vdash_{\mathbb{P}} ``\suc_\sigma(C^{p_{\alpha+1}}_{\gamma_{\alpha+1}}
      \setminus \gamma_\alpha) \subseteq \dot{A}_{i_\alpha}"$;
    \item for all $\alpha < \kappa$, $\gamma_\alpha < \epsilon_\alpha < \gamma_{\alpha+1}$ and $p_{\alpha+1} \Vdash_{\mathbb{P}}
      ``\epsilon_\alpha \in \dot{D}."$
  \end{itemize}
  Let $E$ be the set of $\alpha$ in $\acc(\kappa)$ such that:
  \begin{itemize}
    \item $\alpha \in \bigcap_{i < \alpha} \acc^+(B_i)$;
    \item $\alpha = \gamma_\alpha = \sup\{\epsilon_\eta \mid \eta < \alpha\}$.
  \end{itemize}
  $E$ is club in $\kappa$, so we can fix $\beta \in E \cap S$. We claim that $q_\beta$ and $\beta$ are as desired. Indeed,
  $q_\beta \Vdash_{\mathbb{P}}``\{\epsilon_\alpha \mid \alpha < \beta\} \subseteq \dot{D}"$ and $\beta =
  \sup\{\epsilon_\alpha \mid \alpha < \beta\}$, so, as $\dot{D}$ is forced to be club, we have
  $q_\beta \Vdash_{\mathbb{P}}``\beta \in \dot{D}."$ Also, if $i < \beta$ and $\eta < \beta$, fix
  $\alpha \in B_i$ such that $\eta < \alpha \leq \gamma_\alpha < \beta$. By construction,
  $q_\beta \Vdash ``\suc_\sigma(C^{q_\beta}_\beta \setminus \gamma_\alpha) \subseteq \dot{A}_i."$
  Therefore, $q_\beta \Vdash_\mathbb{P}``\text{for all }i < \beta,~\sup\{\alpha \in C^{q_\beta}_\beta \mid \suc_\sigma(C^{q_\beta}_\beta
  \setminus \alpha) \subseteq \dot{A}_i\} = \beta."$
\end{proof}

We will sometimes want to do further forcing over $V[g]$ to eliminate certain instances
of incompactness. We describe this forcing here. First, in
$V[g]$, let $\mathbb{T}$ be the forcing to add a thread through $\vec{C}$. More precisely, conditions in
$\mathbb{T}$ are the clubs $C_\alpha$ for $\alpha < \kappa$, and $\mathbb{T}$ is ordered
by end-extension, i.e., for $\alpha < \beta < \kappa$, $C_\beta \leq_{\mathbb T} C_\alpha$ iff
$C_\alpha \sq C_\beta$.

Also in $V[g]$, we define a forcing iteration $\langle \mathbb{Q}_\eta, \dot{\mathbb{R}}_\xi
\mid \eta \leq \kappa^+, \xi < \kappa^+ \rangle$, taken with supports of size $<\kappa$,
so that, for each $\xi < \kappa^+$,
there is a $\mathbb{Q}_\xi$-name $\dot{S}_\xi$ for a subset of $\kappa$ such that:
\begin{itemize}
  \item $\Vdash_{\mathbb{Q}_\xi \times \mathbb{T}}``\dot{S}_\xi$ is non-stationary$;"$
  \item $\Vdash_{\mathbb{Q}_\xi}``\dot{\mathbb{R}}_\xi$ is the forcing to shoot a club
    through $\kappa$, disjoint from $\dot{S}_\xi$, by closed initial segments$."$
\end{itemize}
Let $\mathbb{Q} = \mathbb{Q}_{\kappa^+}$. A straightforward $\Delta$-system argument, together
with the assumption that $2^\kappa = \kappa^+$, yields the fact that $\mathbb{Q}$ has the
$\kappa^+$-c.c. Therefore, by employing an appropriate bookkeeping device, we can choose the
names $\langle \dot{S}_\xi \mid \xi < \kappa^+ \rangle$ in such a way so that, in
$V[g]^\mathbb{Q}$, if $S \subseteq \kappa$ is stationary, then $\not\Vdash_\mathbb{T}``S$
is non-stationary$."$

The following Lemma is proven in Section 3 of \cite{hayut_lh}.

\begin{lemma} \label{directed_closed_lemma}
  In $V$, $\mathbb{P} * (\dot{\mathbb{Q}} \times \dot{\mathbb{T}})$ has a dense $\kappa$-directed closed subset. \qed
\end{lemma}

Let $\mathbb{U}$ be the dense $\kappa$-directed closed subset identified by Lemma \ref{directed_closed_lemma}.
A salient feature of $\mathbb{U}$ is that, for all $(p, \dot{q}, \dot{t}) \in \mathbb{U}$,
$p \Vdash_{\mathbb{P}}``\dot{t} = C^p_{\gamma^p}."$ Let $h$ be $\mathbb{Q}$-generic over $V[g]$.

\begin{thm}\label{iteration_lemma}
  In $V[g*h]$, $\vec{C}$ witnesses $\p^-(\kappa, 2,{\sq_\chi}, \kappa, \ns^+_\kappa, 2, n)$ simultaneously for
  every positive $n < \omega$.
\end{thm}

\begin{proof}
  Suppose not. Then we can find a sequence $\langle A_i \mid i < \kappa \rangle$ of cofinal
  subsets of $\kappa$, a stationary set $S \subseteq \kappa$, and a positive  $n < \omega$
  such that, for all $\beta \in S$, there is $i_\beta < \beta$ such that
  $\sup\{\alpha \in C_\beta \mid \suc_n(C_\beta \setminus \alpha) \subseteq A_i\} < \beta$. By two
  applications of Fodor's Lemma, we may in fact assume that there are fixed $i^*, \alpha^* < \kappa$
  such that, for all $\beta \in S$, $i_\beta = i^*$ and $\sup\{\alpha \in C_\beta \mid
  \suc_n(C_\beta \setminus \alpha) \subseteq A_i\} = \alpha^*$. Fix $(p_0,\dot{q}_0) \in g*h$ and
  $\mathbb{P} * \dot{\mathbb{Q}}$-names $\dot{S}$ and $\dot{A}_{i^*}$ such that that $(p_0, \dot{q}_0)$
  forces the following:
  \begin{itemize}
    \item $\dot{S} \subseteq \kappa$ is stationary;
    \item $\dot{A}_{i^*} \subseteq \kappa$ is cofinal;
    \item for all $\beta \in \dot{S}$, $\sup\{\alpha \in \dot{C}_\beta \mid \suc_n(\dot{C}_\beta \setminus \alpha) \subseteq \dot{A}_{i^*}\} = \alpha^*$.
  \end{itemize}
  Work now in $V$. By our definition of $\mathbb{Q}$,
  we can find $(p_1, \dot{q}_1, \dot{t}_1) \in \mathbb{P} * (\dot{\mathbb{Q}} \times \dot{\mathbb{T}})$
  such that $(p_1, \dot{q}_1) \leq (p_0, \dot{q}_0)$ and $(p_1, \dot{q}_1, \dot{t}_1) \Vdash_{\mathbb{P} * (\dot{\mathbb{Q}} \times \dot{\mathbb{T}})}
  ``\dot{S}$ is stationary$."$ Without loss of generality, $p_1 \Vdash_{\mathbb{P}}``\dot{t}_1 = C^{p_1}_{\gamma^{p_1}}"$ and
  $\gamma^{p_1} > \alpha^*$.

  Find $(p_2, \dot{q}_2) \leq (p_1, \dot{q}_1)$ and ordinals $\{\xi_m \mid m < n\}$ such that:
  \begin{itemize}
    \item $\gamma^{p_1} < \xi_0 < \xi_1 < \ldots < \xi_{n-1} < \gamma^{p_2}$;
    \item $(p_2, \dot{q}_2) \Vdash ``\{\xi_m \mid m < n\} \subseteq \dot{A}_{i^*}."$
  \end{itemize}
  Let $p_3$ be the unique extension of $p_2$ such that $\gamma^{p_3} = \gamma^{p_2} + \omega$ and
  \[C^{p_3}_{\gamma^{p_3}} := C^{p_1}_{\gamma^{p_1}} \cup \{\gamma^{p_1}\} \cup \{\xi_m \mid m < n\}
  \cup \{\gamma^{p_2} + \ell \mid \ell < \omega\}.\]
  Let $\dot{q}_3 = \dot{q}_2$, and let $\dot{t}_3$ be a $\mathbb{P}$-name forced by $p_3$ to be equal
  to $C^{p_3}_{\gamma^{p_3}}$.

  Let $\dot{T}$ be the canonical name for the club in $\kappa$ introduced by $\dot{\mathbb{T}}$.
  Then $$(p_3, \dot{q}_3, \dot{t}_3) \Vdash_{\mathbb{P} * (\dot{\mathbb{Q}} \times \dot{\mathbb{T}})}
  ``\text{for all }\beta \in \dot{T} \setminus \gamma^{p_3},~\sup\{\alpha \in \dot{C}_\beta \mid
  \suc_n(\dot{C}_\beta \setminus \alpha) \subseteq \dot{A}_{i^*}\} \geq \gamma^{p_1} > \alpha^*."$$
  Moreover, $(p_3, \dot{q}_3,\dot{t}_3) \Vdash_{\mathbb{P} * (\dot{\mathbb{Q}} \times \dot{\mathbb{T}})} ``\dot{S}\text{ is
  stationary},"$ so $(p_3, \dot{q}_3, \dot{t}_3) \Vdash_{\mathbb{P} * (\dot{\mathbb{Q}} \times \dot{\mathbb{T}})}
  ``\dot{S} \cap (\dot{T} \setminus \gamma^{p_3}) \neq \emptyset."$ This contradicts the fact that
  $(p_3, \dot{q}_3) \leq (p_0, \dot{q}_0)$ and $(p_0, \dot{q}_0) \Vdash_{\mathbb{P} * \dot{\mathbb{Q}}}
  ``\text{for all }\beta \in \dot{S},~\sup\{\alpha \in \dot{C}_\beta \mid \suc_n(\dot{C}_\beta \setminus \alpha) \subseteq \dot{A}_{i^*}\} = \alpha^*."$
\end{proof}

\section{Consistency results} \label{consistency_sect}
In this section, we produce a number of models illustrating that incompactness for the chromatic
number of graphs is compatible with a wide array of set-theoretic compactness principles. We first
deal with stationary reflection.

The following Theorem assumes the consistency of the indestructible-reflection principle $\refl^*$.
We remark that, from suitable large cardinal hypotheses, one can force various instances of $\refl^*(S)$.
For example, in \cite{hayut_lh}, it is shown how to arrange $\refl^*(E^{\aleph_2}_{\aleph_0})$,
$\refl^*(\aleph_{\omega+1})$, and $\refl^*(\kappa)$ in a model in which $\kappa$ is the least
inaccessible cardinal. Similar techniques will work at other cardinals. We now show how instances
of $\refl^*(S)$ can be used to obtain instances of incompactness for chromatic numbers together with
stationary reflection.

\begin{thm} \label{refl_thm}
  Suppose $\kappa \geq \aleph_2$ is a regular cardinal, $S \subseteq \kappa$ is stationary,
  and $\refl^*(S)$ holds.
  Then there is a forcing extension preserving all cardinalities and
  cofinalities $\leq \kappa$ in which $S$ remains stationary and $\refl(S)$ and
  $\p^-(\kappa, 2,{\sq}, \kappa, \ns^+_\kappa, 2, 2)$ both hold.
\end{thm}

\begin{proof}
  Let $\mathbb{P}$ be $\mathbb{P}(\kappa, \aleph_0)$ of Definition~\ref{theposet}, i.e. the standard forcing to add a
  $\square(\kappa)$-sequence by initial segments. Let $g$ be $\mathbb{P}$-generic over $V$. In
  $V[g]$, let $\vec{C} := \bigcup g = \langle C_\alpha \mid \alpha < \kappa \rangle$, and let
  $\mathbb{T}$ and $\mathbb{Q}$ be as in Subsection \ref{forcing_subsection}, i.e. $\mathbb{T}$ is the
  forcing to thread $\vec{C}$, and $\mathbb{Q}$ is an iteration to destroy the stationarity of
  subsets of $\kappa$ that are forced to be non-stationary by $\mathbb{T}$.

  Let $h$ be $\mathbb{Q}$-generic over $V[g]$. We claim that $V[g*h]$ is the desired model.
  By Lemma \ref{directed_closed_lemma}, $\mathbb{P} * (\dot{\mathbb{Q}} \times \dot{\mathbb{T}})$
  has a $\kappa$-directed closed subset in $V$. Therefore, in a further extension of $V[g*h]$
  we easily have that all $V$-cardinalities and cofinalities $\leq \kappa$ are preserved and
  $S$ is stationary in $\kappa$. Since these are clearly downward absolute, they hold in $V[g*h]$
  as well. Also, by Theorem~\ref{iteration_lemma}, $\vec{C}$ witnesses $\p^-(\kappa, 2,{\sq}, \kappa, \ns^+_\kappa, 2, 2)$
  in $V[g*h]$. It thus remains to show that $\refl(S)$ holds in $V[g*h]$.

  To this end, fix a stationary $T \subseteq S$ in $V[g*h]$. By construction of $\mathbb{Q}$,
  there is $t \in \mathbb{T}$ such that $t \Vdash_{\mathbb{T}}``T$ is stationary$."$ Let
  $k$ be $\mathbb{T}$-generic over $V[g*h]$ with $t \in k$. Since $\refl^*(S)$ holds in $V$
  and $\mathbb{P} * (\dot{\mathbb{Q}} \times \dot{\mathbb{T}})$ has a dense $\kappa$-directed
  closed subset, $\refl(S)$ holds in $V[g*h*k]$. Moreover, since $t \in k$, $T$ is stationary
  in $V[g*h*k]$. Therefore, $T$ reflects in $V[g*h*k]$. Since this is downward absolute, it
  holds in $V[g*h]$ as well.
\end{proof}

We now turn to $\Delta$-reflection, both at successors of singular cardinals and at inaccessible cardinals.
The following Corollary follows easily from our work thus far and a result of Fontanella and Hayut.

\begin{cor} \label{fontanella_hayut_cor}
  If $\zfc$ is consistent with the existence of infinitely many supercompact cardinals, then
  $\zfc$ is consistent with $\Delta_{\aleph_{\omega^2}, \aleph_{\omega^2+1}}$ together with
  $\p(\aleph_{\omega^2+1},2,\allowbreak{\sq},\aleph_{\omega^2+1},\{\aleph_{\omega^2+1}\},2,2)$.
\end{cor}

\begin{proof}
  In \cite{fontanella_hayut}, starting in a model with infinitely many supercompact cardinals, Fontanella and Hayut
  produce a model in which $\Delta_{\aleph_{\omega^2}, \aleph_{\omega^2+1}}$
  and $\square(\aleph_{\omega^2 + 1})$ both hold. In their model,
  $\ch_{\aleph_{\omega^2}}$ holds and $\aleph_{\omega^2+1}>\beth_\omega$.
  Moreover,  $\Delta_{\aleph_{\omega^2}, \aleph_{\omega^2+1}}$ implies $\refl(\aleph_{\omega^2+1})$, and so,
  by Theorem \ref{improved_square_thm_1}(5),
  $\p(\aleph_{\omega^2+1},2,\allowbreak{\sq},\aleph_{\omega^2+1},\{\aleph_{\omega^2+1}\},2,2)$ holds.
\end{proof}

The following Lemma is standard. We provide a proof for completeness.

\begin{lemma} \label{supercompact_lemma}
  Suppose $\kappa$ is a supercompact cardinal. Then $\Delta_\kappa$ holds.
\end{lemma}

\begin{proof}
  Let $\nu \geq \kappa$ be an arbitrary regular cardinal. We prove $\Delta_{\kappa, \nu}$. To this end, fix
  a stationary $S \subseteq E^\nu_{<\kappa}$ and an algebra $A$ on $\nu$ with fewer than $\kappa$ operations.
  Fix an elementary embedding $j:V \rightarrow M$ witnessing that $\kappa$ is $\nu$-supercompact. Then, in $M$, the following statements hold:
  \begin{itemize}
    \item $j``\nu$ is a subalgebra of $j(A)$;
    \item $\otp(j``\nu) = \nu$, and $\nu < j(\kappa)$ is a regular cardinal;
    \item $j(S) \cap j``\nu = j``S$ is stationary in $\sup(j``\nu)$.
  \end{itemize}
  Therefore, by elementarity, in $V$, there is a subalgebra $A'$ of $A$ such that
  $\eta := \otp(A')$ is a regular cardinal $< \kappa$ and
  $S \cap A'$ is stationary in $\sup(A')$, as required by $\Delta_{\kappa, \nu}$.
\end{proof}

\begin{thm} \label{inaccessible_thm_1}
  Suppose that $\kappa$ is a supercompact cardinal. Then there is a forcing extension in which
  $\kappa$ remains an inaccessible cardinal and $\Delta_\kappa$ and
  $\p^-(\kappa, 2,{\sq}, \kappa, \ns^+_\kappa,\allowbreak2, 2)$ both hold.
\end{thm}

\begin{proof}
  By standard arguments \cite{MR0472529}, we may assume that the supercompactness of $\kappa$ is indestructible under $\kappa$-directed closed
  forcing and $\ch_\kappa$ holds. Let $\mathbb{P}$ be $\mathbb{P}(\kappa, \aleph_0)$ of Definition~\ref{theposet},
  let $g$ be $\mathbb{P}$-generic over $V$, and let $\vec{C} := \bigcup g$. Let
  $\mathbb{T}$ and $\mathbb{Q}$ be as in Subsection \ref{forcing_subsection}.

  Let $h$ be $\mathbb{Q}$-generic over $V[g]$. We claim that $V[g*h]$ is the desired model. Clearly,
  $\kappa$ remains inaccessible in $V[g*h]$. Moreover, by Theorem~\ref{iteration_lemma},
  $\vec{C}$ witnesses $\p^-(\kappa, 2,{\sq_\chi}, \kappa, \ns^+_\kappa, 2, 2)$
  in $V[g*h]$. It thus remains
  to show that $\Delta_\kappa$ holds in $V[g*h]$. To this end, fix a regular cardinal $\nu \geq \kappa$,
  a stationary $S \subseteq E^\nu_{<\kappa}$, and an algebra $A$ on $\nu$
  with fewer than $\kappa$ operations.

  If $\nu = \kappa$, then, by our construction of $\mathbb{Q}$, we can find $t \in \mathbb{T}$
  such that $t \Vdash_{\mathbb{T}}``S$ is stationary in $\nu."$ If $\nu > \kappa$, then, as
  $|\mathbb{T}| = \kappa$, we have $\Vdash_{\mathbb{T}}``S$ is stationary in $\nu."$ In either case,
  we can find a $\mathbb{T}$-generic filter $k$ over $V[g*h]$ such that $S$ remains stationary in
  $\nu$ in $V[g*h*k]$. By Lemma \ref{directed_closed_lemma}, $\mathbb{P} * (\dot{\mathbb{Q}} \times \dot{\mathbb{T}})$ has a dense $\kappa$-directed
  closed subset in $V$, so, as $\kappa$ is indestructibly supercompact in $V$, we have that $\kappa$ is again
  supercompact in $V[g*h*k]$. Therefore, by Lemma \ref{supercompact_lemma}, $\Delta_\kappa$
  holds in $V[g*h*k]$, so, applying it to $S$ and $A$, we find a subalgebra $A'$ of $A$
  such that $\eta:=\otp(A')$ is a regular cardinal $ < \kappa$
  and $S \cap A'$ is stationary in $\sup(A')$. However, by Lemma \ref{directed_closed_lemma},
  $\mathbb{T}$ is $\kappa$-distributive in $V[g*h]$, so we in fact have $A' \in V[g*h]$.
  All of its relevant properties are easily seen to be downward absolute from $V[g*h*k]$ to
  $V[g*h]$, so we have verified that $\Delta_\kappa$ holds in $V[g*h]$.
\end{proof}

\begin{cor} \label{frp_cor}
  Suppose $\zfc$ is consistent with the existence of a Mahlo cardinal. Then $\zfc$ is consistent with
  $\frp(\aleph_2)$ together with $\p(\aleph_2,2,{\sq},\aleph_2,\allowbreak\{\aleph_2\},2,2)$.
\end{cor}

\begin{proof}
  Let $\kappa$ be the least Mahlo cardinal in $L$. Force over $L$ with Miyamoto's forcing from
  \cite{frp_consistency} to obtain a forcing extension $V[g]$ in which $\kappa = \aleph_2$ and $\gch$ and $\frp(\aleph_2$)
  both hold. Since $\kappa$ is not weakly compact in $L$, $\square(\aleph_2)$ holds in $V[g]$. By $\frp(\aleph_2)$, every stationary subset of $E^{\aleph_2}_{\aleph_0}$ reflects,
  and then, by Theorem~\ref{improved_square_thm_1}(6), $\p(\aleph_2,2,{\sq},\aleph_2,\{\aleph_2\},2,2)$ holds.
\end{proof}

Corollary \ref{fontanella_hayut_cor}, Theorem \ref{inaccessible_thm_1}, and Corollary \ref{frp_cor}, together with Proposition \ref{delta_reflection_coloring_prop}
and Fact~\ref{frp_thm}, show that, for many values of $\kappa$,
a maximal degree of incompactness for the chromatic number of graphs of size $\kappa$
is compatible with a maximal degree of compactness for the coloring number of graphs of size $\kappa$.

We now turn our attention to other prominent set-theoretic compactness principles.

\begin{cor} \label{chang_cor}
  Suppose $\kappa_0, \kappa_1, \lambda_0, \lambda_1$ are infinite cardinals, $\kappa_1$ is regular,
  and the Chang's Conjecture variant $(\kappa_1, \lambda_1) \twoheadrightarrow (\kappa_0, \lambda_0)$
  holds. Then there is a forcing extension preserving all cardinalities and cofinalities $\leq \kappa_1$
  in which $(\kappa_1, \lambda_1) \twoheadrightarrow (\kappa_0, \lambda_0)$ remains valid and there is a sequence $\vec{C}$ witnessing
  $\p^{-}(\kappa_1, 2, {\sq}, \kappa_1, (\ns^+_{\kappa_1})^V, 2, \sigma)$ simultaneously for all $\sigma<\kappa_1$.
\end{cor}

\begin{proof}
  Let $\mathbb{P}$ be $\mathbb{P}(\kappa_1, \aleph_0)$ of Definition~\ref{theposet}, let $g$ be $\mathbb{P}$-generic over $V$, and let
  $\vec{C} := \bigcup g$. As $\mathbb{P}$ is $\kappa_1$-strategically
  closed and $(\kappa_1, \lambda_1) \twoheadrightarrow (\kappa_0, \lambda_0)$ is preserved by $\kappa_1$-strategically closed forcing, we immediately have that, in $V[g]$,
  all cardinalities and cofinalities $\leq \kappa_1$ are preserved and $(\kappa_1, \lambda_1) \twoheadrightarrow (\kappa_0, \lambda_0)$
  holds. In addition, in $V[g]$, by Theorem~\ref{genericity_lemma}, $\vec{C}$ witnesses
  $\p^{-}(\kappa_1, 2, {\sq}, \kappa_1, (\ns^+_{\kappa_1})^V,\allowbreak 2, \sigma)$ for all $\sigma<\kappa_1$.
\end{proof}

\begin{cor} \label{supercompact_cor}
  Suppose $\lambda < \kappa$ are regular cardinals, with $\lambda$ indestructibly supercompact. Then there is a forcing
  extension preserving all cardinalities and cofinalities $\leq \kappa$ in which $\lambda$ remains supercompact
  and there is a sequence $\vec{C}$ witnessing $\p^{-}(\kappa, 2, {\sq_\lambda}, \kappa, (\ns^+_\kappa)^V, 2, \sigma)$ simultaneously
  for all $\sigma < \kappa$.
\end{cor}

\begin{proof}
  Let $\mathbb{P}$ be $\mathbb{P}(\kappa, \lambda)$ of Definition~\ref{theposet}, let $g$ by $\mathbb{P}$-generic over $V$, and let
  $\vec{C} := \bigcup g$. As $\mathbb{P}$ is $\lambda$-directed closed, $\lambda$
  remains supercompact in $V[g]$. Moreover, in $V[g]$, by Theorem~\ref{genericity_lemma},
  $\vec{C}$ witnesses $\p^{-}(\kappa, 2, {\sq_\lambda}, \kappa, (\ns^+_\kappa)^V, 2, \sigma)$
  for all $\sigma < \kappa$.
\end{proof}

We remark that, by arguments of Cummings and Magidor from \cite[$\S3$]{MR2811288}, we can in fact perform a class-length iteration that
preserves the supercompactness
of $\lambda$ while forcing the statement that, for all regular $\kappa > \lambda$, there is a
sequence $\vec{C}$ witnessing $\p^{-}(\kappa, 2, {\sq_\lambda}, \kappa, (\ns^+_\kappa)^V, 2, \sigma)$
simultaneously for all $\sigma < \kappa$.

\begin{cor}\label{aleph_2_cor}
  Suppose that $\zfc$ is consistent with the existence of a supercompact cardinal. Then $\zfc$ is consistent
  with each of the following:
  \begin{enumerate}
    \item Martin's Maximum together with the statement that, for all regular $\kappa > \aleph_2$,
      there is a sequence $\vec{C}$ witnessing $\p^{-}(\kappa, 2, {\sq_{\aleph_2}}, \kappa, (\ns^+_\kappa)^V, 2, \sigma)$
      simultaneously for all $\sigma < \kappa$.
    \item Rado's Conjecture together with the statement that, for all regular $\kappa > \aleph_2$,
      there is a sequence $\vec{C}$ witnessing $\p^{-}(\kappa, 2, {\sq_{\aleph_2}}, \kappa, (\ns^+_\kappa)^V, 2, \sigma)$
      simultaneously for all $\sigma < \kappa$.
  \end{enumerate}
\end{cor}

\begin{proof}
  Starting in a model with a supercompact cardinal, one can force Martin's Maximum as in \cite{MR924672}. Martin's Maximum
  is preserved by $\aleph_2$-directed closed set forcing, so, by arguments from the proof of Corollary
  \ref{supercompact_cor} and the following remarks, we can force over the model of Martin's Maximum
  with a class-length iteration that preserves Martin's Maximum and forces that, for all regular $\kappa > \aleph_2$,
  there is a sequence $\vec{C}$ witnessing $\p^{-}(\kappa, 2, {\sq_{\aleph_2}}, \kappa, (\ns^+_\kappa)^V, 2, \sigma)$
  simultaneously for all $\sigma < \kappa$.

  The argument for Rado's Conjectures is similar, exploiting the theorem from \cite{MR686495} stating that, if $\lambda$ is
  supercompact, then Rado's Conjecture holds after forcing with $\coll(\aleph_1, < \lambda)$.
  Moreover, by standard arguments, in the resulting forcing extension, Rado's Conjecture is preserved by $\aleph_2$-directed closed set forcing.
  Now proceed as in the previous paragraph.
\end{proof}

The following results of Todorcevic show that Corollary \ref{aleph_2_cor} is sharp.

\begin{thm}[Todorcevic, {\cite[Theorem~1]{MR763902}}, {\cite[Theorem~10]{MR1261218}}]
  Assume the Proper Forcing Axiom or Rado's Conjecture, let $\kappa \geq \aleph_2$ be regular, and let
  $E^\kappa_{\aleph_1} \subseteq \Gamma \subseteq \acc(\kappa)$. Suppose
  $\langle C_\alpha \mid \alpha \in \Gamma \rangle$ is such that:
  \begin{enumerate}
    \item for all $\alpha \in \Gamma$, $C_\alpha$ is club in $\alpha$;
    \item for all $\beta \in \Gamma$ and all $\alpha \in \acc(C_\beta)$, we have
      $\alpha \in \Gamma$ and $C_\alpha \sq C_\beta$.
  \end{enumerate}
  Then there is a club $D \subseteq \kappa$ such that, for all $\alpha \in \acc(D)$,
  we have $\alpha \in \Gamma$ and $C_\alpha \sq D$.
\end{thm}

\section{Realizing all closed intervals}\label{section5}

Recall the following definition.
\begin{defn}[The chromatic spectrum of a graph, \cite{rinot17}] For a class $\mathcal{P}$ of forcing notions
and a graph $\mathcal G$, let $$\chr_{\mathcal{P}}(\mathcal G) := \{\kappa \mid \exists \mathbb{P}
\in \mathcal{P} (\forces_\mathbb{P} ``\chr(\mathcal G) = \kappa")\}.$$
\end{defn}

In \cite{rinot17},
the second author proves that, if $V = L$ and $\mathcal{P}$ is the class of cofinality-preserving
and $\gch$-preserving forcing posets, then any closed interval of infinite cardinals whose
maximum is below the first cardinal fixed point can be realized as $\chr_{\mathcal{P}}(\mathcal G)$
for some graph $\mathcal G$. The proof uses the $C$-sequence graph $G_\lambda(\vec{C_\lambda})$ as a building block,
where $\vec C_\lambda$ is a $\square_\lambda$-sequence and $G_\lambda$ is some stationary subset of $E^{\lambda^+}_{\cf(\lambda)}$,
chosen in such a way that the $G_\lambda$'s (for different values of $\lambda$) satisfy some sort of \emph{mutual stationarity} condition,
made possible by the fact that, for every infinite cardinal $\theta$ below the first cardinal fixed point,
$[\aleph_0,\theta)$ may be partitioned into finitely many
progressive sets.\footnote{That is, $\card[\aleph_0,\theta)=\mathfrak a_0\uplus\ldots\uplus\mathfrak a_m$, with $\min(\mathfrak a_i)>|\mathfrak a_i|$ for all $i\le m$.}

The forcing notions from \cite{rinot17} witnessing the chromatic spectra are full-support products of posets that build upon Clause~(1) of Lemma~\ref{lemma1}.
Note that Clause~(2) of Lemma~\ref{lemma1} is irrelevant for $\square_\lambda$-sequences, as any forcing to introduce such a threading club $D$ will necessarily collapse the cardinal $\lambda^+$.

\medskip

In this section, we produce a forcing extension satisfying the
same statement about the chromatic spectrum of a graph, but without the restriction that the interval be below the first cardinal fixed
point. More precisely, we
will produce a class forcing extension satisfying \gch\ in which every closed
interval of infinite cardinals is realizable as $\chr_{\mathcal{P}}(G)$ for some
graph $\mathcal G$, where $\mathcal{P}$ is again the class of cofinality-preserving, $\gch$-preserving forcing posets.
Of course, we shall use the $C$-sequence graph as a building block,
but this time, $\vec C$ will be a generic $\square(\kappa)$-sequence, $G$ will simply be $\acc(\kappa)$,
and the witnessing notion of forcing will be an Easton-support product of posets building upon Clause~(2) of Lemma~\ref{lemma1}.

It remains open whether such an unrestricted result follows from $V=L$.

\medskip

Recall the following basic definition.

\begin{defn}
  Suppose $\mathbb{P}$ and $\mathbb{Q}$ are forcing posets. A map $\pi:\mathbb{Q}
  \rightarrow \mathbb{P}$ is a \emph{projection} if:
  \begin{itemize}
    \item $\pi$ is order-preserving, i.e. for all $q_0, q_1 \in \mathbb{Q}$,
      if $q_1 \leq_\mathbb{Q} q_0$, then $\pi(q_1) \leq_\mathbb{P} \pi(q_0)$;
    \item $\pi(\one_\mathbb{Q}) = \one_\mathbb{P}$;
    \item for all $q \in \mathbb{Q}$ and all $p \leq_\mathbb{P} \pi(q)$, there
      is $q' \leq_\mathbb{Q} q$ such that $\pi(q') \leq_\mathbb{P} p$.
  \end{itemize}
  If $\pi:\mathbb{Q} \rightarrow \mathbb{P}$ is a projection and $H$ is $\mathbb{P}$-generic
  over $V$, then let $\mathbb{Q} / H$ denote the poset whose set of conditions is
  $\{q \in \mathbb{Q} \mid \pi(q) \in H\}$ and whose order is inherited from $\mathbb{Q}$.
  Note that, if $\pi:\mathbb{Q} \rightarrow \mathbb{P}$ is a projection and $\dot{H}$
  is the canonical $\mathbb{P}$-name for the generic filter, then $\mathbb{Q}$ is isomorphic
  to a dense subset of $\mathbb{P} * \mathbb{Q}/\dot{H}$ via the map $q \mapsto (\pi(q), \check{q})$.
\end{defn}

If $\kappa$ is a regular, uncountable cardinal, let $\mathbb{S}(\kappa)$ denote the forcing poset
$\mathbb{P}(\kappa, \aleph_0)$ of Definition~\ref{theposet}, i.e. $\mathbb{S}(\kappa)$ is the
standard forcing to add a $\square(\kappa)$-sequence by initial segments.
Let $\mathbb{S}^*(\kappa)$ denote the forcing poset with the same set of conditions as
$\mathbb{S}(\kappa)$ but with an ordering given by $t \leq_{\mathbb{S}^*(\kappa)} s$ iff $t \supseteq s$
and $C^s_{\gamma^s} \sq C^t_{\gamma^t}$.

\begin{prop} \label{projection_prop}
  The identity map $\id:\mathbb{S}^*(\kappa) \rightarrow \mathbb{S}(\kappa)$ is a projection.
\end{prop}

\begin{proof}
Clearly,  $\id$ is order-preserving, and $\id(\one_{\mathbb{S}^*(\kappa)}) = \one_{\mathbb{S}(\kappa)} = \emptyset$.
  Fix $s_0, s_1 \in \mathbb{S}(\kappa)$ with $s_1 \leq_{\mathbb{S}(\kappa)} s_0$. We must produce $s_2$ such that
  $s_2 \leq_{\mathbb{S}^*(\kappa)} s_0$ and $s_2 \leq_{\mathbb{S}(\kappa)} s_1$.

  For $i < 2$, let $\gamma_i = \gamma^{s_i}$, and let $\gamma_2 := \gamma_1 + \omega$. Let
  $C := C^{s_0}_{\gamma_0} \cup \{\gamma_0\} \cup \{\gamma_1 + n \mid n < \omega\}$, and
  define $s_2 \in \mathbb{S}^*(\kappa)$ with $\gamma^{s_2} := \gamma_2$ by letting
  $C^{s_2}_{\gamma_2} := C$ and $C^{s_2}_{\delta} := C^{s_1}_{\delta}$ for all limit $\delta < \gamma_2$.
  It is easily verified that $s_2$ is as desired, so $\id$ is indeed a projection.
\end{proof}

Unlike $\mathbb S(\kappa)$, which is merely $\omega_1$-directed closed, we have:
\begin{prop} \label{closure_prop}
  $\mathbb{S}^*(\kappa)$ is $\kappa$-directed closed.
\end{prop}

\begin{proof}
  Write $\mathbb{S}^* := \mathbb{S}^*(\kappa)$. First note that $\mathbb{S}^*$ is tree-like, so it suffices to show that
  it is $\kappa$-closed. To this end, fix a limit ordinal $\delta < \kappa$, and let
  $\langle s_\eta \mid \eta < \delta \rangle$ be a strictly decreasing sequence of conditions from
  $\mathbb{S}^*$. Let $\gamma := \sup\{\gamma^{s_\eta} \mid \eta < \delta\}$. As $\kappa$ is regular,
  we have $\gamma < \kappa$. We will define a lower bound $s \in \mathbb{S}^*$ with $\gamma^s := \gamma$.
  To specify $s$, it is enough to set $C^s_\gamma := \bigcup_{\eta < \delta} C^{s_\eta}_{\gamma^{s_\eta}}$.
  Note that, by the definition of $\mathbb{S}^*$, we have that, for all $\eta < \xi < \delta$,
  $C^{s_\eta}_{\gamma^{s_\eta}} \sq C^{s_\xi}_{\gamma^{s_\xi}}$. This implies that $C^s_\gamma$ is in fact a
  club in $\gamma$ and, for all $\alpha \in \acc(C^s_\gamma)$, $C^s_\alpha \sq C^s_\gamma$. Therefore,
  $s$ is a condition in $\mathbb{S}^*$ and is a lower bound for $\langle s_\eta \mid \eta < \delta \rangle$.
\end{proof}

\begin{remarks}\begin{enumerate}
\item  If $\kappa^{<\kappa} = \kappa$, then $\mathbb{S}^*(\kappa)$ is a $\kappa$-directed closed forcing poset of
  size $\kappa$ and therefore forcing equivalent to the forcing to add a Cohen subset of $\kappa$.

\item  Suppose that $S$ is $\mathbb{S}(\kappa)$-generic over $V$, and let $\vec{C}:=\bigcup S =
  \langle C_\alpha \mid \alpha < \kappa \rangle$ be the $\square(\kappa)$-sequence
  added by $S$. In $V[S]$, $\mathbb{S}^*(\kappa)/S$ adds a thread through $\vec{C}$: if
  $T$ is $\mathbb{S}^*(\kappa)/S$-generic over $V[S]$, then $D := \bigcup_{t \in T} C_{\gamma^t}$
  is a club in $\kappa$ and, for all $\alpha \in \acc(D)$, $C_\alpha \sq D$.
\end{enumerate}
\end{remarks}

Recall that a set $X$ of ordinals is an \emph{Easton set} if, for every infinite, regular cardinal $\kappa$,
$|X \cap \kappa| < \kappa$. Let $\mathbb{P}$ be the class-length Easton support product forcing where, for all ordinals $i$, the $i^{\mathrm{th}}$
factor is $\mathbb{S}(i)$ if $i$ is a regular, uncountable cardinal and trivial forcing otherwise. Throughout
our discussion, we will disregard coordinates on which trivial forcing is being done. Conditions
of $\mathbb{P}$ are therefore all functions $p$ such that:
\begin{itemize}
  \item $\dom(p)$ is an Easton set of regular, uncountable cardinals;
  \item for all $i \in \dom(p)$, we have $p(i) \in \mathbb{S}(i)$.
\end{itemize}
For $p,q \in \mathbb{P}$, we let $q \leq p$ iff $\dom(q) \supseteq \dom(p)$ and, for all
$i \in \dom(p)$, $q(i) \leq_{\mathbb{S}(i)} p(i)$.
For ordinals $i < j$, let $\mathbb{P}_{i,j}$ denote the poset whose conditions are all $p \in \mathbb{P}$ such that
$\dom(p) \subseteq [i,j)$ and whose order is inherited from $\mathbb{P}$.
For an ordinal $i>0$, let $\mathbb{P}_i$ denote $\mathbb{P}_{0,i}$, and let $\mathbb{P}^i$ denote the class of $p \in \mathbb{P}$
such that $\dom(p) \cap i = \emptyset$. We therefore have, for all $i < j$,
$\mathbb{P} \cong \mathbb{P}_i \times \mathbb{P}_{i,j} \times \mathbb{P}^j$.

Assume \gch\ for the remainder of the section. The next proposition plays the role of Lemma~3.7 of \cite{rinot17}.

\begin{prop} \label{easton_prop}
  Suppose $i$ is a regular, uncountable cardinal and $j > i^+$. Then:
  \begin{enumerate}
    \item $\mathbb{P}_{i^+}$ has the $i^+$-c.c.;
    \item $\mathbb{P}_{i^+,j}$ is $i^+$-strategically closed;
    \item $\forces_{\mathbb P_{i^+}}``\mathbb{P}_{i^+,j}\text{ is }i^+\text{-distributive}."$
  \end{enumerate}
\end{prop}

\begin{proof}
  (1) Note that, as $i^{<i} = i$, the number of Easton subsets of
  $i^+ \cap \reg$ is $i$. For each such Easton subset $d \subseteq i^+ \cap \reg$
  and each $k \in d$, we have $|\mathbb{S}(k)| = k \leq i$, so
  $|\prod_{k \in d} \mathbb{S}(k)| = i$. It follows that $|\mathbb{P}_{i^+}| = i$.
  In particular, $\mathbb{P}_{i^+}$ has the $i^+$-c.c.

  (2) By Lemma \ref{strat_closed_lemma}, we know that,
  for every regular, uncountable $k \in [i^+, j)$, $\mathbb{S}(k)$ is
  $i^+$-strategically closed. Fix a winning strategy $\sigma_k$ for \textrm{II} in
  the game $\Game_{i^+}(\mathbb{S}(k))$.\footnote{Recall Definition~\ref{thegame}.}
  We describe a winning strategy $\sigma$ for \textrm{II} in the game $\Game_{i^+}(\mathbb{P}_{i^+,j})$.
  We will inductively arrange that, if $\langle p_\xi \mid \xi < i^+ \rangle$
  is a run of the game in which \textrm{II} plays according to $\sigma$ and $k \in \bigcup_{\xi < i^+} \dom(p_\xi)$,
  then, letting $\xi_k < i^+$ be least such that $k \in \dom(p_{\xi_k})$, we have that $\xi_k$ is an odd ordinal and
  $\langle \emptyset \rangle ^\frown \langle p_\xi(k) \mid \xi_k \leq \xi < i^+ \rangle$ is
  a run of $\Game_{i^+}(\mathbb{S}(k))$ in which \textrm{II} plays according to $\sigma_k$.

  Suppose that $\eta < i^+$ is an even ordinal and $\langle p_\xi \mid \xi < \eta \rangle$ is a partial
  run of the game in which \textrm{II} has played according to $\sigma$. Let $X := \bigcup_{\xi < \eta} \dom(p_\xi)$.
  Since $\dom(p_\xi)$ is an Easton subset of $[i^+, j)$ for all $\xi < \eta$ and, for all
  regular $k \in [i^+, j)$, we have $\eta < i^+ \leq k$, it follows that $X$ is an Easton subset of $[i^+, j)$.
  For all $k \in X$, let $\xi_k < \eta$ be least such that $k \in \dom(p_{\xi_k})$.
  Define a condition $p$ by letting $\dom(p) := X$ and, for every regular, uncountable $k \in X$,
  letting $p(k) := \sigma_k(\langle \emptyset \rangle ^\frown \langle p_\xi(k) \mid \xi_k \leq \xi < \eta \rangle)$.
  By our inductive assumptions about $\sigma$, this is well-defined. Let $\sigma(\langle p_\xi \mid \xi < \eta \rangle) := p$.
  It is easily verified that this maintains our inductive assumptions and defines a winning strategy for
  \textrm{II} in $\Game_{i^+}(\mathbb{P}_{i,j})$.

  (3) By Clauses (1),(2), and the strategic closure version of Easton's Lemma (cf. \cite[Remark~5.17]{MR2768691}).
\end{proof}

By Clause~(3) of Proposition \ref{easton_prop}, $V^{\mathbb{P}}$ is a model of \zfc.
We next argue that $V^{\mathbb{P}}$ has the same cofinalities (and hence cardinalities)
as $V$. It suffices to show that $\cf(\kappa) > \mu$ in $V^{\mathbb{P}}$ for all $V$-regular cardinals $\mu < \kappa$.
Fix such $\mu$ and $\kappa$. By Proposition~\ref{easton_prop}(1),
$\mathbb{P}_{\mu^+}$ has the $\mu^+$-c.c., so, as $\mu^+ \leq \kappa$, $\kappa$
remains regular in $V^{\mathbb{P}_{\mu^+}}$. By Proposition~\ref{easton_prop}(3),
for all $\lambda > \mu^+$, $\mathbb{P}_{\mu^+, \lambda}$ is $\mu^+$-distributive in $V^{\mathbb{P}_{\mu^+}}$ and
thus cannot add any new functions from $\mu$ to $\kappa$. Therefore, $\cf(\kappa) > \mu$
in $V^{\mathbb{P}_\lambda}$ for all $\lambda$, and hence in $V^{\mathbb{P}}$ as well.

We next argue that \gch\ holds in $V^{\mathbb{P}}$. To do this, we must show that,
for every infinite cardinal $\kappa$, $\kappa^{\cf{\kappa}} = \kappa^+$ in $V^{\mathbb{P}}$.
Fix such a $\kappa$. By the arguments of the previous paragraph,
${^{\cf(\kappa)}}\kappa \cap V^{\mathbb{P}} = {^{\cf(\kappa)}}\kappa \cap V^{\mathbb{P}_{\cf(\kappa)^+}}$.
A nice $\mathbb{P}_{\cf(\kappa)^+}$-name for an element of ${^{\cf(\kappa)}}\kappa$
consists of a function from $\cf(\kappa) \times \kappa$ to the set of antichains of $\mathbb{P}_{\cf(\kappa)^+}$.
Since $\mathbb{P}_{\cf(\kappa)^+}$ has the
$\cf(\kappa)^+$-c.c. and $|\mathbb{P}_{\cf(\kappa)^+}| = \cf(\kappa)$, there are only
$\cf(\kappa)^+$ possible antichains of $\mathbb{P}_{\cf(\kappa)^+}$ and hence only
$(\cf(\kappa)^+)^\kappa = \kappa^+$ nice $\mathbb{P}_{\cf(\kappa)^+}$-names for elements of
${^{\cf(\kappa)}}\kappa$. Therefore, $\kappa^{\cf(\kappa)} = \kappa^+$ in $V^{\mathbb{P}_{\cf(\kappa)^+}}$
and hence in $V^{\mathbb{P}}$.

\medskip

For ordinals $i < j$, let $\mathbb{P}^*_{i,j}$ be the poset with the same conditions
as $\mathbb{P}_{i,j}$ but ordered on regular, uncountable coordinates $k \in [i,j)$
by $\leq_{\mathbb{S}^*(k)}$ rather than by $\leq_{\mathbb{S}(k)}$. The following is immediate
from Lemmas \ref{projection_prop} and \ref{closure_prop}

\begin{lemma}
  Suppose $i < j$.
  \begin{enumerate}
    \item The identity map $\id:\mathbb{P}^*_{i,j} \rightarrow \mathbb{P}_{i,j}$ is a projection;
    \item If $\ell$ is the least regular cardinal in the interval $[i,j)$,
        then $\mathbb{P}^*_{i,j}$ is $\ell$-directed closed. \qed
  \end{enumerate}
\end{lemma}

Let $S$ be $\mathbb{P}$-generic over $V$. For ordinals $i < j$, let $S_i$, $S_{i,j}$, and $S^i$
denote the generic filters for $\mathbb{P}_i$, $\mathbb{P}_{i,j}$, and $\mathbb{P}^i$, respectively,
induced by $S$. In $V[S]$, let $\mathcal{P}$ be the class of all
cofinality-preserving, $\gch$-preserving forcing posets.

\begin{thm}
  In $V[S]$, for every pair of infinite cardinals $\mu \leq \kappa$, there exists a graph $G_{\mu,\kappa}$
  such that $\chr_\mathcal{P}(G_{\mu, \kappa})$ is the set of cardinals $\lambda$ such that
  $\mu \leq \lambda \leq \kappa$.
\end{thm}

\begin{proof}
  If $\mu = \kappa$, then we can simply take $G_{\mu, \kappa} := K_\mu$, where
  $K_\mu$ denotes the complete graph on $\mu$ vertices. Thus, assume that $\mu < \kappa$.

  Work in $V[S]$. Let $k$ be an arbitrary regular, uncountable cardinal.
  Let $\vec{C}^k := \bigcup_{p \in S} p(k) = \langle C^k_\alpha \mid \alpha < k \rangle $
  be the generic $\square(k)$-sequence added by the $k^\mathrm{th}$ coordinate  of $\mathbb{P}$,
  and consider the corresponding graph  $\mathcal G_k := G(\vec{C}^k)$ of Definition~\ref{c_graph_defn}, using $G := \acc(k)$.

  \begin{claim}
    In $V[S_{k^+}]$, for every nonzero $\theta < k$ and every sequence $\vec{A} = \langle A_\eta \mid \eta < \theta \rangle$
    of cofinal subsets of $k$, there is a limit ordinal $\delta < k$ such that $\delta$ captures $\vec{A}$
    with respect to $\vec{C}^k$.\footnote{Recall Definition~\ref{capturing}.}
  \end{claim}

  \begin{proof}
    Fix such a $\theta$ and $\vec{A}$, and suppose the claim fails for $\vec{A}$. For each $\eta < \theta$, let
     $$\Omega_\eta:=\acc(k)\setminus \{\delta\in\acc(k)\mid \min(C^k_\delta)\ge \min(A_0)\ \&\ \exists\iota < \otp(C^k_\delta)[C^k_\delta(\iota), C^k_\delta(\iota+1) \in A_\eta]\}.$$ By our assumption, $\bigcup_{\eta < \theta} \Omega_\eta
    = \acc(k)$.

    In $V$, $\mathbb{P}^*_{k^+}$ is isomorphic to a dense subset of $\mathbb{P}_{k^+} * \mathbb{P}^*_{k^+}/\dot{S}_{\mathbb{P}_{k^+}}$,
    where $\dot{S}_{\mathbb{P}_{k^+}}$ is the canonical name for the $\mathbb{P}_{k^+}$-generic filter.
    By the arguments used to prove the analogous fact about $\mathbb{P}_{k^+}$, forcing with $\mathbb{P}^*_{k^+}$
    over $V$ preserves cofinalities and \gch. Therefore, forcing with $\mathbb{Q} := \mathbb{P}^*_{k^+}/S_{k^+}$
    over $V[S_{k^+}]$ preserves cofinalities and \gch. In particular, in $V[S_{k^+}]^{\mathbb{Q}}$,
    $k$ remains a regular cardinal, and it follows that there is $\eta < \theta$ such that $\Omega_\eta$ is stationary
    in $V[S_{k^+}]^{\mathbb{Q}}$.

    Let $\zeta := \sup\{\min(A_\eta) \mid \eta < \theta\}$. Since $k$ is regular, we have $\zeta < k$. By genericity,
    there is $\beta_0 \in \acc(k)$ such that $\min(C^k_{\beta_0}) > \zeta$. Let $q_0 \in \mathbb{Q}$ be such that
    $\dom(q_0) = \{k\}$ and $q_0(k) = \langle C^k_\alpha \mid \alpha \leq \beta_0 \rangle$.

    Find $q \leq q_0$ and $\eta < \theta$ such that $q \Vdash_{\mathbb{Q}}``\Omega_\eta$ is stationary in $k."$
    Let $\beta < k$ be such that $q(k) = \langle C^k_\alpha \mid \alpha \leq \beta \rangle$.
    Let $\xi_0 := \min(A_\eta \setminus (\beta + 1))$ and $\xi_1 := \min(A_\eta \setminus (\xi_0 + 1))$. By genericity,
    there is $\gamma < k$ such that $(C^k_\beta) \cup \{\beta, \xi_0, \xi_1\} \sq C^k_\gamma$. Define $q^* \leq_{\mathbb{Q}} q$
    by letting $\dom(q^*) := \dom(q)$, $q^*(i) := q(i)$ for all $i \in \dom(q) \setminus \{k\}$, and $q^*(k) :=
    \langle C^k_\alpha \mid \alpha \leq \gamma \rangle$. Let $R$ be $\mathbb{Q}$-generic over $V[S_k]$ with $q^* \in R$.

    Let $D := \bigcup_{r \in R}C^k_{\gamma^{r(k)}}$. Then $D$ is a thread through $\vec{C}^k$ and $\gamma \in \acc(D)$.
    Therefore, if $\iota = \otp(C^k_\gamma \cap \xi_0)$, then, for every $\delta \in \acc(D) \setminus \gamma$,
    we have $\min(C^k_\delta) = \min(C^k_{\beta_0}) > \min(A_\eta)$,
    $C^k_\delta(\iota) = \xi_0$, and $C^k_\delta(\iota+1) = \xi_1$. In particular, $\delta \notin \Omega_\eta$. Hence,
    $\Omega_\eta$ is non-stationary in $V[S_{k^+} * R]$, contradicting the fact that $q^* \in R$ and
    $q^* \leq q \Vdash_{\mathbb{Q}} ``\Omega_\eta$ is stationary$."$
  \end{proof}

  It then follows from Lemma \ref{large_chromatic_number_lemma} that  $V[S_{k^+}]\models\chr(\mathcal G_k) = k$.
  By $V[S]=V[S_{k^+}][S^{k^+}]$  and Proposition~\ref{easton_prop}(3), moreover,  $V[S]\models\chr(\mathcal G_k) = k$.

  Let $\mathcal G_{\mu, \kappa}$ be the disjoint graph union of $K_\mu$ and
  $\mathcal G_k$ for all regular, uncountable $k \in [\mu, \kappa]$. Then, in $V[S]$,
  $\chr(\mathcal G_{\mu,\kappa}) = \kappa$, and, as $K_\mu$ is a subgraph of $G_{\mu, \kappa}$, we know that,
  in any outer model of $V[S]$ with the same cardinals, $\chr(\mathcal G_{\mu, \kappa}) \geq \mu$. We thus
  must show that, for every cardinal $\lambda \in [\mu, \kappa)$, there is a cofinality-preserving,
  \gch-preserving poset $\mathbb{Q}(\lambda)$ such that, $V[S]^{\mathbb{Q}(\lambda)}\models\chr(\mathcal G_{\mu, \kappa}) = \lambda$.

  To this end, fix such a $\lambda$, and let $\mathbb{Q}(\lambda) := \mathbb{P}^*_{\lambda^+, \kappa^+}/S_{\lambda^+, \kappa^+}$.
  In $V$, let $\dot{\mathbb{Q}}(\lambda)$ be the canonical $\mathbb{P_{\lambda^+, \kappa^+}}$-name for $\mathbb{Q}(\lambda)$.
  Then $\mathbb{P}^*_{\lambda^+, \kappa^+}$ is isomorphic to a dense subset of $\mathbb{P}_{\lambda^+, \kappa^+} * \dot{\mathbb{Q}}(\lambda)$,
  so $\mathbb{P} * \dot{\mathbb{Q}}(\lambda) \cong \mathbb{P}_{\lambda^+} \times (\mathbb{P}_{\lambda^+, \kappa^+} * \dot{\mathbb{Q}}(\lambda))
  \times \mathbb{P}^{\kappa^+}$ is forcing equivalent to $\mathbb{P}_{\lambda^+} \times \mathbb{P}^*_{\lambda^+, \kappa^+} \times \mathbb{P}^{\kappa^+}$.
  This is itself a class-length Easton product, and standard arguments just like those for $\mathbb{P}$ show that forcing with
  $\mathbb{P}_{\lambda^+} \times \mathbb{P}^*_{\lambda^+, \kappa^+} \times \mathbb{P}^{\kappa^+}$ over $V$ preserves
  cofinalities and \gch. Therefore, $\mathbb{Q}(\lambda)$ preserves cofinalities
  and \gch\ over $V[S]$.

  We now show that, in $V[S]$, $\forces_{\mathbb{Q}(\lambda)}``\chr(\mathcal G_{\mu, \kappa}) = \lambda."$ First note that,
  for all regular, uncountable $k \in [\lambda^+, \kappa]$, forcing with $\mathbb{Q}(\lambda)$ adds a thread
  through $\vec{C}^k$, so that, by Lemma~\ref{lemma1}(2), $\forces_{\mathbb{Q}(\lambda)}``\chr(\mathcal G_k) \leq \aleph_0."$
  Consequently, $\Vdash_{\mathbb{Q}(\lambda)}``\chr(\mathcal G_{\mu, \kappa}) \leq \lambda."$

  To show the reverse inequality,   we consider three cases:

  $\br$ Suppose $\lambda = \mu$. As $K_\mu$ is a subgraph of $\mathcal G_{\mu, \kappa}$, we immediately obtain
  $\Vdash_{\mathbb{Q}(\lambda)}``\chr(\mathcal G_{\mu, \kappa}) \geq \lambda."$

  $\br$ Suppose $\lambda > \mu$ and $\lambda$ is a regular cardinal. It suffices to show that
  $\Vdash_{\mathbb{Q}(\lambda)}``\chr(\mathcal G_\lambda) = \lambda."$ To see this, it is enough
  to verify that $\mathbb{Q}(\lambda)$ does not add any new functions from $\lambda$ to $\lambda$.
  By the strategic closure version of Easton's Lemma,
  we have that, for all $j > \kappa^+$, in $V[S_{\lambda^+}]$, $\mathbb{P}^*_{\lambda^+, \kappa^+} \times \mathbb{P}_{\kappa^+, j}$
  is $\lambda^+$-distributive and hence does not add any new functions from $\lambda$ to $\lambda$.
  Since $\mathbb{P} * \dot{\mathbb{Q}}(\lambda)$ is forcing equivalent to $\mathbb{P}_{\lambda^+} \times \mathbb{P}^*_{\lambda^+, \kappa^+} \times \mathbb{P}^{\kappa^+}$,
  this implies that ${^\lambda}\lambda \cap V^{\mathbb{P} * \dot{\mathbb{Q}}(\lambda)} \subseteq V^{\mathbb{P}_{\lambda^+}}$.
  In particular, forcing with $\mathbb{Q}(\lambda)$ over $V[S]$ does not add any new functions from
  $\lambda$ to $\lambda$.

  $\br$ Suppose $\lambda > \mu$ and $\lambda$ is singular. As in the previous case, it suffices to show
  that, for all regular, uncountable $k \in [\mu, \lambda)$, $\mathbb{Q}(\lambda)$ does not add any new
  functions from $k$ to $k$ and, therefore, $\forces_{\mathbb{Q}(\lambda)}``\chr(\mathcal G_k) = k."$
  Fix such a $k$. In $V[S_{k^+}]$, again by the strategic closure version of Easton's Lemma,
  we have that, for all $j > \kappa^+$, $\mathbb{P}_{k^+, \lambda^+} \times \mathbb{P}^*_{\lambda^+, \kappa^+} \times \mathbb{P}_{\kappa^+, j}$
  is $k^+$-distributive and hence does not add any new functions from $k$ to $k$. Therefore,
  ${^k}k \cap V^{\mathbb{P} * \dot{\mathbb{Q}}(\lambda)} \subseteq V^{\mathbb{P}_{k^+}}
  \subseteq V^{\mathbb{P}_{\lambda^+}}$, thus completing the proof.
\end{proof}

\section*{Acknowledgments}
The authors were partially supported by the Israel Science Foundation (grant $\#$1630/14).

We thank S. Fuchino for informing us about \cite{rc_implies_frp}.
The results of this paper were presented by the first author at the Hebrew University Logic Seminar in January 2017 and
at the University of Helsinki Logic Seminar in April 2017,
and by the second author at the \emph{Set Theory} workshop in Oberwolfach, February 2017. We thank the organizers for the warm hospitality.

Finally, we thank the anonymous referees for their thoughtful feedback.

\end{document}